\documentclass[10pt]{amsart}

\usepackage[mathscr]{euscript}
\DeclareFontFamily{OT1}{rsfs}{}
\DeclareFontShape{OT1}{rsfs}{n}{it}{<-> rsfs10}{}
\DeclareMathAlphabet{\curly}{OT1}{rsfs}{n}{it}

\makeatletter
\newcommand{\eqnum}{\refstepcounter{equation}\textup{\tagform@{\theequation}}}
\makeatother

\newcommand\beq[1]{\begin{equation}\label{#1}}
	\newcommand\eeq{\end{equation}}
\newcommand\beqa{\begin{eqnarray*}}
	\newcommand\eeqa{\end{eqnarray*}}

\title[Categorical wall-crossing formula]{Categorical wall-crossing formula 
for Donaldson-Thomas theory on the resolved conifold}
\date{}
\author{Yukinobu Toda}

\usepackage{amscd}
\usepackage{amsmath}
\usepackage{amssymb}
\usepackage{amsthm}
\usepackage{float}
\usepackage{graphicx}
\usepackage{tikz}
\usepackage{xypic}
\usetikzlibrary{cd}
\usepackage{here}
\usetikzlibrary{intersections, calc, arrows.meta}

\usepackage{youngtab}

\usepackage{array}
\usepackage{amscd}
\usepackage[all]{xy}

\usepackage{tabulary}
\usepackage{booktabs}

\usepackage{amssymb, paralist, xspace, url, amscd, euscript, mathrsfs, stmaryrd}

\DeclareFontFamily{U}{rsfs}{%
	\skewchar\font127}
\DeclareFontShape{U}{rsfs}{m}{n}{%
	<-6>rsfs5<6-8.5>rsfs7<8.5->rsfs10}{}
\DeclareSymbolFont{rsfs}{U}{rsfs}{m}{n}
\DeclareSymbolFontAlphabet
{\mathrsfs}{rsfs}
\DeclareRobustCommand*\rsfs{%
	\@fontswitch\relax\mathrsfs}

\theoremstyle{plain}
\newtheorem{thm}{Theorem}[section]
\newtheorem{prop}[thm]{Proposition}
\newtheorem{lem}[thm]{Lemma}

\newtheorem{defi}[thm]{Definition}
\newtheorem{rmk}[thm]{Remark}
\newtheorem{cor}[thm]{Corollary}

\newtheorem{prop-defi}[thm]{Proposition-Definition}
\newtheorem{thm-defi}[thm]{Theorem-Definition}
\newtheorem{lem-defi}[thm]{Lemma-Definition}

\newcommand{\ssslash}{/\!\!/}

\newcommand{\cC}{\mathcal{C}}
\newcommand{\dD}{\mathcal{D}}
\newcommand{\eE}{\mathcal{E}}

\newcommand{\gG}{\mathcal{G}}
\newcommand{\hH}{\mathcal{H}}

\newcommand{\mM}{\mathcal{M}}

\newcommand{\oO}{\mathcal{O}}
\newcommand{\pP}{\mathcal{P}}

\newcommand{\sS}{\mathcal{S}}

\newcommand{\uU}{\mathcal{U}}

\newcommand{\wW}{\mathcal{W}}

\newcommand{\yY}{\mathcal{Y}}

\newcommand{\fM}{\mathfrak{M}}

\newcommand{\Hom}{\mathop{\rm Hom}\nolimits}

\newcommand{\dR}{\mathbf{R}}

\newcommand{\Pic}{\mathop{\rm Pic}\nolimits}

\newcommand{\id}{\textrm{id}}

\newcommand{\Ext}{\mathop{\rm Ext}\nolimits}
\newcommand{\Spec}{\mathop{\rm Spec}\nolimits}

\newcommand{\Coh}{\mathop{\rm Coh}\nolimits}

\newcommand{\us}{\mathchar`-\rm{us}}
\newcommand{\sss}{\mathchar`-\rm{ss}}

\newcommand{\cneq}{\mathrel{\raise.095ex\hbox{:}\mkern-4.2mu=}}
\newcommand{\eqcn}{\mathrel{=\mkern-4.5mu\raise.095ex\hbox{:}}}

\newcommand{\ext}{\mathop{\rm ext}\nolimits}

\newcommand{\Aut}{\mathop{\rm Aut}\nolimits}

\newcommand{\PPer}{\mathop{\rm Per}\nolimits}

\newcommand{\IC}{\mathop{\rm IC}\nolimits}

\newcommand{\DT}{\mathop{\rm DT}\nolimits}

\newcommand{\modu}{\mathop{\rm mod}\nolimits}

\newcommand{\End}{\mathop{\rm End}\nolimits}

\newcommand{\RHom}{\mathop{\dR\mathrm{Hom}}\nolimits}
\newcommand{\Ker}{\mathop{\rm Ker}\nolimits}

\newcommand{\GL}{\mathop{\rm GL}\nolimits}

\newcommand{\cl}{\mathop{\rm cl}\nolimits}

\newcommand{\MF}{\mathop{\rm MF}\nolimits}

\newcommand{\Crit}{\mathop{\rm Crit}\nolimits}

\newcommand{\wt}{\mathrm{wt}}

\newcommand{\qcoh}{\mathrm{qcoh}}

\newcommand{\inclusion}{\ar@<-0.3ex>@{^{(}->}[r]}
\newcommand{\upinclusion}{\ar@<-0.3ex>@{^{(}->}[u]}
\newcommand{\leinclusion}{\ar@<-0.3ex>@{^{(}->}[l]}
\newcommand{\doinclusion}{\ar@<-0.3ex>@{^{(}->}[d]}
\newcommand{\diasquare}{\ar@{}[rd]|\square}

\newcommand{\lgakko}{(\!(}
\newcommand{\rgakko}{)\!)}

\newcommand{\C}{\mathbb{C}^{\ast}}

\makeatletter
\renewcommand{\theequation}{%
	\thesection.\arabic{equation}}
\@addtoreset{equation}{section}
\makeatother

\begin{document}
	
	\begin{abstract}
		We prove wall-crossing formula 
		for categorical Donaldson-Thomas invariants on the 
		resolved conifold, which categorifies Nagao-Nakajima 
		wall-crossing formula for numerical DT invariants on it. 
		The categorified Hall products are used to describe the wall-crossing 
		formula as semiorthogonal decompositions. 
		A successive application of 
		categorical wall-crossing formula 
		yields semiorthogonal decompositions of 
		categorical Pandharipande-Thomas stable pair invariants 
		on the resolved conifold, which categorify the product expansion 
		formula of the generating series of numerical PT invariants on it. 
	\end{abstract}
	
	\maketitle
	

	\section{Introduction}
	\subsection{Background and summary of the paper}
	In this paper, we establish
	wall-crossing formula
	for categorical Donaldson-Thomas invariants 
	on the resolved conifold, 
	and apply it to 
	give a complete description of 
	categorical Pandharipande-Thomas (PT) stable pair invariants
	on it. 
	
	The PT invariants count stable pairs
	on CY 3-folds, which were 
	introduced in~\cite{PT} in order to give 
	a better formulation of GW/DT correspondence conjecture~\cite{MNOP}. 
	They are special cases of Donaldson-Thomas (DT) type invariants 
	counting stable objects in the derived category, 
	and are 
	now understood as fundamental enumerative invariants of curves 
	on CY 3-folds as well as Gromov-Witten invariants and 
	Gopakumar-Vafa invariants. 
Now by efforts from derived algebraic geometry~\cite{PTVV, BBJ}, 
the moduli spaces which define DT (in particular PT) 
invariants are known to be locally 
written as critical loci. 
	In~\cite{TocatDT}, we proposed a study of categorical 
	DT theory by gluing locally defined dg-categories of matrix factorizations
	on these moduli spaces. 
	A definition of categorical DT invariants is introduced in the case of local 
	surfaces in~\cite{TocatDT} via Koszul duality and singular 
	support quotients.
	We also proposed several conjectures on wall-crossing 
	of categorical DT invariants on local surfaces, 
	motivated by 
	d-critical analogue of
	Bondal-Orlov and Kawamata's D/K equivalence conjecture~\cite{B-O2, MR1949787}, and also categorifications of wall-crossing formulas 
	of numerical DT invariants~\cite{JS, K-S}.
	In~\cite{TocatDT}, we also derived wall-crossing formula of categorical PT invariants
	on local surfaces in the setting of simple wall-crossing 
	(i.e. there are at most two Jordan-H\"{o}lder factors at the wall). 
	
	The purpose of this paper is to prove wall-crossing formula for categorical DT invariants 
	 on the resolved conifold, which categorifies Nagao-Nakajima 
	 wall-crossing formula~\cite{NN} for numerical DT invariants on it. 
	In this case the relevant moduli spaces are global critical loci, so there is no issue 
	on gluing dg-categories of 
	matrix factorizations. However
	 wall-crossing is not necessary a simple wall-crossing, and 
	the analysis of categorical wall-crossing is much harder. 
		Our strategy is to 
	use categorified Hall products for quivers with super-potentials 
	introduced by P{\u{a}}durariu~\cite{Tudor, Tudor1.5}. 
	A key observation is that, up to Kn\"{o}rrer periodicity, a wall-crossing diagram 
	for the resolved conifold 
	locally looks like a Grassmannian flip together with some super-potential 
	(d-critical Grassmannian flip
	in the sense of d-critical birational geometry~\cite{Toddbir}). 
We 
refine the result of~\cite{BNFV} on derived 
	categories of Grassmannian flips via categorified Hall products, 
	and compare them with more global categorified Hall products 
	under the Kn\"{o}rrer periodicity. 
	The above approach via categorified Hall products 
	yields a desired categorical wall-crossing formula. 
	A successive iteration of wall-crossing gives a 
	semiorthogonal decomposition  
	of categorical PT invariants on the resolved conifold
	whose semiorthogonal summands are the simplest
	categories of matrix factorizations over a point. 
	 We emphasize that the result of this paper is a first instance where 
	categorical wall-crossing formula is obtained for non-simple wall-crossing
	in the context of categorical DT theory.

	\subsection{Categorical PT stable pair theory on the resolved conifold}
	The \textit{resolved conifold} $X$ is defined by 
	\begin{align*}
		X \cneq \mathrm{Tot}_{\mathbb{P}^1}(\oO_{\mathbb{P}^1}(-1)^{\oplus 2}),
		\end{align*}
	which is also obtained as a crepant small resolution 
	of the conifold singularity 
	$\{xy+zw=0\} \subset \mathbb{C}^4$. 
	The resolved conifold is a non-compact CY 3-fold, and 
	an important toy model for enumerative geometry on 
	CY 3-folds
	such as PT invariants. 
	
	For each $(\beta, n) \in \mathbb{Z}^2$, 
	we denote by 
	\begin{align*}
		P_n(X, \beta)
		\end{align*}
	 the moduli space of 
	PT stable pairs $(F, s)$ on $X$, i.e. 
	$F$ is a pure one dimensional coherent sheaf on $X$
	and $s \colon \oO_X \to F$ is surjective in dimension one, 
	satisfying $[F]=\beta[C]$ and $\chi(F)=n$. Here $C \subset X$ is the 
	zero section of the projection $X \to \mathbb{P}^1$, and $[F]$ is the fundamental 
	one cycle of $F$. 
	The PT invariant $P_{n, \beta} \in \mathbb{Z}$ is 
	defined by either taking the integration over the zero dimensional 
	virtual 
	fundamental class on $P_n(X, \beta)$, or weighted Euler characteristic of the Behrend 
	constructible function~\cite{MR2600874} on it. 
	It is well-known that the generating series of PT 
	invariants on $X$ is given by the formula:
	\begin{align}\label{intro:formula}
		\sum_{n, \beta}P_{n, \beta}q^n t^{\beta} =
		\prod_{m\ge 1}(1-(-q)^m t)^m. 
		\end{align}
	The above formula is available in~\cite[Theorem~3.15]{NN}, 
	which is also obtained 
	from the DT calculation in~\cite{BeBryan} together with 
	the DT/PT correspondence~\cite{BrH, Tcurve1, StTh}. 
	
	The purpose of this paper is to give a 
	categorification of the formula (\ref{intro:formula}).
In the case of the resolved conifold, 
	the moduli space $P_n(X, \beta)$ is written as a 
	global critical locus, i.e. 
	there is a pair $(M, w)$ where $M$ is a smooth quasi-projective 
	scheme and $w \colon M \to \mathbb{A}^1$ is a 
	regular function such that $P_n(X, \beta)$ is isomorphic 
	to the critical locus of $w$. 
		A choice of $(M, w)$ is not unique, and we 
	take it 
	 using Van den Bergh's non-commutative crepant 
	 resolution of $X$~\cite{MR2057015} (see Subsection~\ref{subsec:catPT}). 
	 We define the \textit{categorical PT invariant} on $X$
	 to be 
	 the triangulated category of matrix  
	 factorizations of $w$ (see Definition~\ref{defi:catPT})
	 \begin{align*}
	 	\mathcal{DT}(P_n(X, \beta)) \cneq \mathrm{MF}(M, w). 
	 	\end{align*}
 	The above triangulated category (or more precisely its 
 	dg-enhancement) recovers $P_{n, \beta}$ by taking the 
 	Euler characteristic of its periodic cyclic homology (see Equation (\ref{eqn:PT})). 
 	The following is a consequence of 
 	the main result in this paper:
 	\begin{thm}\emph{(Corollary~\ref{cor:sod2.5})}\label{intro:thm}
 		There exists a semiorthogonal decomposition of the form 
 		\begin{align}\label{intro:PT:sod}
 				\mathcal{DT}(P_n(X, \beta))=\langle  a_{n, \beta}
 				\mbox{-copies of }\MF(\Spec \mathbb{C}, 0)
 				\rangle. 
 			\end{align}
 		Here $a_{n, \beta}$ is defined by 
 		\begin{align}\label{def:an}
 			a_{n, \beta}
 			\cneq \sum_{\begin{subarray}{c}
 					l \colon \mathbb{Z}_{\ge 1} \to \mathbb{Z}_{\ge 0} \\
 					\sum_{m\ge 1} l(m) \cdot (m, 1)=(n, \beta)
 				\end{subarray}	
 			} \prod_{m\ge 1}\binom{m}{l(m)}. 
 		\end{align}
 		
 		\end{thm}
 	Here $\MF(\Spec \mathbb{C}, 0)$ is the 
 category of matrix factorizations of the zero super-potential 
 over the point, which is equivalent to the $\mathbb{Z}/2$-periodic 
 derived category of finite dimensional $\mathbb{C}$-vector spaces. 
 As the formula (\ref{intro:formula}) is equivalent to $P_{n, \beta}=(-1)^{n+\beta}a_{n, \beta}$, 
 by taking the periodic cyclic homologies of 
 	both sides and Euler characteristics, the 
 	result of Theorem~\ref{intro:thm} recovers the formula (\ref{intro:formula})
 	(see Remark~\ref{rmk:recover}).

	\subsection{Categorical wall-crossing formula}
	In~\cite[Theorem~3.15]{NN}, Nagao-Nakajima
	derived the formula (\ref{intro:formula})
	by proving wall-crossing formula for stable perverse 
	coherent systems on $X$. 
	Under a derived equivalence of $X$ with a non-commutative 
	crepant resolution of the conifold~\cite{MR2057015}, 
	the category of perverse coherent systems on $X$
	is equivalent to the 
	category of representations of the following 
	quiver with super-potential $(Q^{\dag}, W)$
	\[
	Q^{\dag}=
\begin{tikzcd}
		\bullet_{\infty} \arrow[d] & & \\
	\bullet_{0}
		\arrow[rr, bend left,  "a_2"]
		\arrow[rr,bend left=70,  "a_1"]
 & &
		\bullet_1
		\arrow[ll, bend left, "b_1" ]
			\arrow[ll, bend left=70, "b_2"]
		\end{tikzcd}
	\quad 
	W=a_1 b_1 a_2 b_2-a_1 b_2 a_2 b_1.	
\]

	For $v=(v_0, v_1) \in \mathbb{Z}_{\ge 0}^2$, we denote by 
	$\mM_Q^{\dag}(v)$ the $\C$-rigidified moduli stack 
	of $Q^{\dag}$-representations with dimension vector $(1, v_0, v_1)$, 
	where $1$ is the dimension vector at $\infty$. 
	It is equipped with a super-potential 
	\begin{align*}
		w=\mathrm{Tr}(W) \colon \mM_Q^{\dag}(v) \to \mathbb{A}^1
		\end{align*}
		whose critical locus is isomorphic to 
		the moduli stack of $(Q^{\dag}, W)$-representations with dimension vector $(1, v_0, v_1)$. 
	There is also a stability parameter $\theta=(\theta_0, \theta_1)
	\in \mathbb{R}^2$ of $(Q^{\dag}, W)$-representations, 
	whose wall-chamber structure is given by 
	the following picture
	(see~\cite[Figure~1]{NN}): 
	\begin{figure}[H]
		\centering
		\begin{tikzpicture}[scale=0.6][node distance=1cm]
			\draw[thick] (-4.6,0)--(4.6,0)  node [pos=0, anchor=east]{\tiny{$\theta_1=0$}} ;\draw[thick](0,4.6)--(0,-4.6)  node [pos=0, anchor=east]{\tiny{$\theta_0=0$}} ;
			\draw[thick] (-4,4)--(4,-4)  node [pos=0, anchor=east]{\tiny{$\theta_0+\theta_1=0$}} ;
			\draw[thick] (-3.5,4.3)--(3.5,-4.3)  node[pos=0, anchor=east]{\tiny{$m\theta_0+(m-1)\theta_1=0$} } ;
			\draw[thick] (-2.3,4.6)--(2.3,-4.6)  node[pos=0, anchor=east]{\tiny{$2\theta_0+\theta_1=0$} } ;
			\draw[thick] (-4.3,3.5)--(4.3,-3.5)  node[pos=0, anchor=east]{\tiny{$m\theta_0+(m+1)\theta_1=0$} } ;
			\draw[thick] (-4.6,2.3)--(4.6,-2.3)  node[pos=0, anchor=east]{\tiny{$\theta_0+2\theta_1=0$} } ;
			\draw[fill] (-3.5,3.6) circle [radius=0.025];
			\node [thick, right] at (-3.5,3.6) {\tiny{PT}};
			\draw[fill] (-3.6,3.5) circle [radius=0.025];
			\node [thick, below] at (-3.6,3.5) {\tiny{DT}};
			\draw[fill] (3.8,-3.7) circle [radius=0.025];
			\node [thick, above] at (3.8,-3.7) {\tiny{PT}};
			
			\draw[fill] (3.7,-3.8) circle [radius=0.025];
			\node [thick, left] at (3.7,-3.8) {\tiny{DT}};
			
			\node at (4.2,-3.7) {\small{$X^+$}};
			\node at (-4,3.7) {\small{$X$}};
			
			\draw[fill] (-3.46,2.5) circle [radius=0.015];
			\draw[fill] (-3.51,2.4) circle [radius=0.015];
			\draw[fill] (-3.56,2.3) circle [radius=0.015];
			
			\draw[fill] (-2.5,3.46) circle [radius=0.015];
			\draw[fill] (-2.4,3.51) circle [radius=0.015];
			\draw[fill] (-2.3,3.56) circle [radius=0.015];
			
			\draw[fill] (-3.15, 2.97) circle [radius=0.015];
			\draw[fill] (-3.2, 2.9) circle [radius=0.015];
			\draw[fill] (-3.25, 2.83) circle [radius=0.015];
			
			\draw[fill] (-2.97, 3.15) circle [radius=0.015];
			\draw[fill] (-2.9, 3.22) circle [radius=0.015];
			\draw[fill] (-2.83, 3.28 ) circle [radius=0.015];
			
			\node at (2.5,2) {$\begin{subarray}{c}\mathrm{empty\,\, chamber} \end{subarray}$ } ;
			\node at (-3,-2.3) {$\begin{subarray}{c}\mathrm{non-commutative}  \\ \mathrm{chamber} \end{subarray}$ } ;
		\end{tikzpicture} 
		\caption{Wall-chamber structures}
		\label{intro:figure}
	\end{figure}
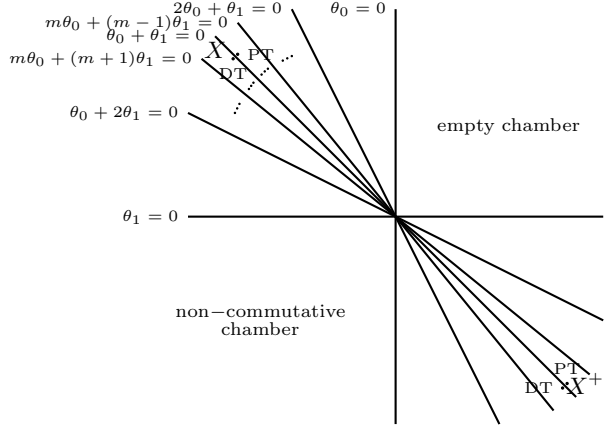
For $m \in \mathbb{Z}_{\ge 1}$, 
there is a wall in the second quadrant in the above picture 
\begin{align*}
	W_m \cneq \mathbb{R}_{>0} (1-m, m) \subset \mathbb{R}^2. 
	\end{align*}
We take a stability condition 
on the wall $\theta \in W_m$ and $\theta_{\pm}=\theta\pm (-\varepsilon, \varepsilon)$
for $\varepsilon>0$ which lie on its adjacent chambers.  
Let $\DT^{\theta_{\pm}}(v_0, v_1) \in \mathbb{Z}$ be the 
DT invariant counting
$\theta_{\pm}$-stable $(Q^{\dag}, W)$-representations 
with dimension vector $(1, v_0, v_1)$. 
We have the 
following wall-crossing formula
proved in~\cite[Theorem~3.12]{NN}
\begin{align}\label{intro:wcf}
	\sum_{(v_0, v_1) \in \mathbb{Z}_{\ge 0}^2} \DT^{\theta_{+}}(v_0, v_1)q_0^{v_0} q_1^{v_1}
	=\left(\sum_{(v_0, v_1) \in \mathbb{Z}_{\ge 0}^2}  \DT^{\theta_{-}}(v_0, v_1)q_0^{v_0} q_1^{v_1}\right)
	\cdot (1+q_0^m (-q_1)^{m-1})^{m}.
	\end{align}	
	The formula (\ref{intro:formula}) is obtained from the above 
	wall-crossing formula by applying it from $m=1$ to $m \gg 0$, 
	and noting that 
	the PT invariants correspond to a chamber which 
	is sufficiently close to (and above) the wall $\mathbb{R}_{>0}(-1, 1)$. 
	
	We prove Theorem~\ref{intro:thm} by giving a categorification of
	the formula (\ref{intro:wcf}). 
	For $\theta \in \mathbb{R}^2$, 
	we denote by 
	\begin{align*}
		\mM^{\dag, \theta \sss}_{Q}(v) \subset \mM_Q^{\dag}(v)
		\end{align*}
	the open substack of $\theta$-semistable $Q^{\dag}$-representations. 
	The following is the main result of this paper, 
	which gives a
	categorification of the formula (\ref{intro:wcf}): 
	\begin{thm}\emph{(Corollary~\ref{cor:sod})}\label{intro:thm2}
		For $\theta \in W_m$, 
		by setting $s_m=(m, m-1)$, 
		there exists a semiorthogonal decomposition 
		\begin{align}\label{intro:catWCF}\MF(\mM_Q^{\dag, \theta_+ \sss}(v), w)
		=\left\langle \binom{m}{l} \mbox{-copies of }
		\MF(\mM_Q^{\dag, \theta_- \sss}(v-ls_m), w) : 
		l\ge 0  \right\rangle.
		\end{align} 
		\end{thm}
	 There is also a precisely defined order among semiorthogonal summands 
	in (\ref{intro:PT:sod}), 
	see~Corollary~\ref{cor:sod} for the precise statement. 
	Again by taking the periodic cyclic homologies and 
	the Euler characteristics, the result of Theorem~\ref{intro:thm2}
	recovers the Nagao-Nakajima formula (\ref{intro:wcf}) (see Remark~\ref{rmk:recover0}). 
	The result of Theorem~\ref{intro:thm} follows by applying 
	Theorem~\ref{intro:thm2} from $m=1$ to $m \gg 0$.
	 
	We also remark that the similar 
	categorical wall-crossing formula holds 
	at other walls except walls at $\{\theta_0+\theta_1=0\}$, i.e. 
	DT/PT wall on $X$ or on its flop
	(see Remark~\ref{rmk:owall}). 
	Note that the numerical DT/PT wall-crossing formula 
	was not directly obtained in~\cite{NN}, but 
	was proved in~\cite{BrH, Tcurve1, StTh}
	using the full machinery of motivic Hall algebras in~\cite{JS, K-S}. 
	
\subsection{Outline of the proof of Theorem~\ref{intro:thm2}}
The strategy of the proof of Theorem~\ref{intro:thm2} is to 
use the following ingredients: 
\begin{enumerate}
	\item The window subcategories 
	for GIT quotient stacks developed by Halpern-Leistner~\cite{MR3327537}
	and Ballard-Favero-Katzarkov~\cite{MR3895631}. 
	\item The categorified Hall products
	for quivers with super-potentials introduced and studied by P{\u{a}}durariu~\cite{Tudor1.5, Tudor1.7, Tudor}.
	\item The descriptions of 
	derived categories under Grassmannian flips
	by Ballard-Chidambaram-Favero-McFaddin-Vandermolen~\cite{BNFV}, 
	which it self relies on an earlier work by Donovan-Segal~\cite{DoSe}
	for Grassmannian flops. 
	\end{enumerate}
For $\theta \in W_m$, 
let $\mM_{Q}^{\dag, \theta \sss}(v) \to M_{Q}^{\dag, \theta \sss}(v)$
be the good moduli space~\cite{MR3237451}. We have the wall-crossing diagram 
\begin{align}\label{intro:dia:wall}
	\xymatrix{
		M_{Q}^{\dag, \theta_+\sss}(v) \ar[rd] \ar@{.>}[rr] & 
		& M_Q^{\dag, \theta_-\sss}(v) \ar[ld] \\
		& M_Q^{\dag, \theta \sss}(v)	
}
	\end{align}
which is shown to be a flip of smooth quasi-projective varieties. 
The D/K principle by Bondal-Orlov~\cite{B-O2} and Kawamata~\cite{MR1949787}
predicts the existence of a fully-faithful functor of 
their derived categories or categories of matrix factorizations. 

The window subcategories have been used 
to investigate the D/K conjecture
under variations of GIT quotients. 
In the above setting, there  
 exist subcategories (called \textit{window subcategories})
$\mathbb{W}_{\rm{glob}}^{\theta_{\pm}}(v) 
	\subset \mathrm{MF}(\mM_Q^{\dag, \theta\sss}(v), w)
$
such that the compositions
\begin{align*}
	\mathbb{W}_{\rm{glob}}^{\theta_{\pm}}(v)
	\hookrightarrow  \mathrm{MF}(\mM_Q^{\dag, \theta \sss}(v), w)
	\twoheadrightarrow \mathrm{MF}(\mM_Q^{\dag, \theta_{\pm} \sss}(v), w)
	\end{align*}
are equivalences. 
If we can show that 
$\mathbb{W}_{\rm{glob}}^{\theta_-}(v) \subset \mathbb{W}_{\rm{glob}}^{\theta_+}(v)$ for some 
choice of window subcategories, then 
we have a desired fully-faithful functor
\begin{align}\label{intro:inclu}
	\mathrm{MF}(\mM_Q^{\dag, \theta_{-} \sss}(v), w)
	\hookrightarrow \mathrm{MF}(\mM_Q^{\dag, \theta_{+} \sss}(v), w). 
	\end{align} 
In fact, the above argument is used in~\cite[Theorem~4.3.5]{TocatDT} to show 
the existence of a fully-faithful functor (\ref{intro:inclu}). 

We are interested in the semiorthogonal complement of the fully-faithful 
functor (\ref{intro:inclu}). 
If the wall-crossing is enough simple, e.g. satisfying the DHT condition 
in~\cite[Definition~4.1.4]{MR3895631}, 
then the above window subcategory argument also 
describes the semiorthogonal complement (see~\cite[Theorem~4.2.1]{MR3895631}).
However our wall-crossing (\ref{intro:dia:wall}) does not necessary 
satisfy the DHT condition, and 
we cannot directly apply it. 
Instead we use categorified Hall products 
to describe the semiorthogonal complement of (\ref{intro:inclu}). 

The categorified Hall product for quivers with super-potentials
is introduced in~\cite{Tudor1.5, Tudor1.7, Tudor} in order to 
give a K-theoretic version of
 critical COHA, which was introduced in~\cite{MR2851153} and developed in~\cite{MR3667216}. 
For $v=v_1+v_2$ with $\theta(v_i)=0$, it is a functor 
\begin{align*}
	\ast \colon 
	\MF(\mM_Q^{\theta \sss}(v_1), w) \boxtimes 
	\MF(\mM_Q^{\dag, \theta \sss}(v_2), w) \to 
	\MF(\mM_Q^{\dag, \theta \sss}(v), w)
	\end{align*}
defined by the pull-back/push-forward with respect 
to the stack of short exact sequences of $Q^{\dag}$-representations. 
We will show that, 
for 
$l\ge 0$ and a sequence of integers 
$0\le j_1 \le \cdots \le j_l \le m-l$, 
the categorified Hall product gives a fully-faithful functor 
\begin{align}\label{intro:FF:global}
	\boxtimes_{i=1}^l \MF(\mM_Q^{\theta\sss}(s_m), w)_{j_i+(2i-1)(m^2-m)}
	\boxtimes \left(\mathbb{W}_{\rm{glob}}^{\theta_-}(v-ls_m) \otimes \chi_0^{j_l+2l(m^2-m)}  \right) 
	\to 	\mathbb{W}_{\rm{glob}}^{\theta_+}(v)
\end{align}
whose essential images form a semiorthogonal decomposition. 
Here the subscript $j_i+(2i-1)(m^2-m)$ indicates the fixed 
$\C$-weight part, and $\chi_0$ is some character regarded as a line 
bundle on $\mM_Q^{\dag}(v)$
(see Theorem~\ref{thm:wincon} for details). 
It follows that the categorified Hall products describe the semiorthogonal 
complement of (\ref{intro:inclu}), which lead to a proof of Theorem~\ref{intro:thm2}. 

In order to show that the functor (\ref{intro:FF:global}) is fully-faithful and they 
form a semiorthogonal decomposition, we prove these statements 
formally locally on the good moduli space $M_Q^{\dag, \theta \sss}(v)$
at any point $p$
corresponding to a $\theta$-polystable $(Q^{\dag}, W)$-representation $R$. 
By the \'{e}tale slice theorem, one can describe the formal 
fibers of the diagram (\ref{intro:dia:wall}) at $p$ 
in terms of a wall-crossing diagram 
of 
the Ext-quiver $Q_p^{\dag}$ associated with $R$, 
which is much simpler than $Q^{\dag}$. 
After removing a quadratic part of the super-potential, one 
observes that the wall-crossing 
diagram for $Q_p^{\dag}$-representations
is the product of a Grassmannian flip with 
some trivial part. 
Here a \textit{Grassmannian flip} is a birational map
\begin{align*}
	G_{a, b}^+(d) \dashrightarrow G_{a, b}^-(d)
	\end{align*}
given by two GIT stable loci of the quotient stack 
\begin{align*}
	\gG_{a, b}(d)=\left[(\Hom(A, V) \oplus \Hom(V, B))/\GL(V) \right]
	\end{align*}
where $d=\dim V$, $a=\dim A$, $b=\dim B$ with $a\ge b$. 

Donovan-Segal~\cite{DoSe}
proved a derived equivalence $D^b(G_{a, b}^-(d)) \simeq D^b(G_{a, b}^+(d))$
in the case of $a=b$ (i.e. Grassmannian flop) using window 
subcategories, and 
the same argument also applies to 
construct a fully-faithful functor $D^b(G_{a, b}^-(d)) \hookrightarrow 
D^b(G_{a, b}^+(d))$. 
However it is in a rather recent work~\cite{BNFV}
where the semiorthogonal complement of 
the above fully-faithful functor is considered. 
We will interpret the description of semiorthogonal complement in~\cite{BNFV} 
in terms of categorified Hall products,
and refine it as a semiorthogonal decomposition (see Corollary~\ref{cor:DG})
\begin{align}\label{sod:grass}
D^b(G_{a, b}^+(d))=\left\langle \binom{a-b}{l}\mbox{-copies of }
D^b(G_{a, b}^-(d-l)) : 0\le l\le d   \right\rangle. 	
\end{align}
The above semiorthogonal decomposition 
unifies Kapranov's exceptional collections of derived categories of 
Grassmannians, and also semiorthogonal decompositions of standard toric 
flips, so it may be of independent interest (see Remark~\ref{rmk:exceptional}, 
Remark~\ref{rmk:d=1}). 

A semiorthogonal decomposition similar to (\ref{sod:grass}) 
also holds 
for categories of factorizations of a super-potential 
of $\gG_{a, b}(d)$. 
Under the Kn\"{o}rrer periodicity, 
we compare global categorified Hall products (\ref{intro:FF:global})
with local categorified Hall products giving
the semiorthogonal decomposition 
(\ref{sod:grass}). 
By combining these arguments, 
we see that the functor (\ref{intro:FF:global}) is fully-faithful 
and they form a semiorthogonal decomposition formally locally 
on $M_Q^{\dag, \theta\sss}(v)$, 
hence they also hold globally.

	\subsection{Related works}
	The wall-crossing formula (\ref{intro:wcf}) was proved by Nagao-Nakajima~\cite{NN}
	in order to give an understanding of the product expansion formula of 
	non-commutative DT invariants of the conifold studied by Szendr{\H o}i~\cite{Sz}. 
	The wall-crossing formula (\ref{intro:wcf}) was later 
	extended to
	the case of a global 
	flopping contraction in~\cite{Tcurve2, Calab}, 
	and to 
	 the motivic DT invariants in~\cite{MMNS}. 
	Recently Tasuki Kinjo studies cohomological DT theory on the resolved 
	conifold and proves a cohomological version of DT/PT correspondence 
	in this case~\cite{Kinjo3}. It would be interesting to extend the argument in this paper 
	and categorify his cohomological DT/PT correspondence. 
	
	As we already mentioned, the study of 
	wall-crossing of categorical PT invariants
	was posed in~\cite{TocatDT}. In the case of local 
	surfaces, a categorical wall-crossing formula is conjectured  in~\cite[Conjecture~6.2.6]{TocatDT}
	in the case of simple wall-crossing, and proved in 
	some cases in~\cite[Theorem~6.3.19]{TocatDT}
	using Porta-Sala categorified Hall products for surfaces~\cite{PoSa}. 
	The wall-crossing we consider in this paper is not necessary 
	simple, so it is beyond the cases we considered in~\cite[Conjecture~6.2.6]{TocatDT}. A similar wall-crossing 
	at $(-1, -1)$-curve
	is also considered in~\cite[Section~7]{Totheta}, but we only proved the 
	existence of fully-faithful functors and their semiorthogonal 
	complements are not considered. 
	
	The categorified (K-theoretic) Hall algebras for quivers with super-potentials 
	was introduced and studied by Tudor P{\u{a}}durariu~\cite{Tudor, Tudor1.5}. 
	He also proved the PBW theorem for K-theoretic Hall algebras~\cite{Tudor, Tudor1.7} via much more 
	sophisticated combinatorial arguments (based on earlier works~\cite{MR3698338, HLKSAM}). 
	We expect that his arguments proving the K-theoretic PBW theorem can be applied to prove
	categorical (or K-theoretic) wall-crossing formula 
	in a broader setting, 
	including DT/PT wall-crossing in this paper. 
	
	Recently 
	Qingyuan Jiang~\cite{JiangQuot} studies
	 derived categories of Quot schemes of 
	locally free quotients, and proposed 
	conjectural semiorthogonal decompositions of them (see~\cite[Conjecture~A.5]{JiangQuot}). 
	He proved the above conjecture 
	in the case of rank two quotients. 
	His conjectural semiorthogonal decompositions resemble 
	the one in Theorem~\ref{intro:thm2}.  
	It would be interesting to see whether the technique in this 
	paper can be applied to his conjecture.

	\subsection{Acknowledgements}
	The author is grateful to Qingyuan Jiang
	and Tudor P{\u{a}}durariu
	for comments on the first version of this paper. 
	The author is supported by World Premier International Research Center
	Initiative (WPI initiative), MEXT, Japan, and Grant-in Aid for Scientific
	Research grant (No.~19H01779) from MEXT, Japan.

	\subsection{Notation and Convention}\label{subsec:notation}
	In this paper, all the schemes or stacks are defined over $\mathbb{C}$. 
	For an Artin stack $\yY$, we denote by $D^b(\yY)$ the 
	bounded derived category of coherent sheaves on $\yY$. 
	For an algebraic group $G$ and its representation $V$, 
	we regard it as a vector bundle on $BG$. 
	For a variety $Y$ on which $G$ acts, we denote by 
	$V \otimes \oO_{[Y/G]}$ the vector bundle 
	given by the pull-back of $V$ by $[Y/G] \to BG$. 	
	For a morphism $\fM \to M$ from a stack $\fM$ to a scheme $M$
	and a closed point $y \in M$, the \textit{formal fiber} at $y$ is defined by 
	\begin{align*}
		\widehat{\fM}_y \cneq \fM \times_{M} \Spec \widehat{\oO}_{M, y} 
		\to \widehat{M}_y \cneq \Spec \widehat{\oO}_{M, y}. 
		\end{align*}
		For a triangulated category $\dD$, its
	triangulated subcategory $\dD' \subset \dD$ is called 
	\textit{dense} if any object in $\dD$ is a direct summand of 
	an object in $\dD'$.

	\section{Preliminary}
	In this section, we review triangulated categories of factorizations, 
	the window theorem for categories of factorizations over 
	 GIT quotient stacks, and the Kn\"{o}rrer periodicity. 	
	\subsection{The category of factorizations}\label{subsec:MF}
		Let $\yY$ be a noetherian algebraic stack 
	over $\mathbb{C}$
	and take $w \in \Gamma(\oO_{\yY})$. 
	A (coherent) factorization of $w$ consists of 
\begin{align}\notag
	\xymatrix{
		\pP_0 \ar@/^8pt/[r]^{\alpha_0} &  \ar@/^8pt/[l]^{\alpha_1} \pP_1,  
	} \
	\alpha_0 \circ \alpha_1=\cdot w, \ 
	\alpha_1 \circ \alpha_0=\cdot w, 
\end{align}
where each $\pP_i$ is a coherent sheaf on $\yY$ and 
$\alpha_i$ are morphisms of coherent sheaves. 
The category of coherent factorizations naturally
forms a dg-category, whose homotopy 
category is denoted by $\mathrm{HMF}(\yY, w)$. 
The subcategory of absolutely acyclic 
objects 
\begin{align*}
	\mathrm{Acy}^{\rm{abs}} \subset \mathrm{HMF}(\yY, w)
	\end{align*}
is defined to be the minimum thick triangulated subcategory 
which contains totalizations of short exact sequences of 
coherent factorizations of $w$. 
The triangulated category of factorizations of $w$ is defined by 
(cf.~\cite{Ornonaff, MR3366002, MR3112502})
\begin{align*}
	\mathrm{MF}(\yY, w) \cneq \mathrm{HMF}(\yY, w)/\mathrm{Acy}^{\rm{abs}}. 
	\end{align*}
If $\yY$ is an affine scheme, then $\MF(\yY, w)$ 
is equivalent to Orlov's triangulated category of matrix factorizations 
of $w$~\cite{Orsin}. 
For two pairs $(\yY_i, w_i)$ for $i=1, 2$, 
we use the notation 
\begin{align*}
	\MF(\yY_1, w_1) \boxtimes \MF(\yY_2, w_2)
\cneq \MF(\yY_1 \times \yY_2, w_1 + w_2).
\end{align*}

It is well-known that 
$\MF(\yY, w)$ only depends on 
an open neighborhood of $\Crit(w) \subset \yY$. 
Namely let $\yY' \subset \yY$ be an open substack 
such that $\Crit(w) \subset \yY'$. Then 
the restriction functor gives an equivalence (see~\cite[Corollary~5.3]{MR3112502}, \cite[Lemma~5.5]{HLKSAM})
\begin{align}\label{Pre:rest}
	\MF(\yY, w) \stackrel{\sim}{\to} \MF(\yY', w|_{\yY'}). 
	\end{align}

Suppose that $\yY=[Y/G]$ where $G$ is an algebraic group which acts on 
a scheme $Y$. 
Assume that $\C \subset G$ lies in the center of $G$ which acts on $Y$ trivially. 
Then $\MF(\yY, w)$ decomposes into the direct sum
\begin{align}\label{decom:Y}
	\MF(\yY, w)=\bigoplus_{j \in \mathbb{Z}} \MF(\yY, w)_{j}
	\end{align}
where each summand corresponds to the
$\C$-weight $j$-part.

	\subsection{Attracting loci}
		Let $G$ be a reductive algebraic group with maximal torus $T$, 
	which acts on a smooth affine scheme $Y$. 
	We denote by 
	$M$ the character lattice of $T$ and $N$ the cocharacter lattice
	of $T$. 
	There is a perfect pairing 
	\begin{align*}
		\langle -, -\rangle \colon N \times M \to \mathbb{Z}. 
	\end{align*}
	For a one parameter subgroup $\lambda \colon \C \to G$, 
	let $Y^{\lambda \ge 0}$, $Y^{\lambda=0}$ be defined by 
	\begin{align*}
		Y^{\lambda \ge 0} &\cneq \{y \in Y: 
		\lim_{t\to 0} \lambda(t)(y) \mbox{ exists }\}, \\
		Y^{\lambda = 0} &\cneq \{y \in Y: 
		\lambda(t)(y)=y \mbox{ for all } t \in \C\}. 
	\end{align*}
	The Levi subgroup and the parabolic subgroup
	\begin{align*}
		G^{\lambda=0} \subset G^{\lambda \ge 0} \subset G
	\end{align*}
	are also similarly defined by the conjugate $G$-action on $G$, i.e. 
	$g\cdot (-)=g(-)g^{-1}$. The $G$-action on 
	$Y$ restricts to the $G^{\lambda \ge 0}$-action 
	on $Y^{\lambda \ge 0}$, 
	and the $G^{\lambda=0}$-action on $Y^{\lambda=0}$. 
	We note that
	$\lambda$ factors through 
	$\lambda \colon \C \to G^{\lambda=0}$, and it
	acts on $Y^{\lambda=0}$ trivially. 
	So we have the decomposition into $\lambda$-weight spaces
	\begin{align*}
		D^b([Y^{\lambda=0}/G^{\lambda=0}])=
		\bigoplus_{j \in \mathbb{Z}}D^b([Y^{\lambda=0}/G^{\lambda=0}])_{\lambda \mathchar`- \wt=j}.
		\end{align*}
	We have the diagram of 
	attracting loci
	\begin{align}\label{dia:attract}
		\xymatrix{
			[Y^{\lambda \ge 0}/G^{\lambda \ge 0}] \ar[r]^-{p_{\lambda}}	 \ar[d]^-{q_{\lambda}}	&
			[Y/G] \\
			[Y^{\lambda=0}/G^{\lambda=0}]\ar@/^10pt/[u]^-{\sigma_{\lambda}}. & 	
		}
	\end{align}
	Here $p_{\lambda}$ is induced by the inclusion 
	$Y^{\lambda \ge 0} \subset Y$, and $q_{\lambda}$ is given by 
	taking the $t \to 0$ limit of the action of $\lambda(t)$
	for $t \in \C$. The morphism $\sigma_{\lambda}$ is a section of 
	$q_{\lambda}$ induced by inclusions $Y^{\lambda=0} \subset Y^{\lambda \ge 0}$
	and $G^{\lambda=0} \subset G^{\lambda \ge 0}$. 
	We will use the following lemma: 
	\begin{lem}\label{lem:vanish}\emph{(\cite[Corollary~3.17, Amplification~3.18]{MR3327537})}
		
		(i)
		For $\eE_i \in D^b([Y^{\lambda \ge 0}/G^{\lambda \ge 0}])$
		with $i=1, 2$, suppose that 
		\begin{align*}
			\sigma_{\lambda}^{\ast}\eE_1 \in D^b([Y^{\lambda=0}/G^{\lambda=0}])_{\lambda \mathchar`- \wt\ge j}, \ \sigma_{\lambda}^{\ast}\eE_2 \in D^b([Y^{\lambda=0}/G^{\lambda=0}])_{\lambda \mathchar`- \wt <j}
			\end{align*}
		for some $j$. 
		Then $\Hom(\eE_1, \eE_2)=0$. 
		
		(ii) For $j \in \mathbb{Z}$, the functor 
		\begin{align*}
			q_{\lambda}^{\ast} \colon 
			D^b([Y^{\lambda=0}/G^{\lambda=0}])_{\lambda  \mathchar`- \wt=j}
			\to D^b([Y^{\lambda \ge 0}/G^{\lambda \ge 0}])
			\end{align*} 
		is fully-faithful. 
	\end{lem}
	
	\subsection{Kempf-Ness stratification}\label{subsec:exam:KN}
Here review Kempf-Ness stratifications associated with GIT quotients of 
reductive algebraic groups, and the corresponding window theorem 
following the convention of~\cite[Section~2.1]{MR3327537}. 
Let $Y$, $G$ be as in the previous subsection. 
For an element 
$l \in 
\Pic([Y/G])_{\mathbb{R}}$, 
we have the open subset of $l$-semistable points
\begin{align}\notag
	Y^{l\sss} \subset Y
\end{align}
characterized 
by the set of points $y \in Y$ such that 
for any one parameter subgroup $\lambda \colon \mathbb{C}^{\ast} \to G$
such that 
the limit $z=\lim_{t\to 0}\lambda(t)(y)$
exists in $Y$, we 
have $\wt(l|_{z})\ge 0$. 
Let $\lvert \ast \rvert$ be the Weyl-invariant norm 
on $N_{\mathbb{R}}$. 
The above subset of $l$-semistable 
points fits into the \textit{Kempf-Ness (KN) stratification} 
\begin{align}\label{KN:strata}
	Y=S_{1} \sqcup S_{2} \sqcup \cdots \sqcup S_N \sqcup 
	Y^{l\sss}.  
\end{align}
Here for each $1\le i\le N$ there exists a 
one parameter subgroup $\lambda_{i} \colon \mathbb{C}^{\ast} \to
T \subset G$, an open and closed subset 
$Z_{i}$ of
$(Y \setminus \cup_{i'<i} S_{i'})^{\lambda_{i}=0}$
(called \textit{center} of $S_i$)
such that 
\begin{align*}
	S_{i}=G \cdot Y_{i}, \ 
	Y_{i}\cneq \{ y \in Y^{\lambda_{i} \ge 0}: 
	\lim_{t \to 0}\lambda_{i}(t)(y) \in Z_{i}\}. 
\end{align*}
Moreover by setting the slope to be
\begin{align}\notag
	\mu_{i} \cneq -\frac{
		\wt(l|_{Z_{i}})}{\lvert \lambda_{i} \rvert} \in \mathbb{R}
\end{align}
we have 
the inequalities
$\mu_1>\mu_2>\cdots>0$.
We have the following diagram (see~\cite[Definition~2.2]{MR3327537})
\begin{align}\label{dia:YZ}
	\xymatrix{
		[Y_{i}/G^{\lambda_{i}\ge 0}] \ar[r]^-{\cong} \ar[d] & [S_{i}/G] \ar[dl]_{p_{i}} \ar@<-0.3ex>@{^{(}->}[r]^-{q_{i}}
		& \left[\left(Y\setminus \cup_{i'<i} S_{i'}\right)/G \right]  \\
		[Z_{i}/G^{\lambda_{i}=0}]. \ar@<-0.3ex>@{^{(}->}[rru]_-{\tau_{i}} & & 
	}
\end{align}
Here the left vertical arrow is given by
taking the $t\to 0$ limit of the action of $\lambda_{i}(t)$
for $t \in \C$, and $\tau_{i}, q_{i}$ are induced by the 
embedding $Z_{i} \hookrightarrow Y$, $S_{i} \hookrightarrow Y$ respectively. 
Let $\eta_{i} \in \mathbb{Z}$ be defined by 
\begin{align}\label{etai}
	\eta_{i} \cneq \wt_{\lambda_{i}}(\det(N_{S_{i}/Y}^{\vee}|_{Z_{i}})).
\end{align}
In the case that $Y$ is a $G$-representation, 
it is also written as 
\begin{align*}
	\eta_i =\langle \lambda_i, (Y^{\vee})^{\lambda_i>0}-(\mathfrak{g}^{\vee})^{\lambda_i>0}   \rangle. 
	\end{align*}
Here for a $G$-representation $W$ and a one parameter 
subgroup $\lambda \colon \C \to T$, 
we denote by $W^{\lambda>0} \in K(BT)$ the 
subspace of $W$ spanned by weights which pair positively with $\lambda$. 
We will use the following version of window theorem. 
\begin{thm}\label{thm:window}\emph{(\cite{MR3327537, MR3895631})}
	For each $i$, we take $m_{i} \in \mathbb{R}$. 
For $N' \le N$, let 
\begin{align}\label{window:m}
	\mathbb{W}_{m_{\bullet}}^l([(Y \setminus \cup_{1\le i\le N'} S_{i})/G]) \subset 
	D^b([(Y \setminus \cup_{1\le i\le N'} S_{i})/G])
\end{align}
be the subcategory of objects $\pP$
satisfying the condition
	\begin{align}\label{condition:P}
	\tau_{i}^{\ast}(\pP) \in 
	\bigoplus_{j \in [m_{i}, m_{i}+\eta_{i})}
	D^b([Z_{i}/G^{\lambda_i=0}])_{\lambda_{i} \mathchar`- \wt= j}
\end{align}
for all $N' <i\le N$. 
Then the composition functor 
\begin{align*}
	\mathbb{W}_{m_{\bullet}}^l([(Y \setminus \cup_{1\le i\le N'} S_{i})/G]) \hookrightarrow 
	D^b([(Y\setminus \cup_{1\le i\le N'}S_i)/G])
	\twoheadrightarrow
	D^b([Y^{l\sss}/G])
\end{align*}
is an equivalence. 
\end{thm}

Let $w \colon Y \to \mathbb{A}^1$ be a $G$-invariant function. 
We will apply Theorem~\ref{thm:window} for a KN stratification 
of $\Crit(w)$
\begin{align*}
	\Crit(w)= S_1' \sqcup S_2' \sqcup \cdots \sqcup S_N' \sqcup \Crit(w)^{l\sss}
\end{align*}
in the following way. 
After discarding 
KN strata $S_{i} \subset Y$ 
with $\Crit(w) \cap S_{i}=\emptyset$, 
the above stratification is obtained by
restricting a KN stratification (\ref{KN:strata}) 
for $Y$ to $\Crit(w)$. 
Let $\lambda_{i} \colon \C \to G$ be a 
one parameter subgroup for $S_{i}'$
with center $Z_{i}' \subset S_{i}'$. 
We define 
$\overline{Z}_{i} \subset Y$ to 
be the union of connected components 
of the $\lambda_{i}$-fixed part of $Y$
which contains $Z_{i}'$, 
and $\overline{Y}_{i} \subset Y$ is the set of 
points $y \in Y$ with 
$\lim_{t\to 0}\lambda_{i}(t)y \in \overline{Z}_{i}$. 
Similarly to (\ref{dia:YZ}), 
we have the diagram
\begin{align}\notag
	\xymatrix{
		[\overline{Y}_{i}/G^{\lambda_{i}\ge 0}] \ar@/^15pt/[rr]^-{\overline{q}_{i}} 
		\ar[d]_-{\overline{p}_{i}} \ar@<-0.3ex>@{^{(}->}[r] & 
		[Y^{\lambda_{i}\ge 0}/G^{\lambda_{i}\ge 0}] \ar[d] 
		\ar[r] & [Y/G] \\
		[\overline{Z}_{i}/G^{\lambda_{i}= 0}] \ar@<-0.3ex>@{^{(}->}[r] & [Y^{\lambda_{i}=0}/G^{\lambda_{i}=0}]. \ar[ru] & 	
	}
\end{align}
Here the left horizontal arrows are 
open and closed immersions. 	
Using the equivalence (\ref{Pre:rest}), 
we also have the following 
version of window theorem for factorization categories 
(see~\cite[Subsection~2.4]{Totheta})
\begin{thm}\label{thm:window:MF}
For each $i$, we take $m_i \in \mathbb{R}$. 
For $N' \le N$, let 
\begin{align*}
	\mathbb{W}_{m_{\bullet}}^l([(Y \setminus \cup_{1\le i\le N'}S_i')/G], w) \subset \MF([(Y \setminus \cup_{1\le i\le N'}S_i')/G], w)
	\end{align*}
be the subcategory
consisting of factorizations $(\pP, d_{\pP})$ such that 
	\begin{align}\label{cond:P}
		(\pP, d_{\pP})|_{[(\overline{Z}_{i} \setminus 
			\cup_{i'<i}S_{i'}')/G^{\lambda_i=0}]} \in 
		\bigoplus_{j \in [m_{i}, m_{i}+\overline{\eta}_{i})}
		\mathrm{MF}[(\overline{Z}_{i} \setminus 
		\cup_{i'<i}S_{i'}')/G^{\lambda_i=0}], w|_{\overline{Z}_{i}})_{\lambda_{i} \mathchar`- \wt= j}
	\end{align}
for all $N'<i \le N$. 
Here $\overline{\eta}_{i}=\wt_{\lambda_{i}}\det (\mathbb{L}_{\overline{q}_{i}})^{\vee}|_{\overline{Z}_{i}}$. 
Then the composition functor
	\begin{align*}
		\mathbb{W}_{m_{\bullet}}^l([(Y \setminus \cup_{1\le i\le N'}S_i')/G], w) \hookrightarrow 
		\MF([(Y \setminus \cup_{1\le i\le N'}S_i')/G], w)
		\twoheadrightarrow 
		\MF([Y^{l\sss}/G], w)
	\end{align*}
	is an equivalence.  
\end{thm}
	
		\subsection{Kn\"{o}rrer periodicity}
	Let $Y$ be a smooth affine scheme and $G$ be an affine 
	algebraic group which acts on $Y$. 
	Let $W$ be a $G$-representation, which determines 
	a vector bundle $\wW \to \yY \cneq [Y/G]$. 
	Given a function $w \colon \yY \to \mathbb{A}^1$, 
	we have another function on the total space of 
	$\wW \oplus \wW^{\vee}$
	\begin{align*}
		w+q \colon \wW \oplus \wW^{\vee} \to \mathbb{A}^1, \ 
		q(x, x')=\langle x, x'\rangle. 
	\end{align*}
	We have the following diagram 
	\begin{align*}
		\xymatrix{
			\wW^{\vee} 
			\ar@<-0.3ex>@{^{(}->}[r]^-{i}
			 \ar[rd]_-{w} \ar[d]_{\mathrm{pr}} 
			& \wW \oplus \wW^{\vee} \ar[d]^-{w+q} \\
			\yY \ar[r]_-{w}	& \mathbb{A}^1. 	
		}
	\end{align*}
	Here $i(x)=(0, x)$. 
	The following is a version of the Kn\"{o}rrer periodicity 
	(cf.~\cite[Theorem~4.2]{MR3631231}):
	\begin{thm}\label{thm:period}
		The following composition functor is an equivalence  
		\begin{align}\label{Phi:period}
			\Phi \cneq  
			i_{\ast}\mathrm{pr}^{\ast} \colon 
			\MF(\yY, w) \stackrel{\mathrm{pr}^{\ast}}{\to}
			\MF(\wW^{\vee}, w) \stackrel{i_{\ast}}{\to}
			\MF(\wW \oplus \wW^{\vee}, w+q). 
		\end{align}
	\end{thm}
	The equivalence (\ref{Phi:period}) is given by 
	taking the tensor product over $\oO_{\yY}$
	 with the following factorization of $q$ on $\wW \oplus \wW^{\vee}$
	\begin{align*}
		\xymatrix{
			i_{\ast}\oO_{\wW^{\vee}} \ar@<1ex>[r] & 0.  \ar@<1ex>[l]. 	
		}
	\end{align*}
	The above factorization is isomorphic to the Koszul factorization
	of $q$ on $\wW \oplus \wW^{\vee}$, 
	which is of the form (see~\cite[Proposition~3.20]{MR3270588})
	
	\begin{align}\label{Koszul}
		\xymatrix{
			\left(\bigwedge^{\rm{even}}\wW^{\vee} \right) \otimes_{\oO_{\yY}} \oO_{\wW \oplus \wW^{\vee}} 
			\ar@<1ex>[r] & \left(\bigwedge^{\rm{odd}}\wW^{\vee} \right) \otimes_{\oO_{\yY}} \oO_{\wW \oplus \wW^{\vee}}.  \ar@<1ex>[l]. 	
		}
	\end{align}
	
	
	Let $\lambda \colon \C \to G$ be a one parameter subgroup. 
	We have the following diagrams of attracting loci 
	\begin{align*}
		\xymatrix{
			\yY^{\lambda \ge 0} \ar[rd]^-{w^{\lambda \ge 0}}\ar[r]^-{p_{\lambda}} \ar[d]_-{q_{\lambda}} & \yY \ar[d]^-{w}, \\
			\yY^{\lambda=0} \ar[r]_-{w^{\lambda=0}} & \mathbb{A}^1,  
		}
		\quad 
		\xymatrix{
			(\wW \oplus \wW^{\vee})^{\lambda \ge 0} \ar[rd]^-{(w+q)^{\lambda \ge 0}}\ar[r]^-{p'_{\lambda}} 
			\ar[d]_-{q'_{\lambda}} & \wW \oplus \wW^{\vee} \ar[d]^-{w+q}. \\
			(\wW \oplus \wW^{\vee})^{\lambda=0} 
			\ar[r]_-{w^{\lambda=0}+q^{\lambda=0}}  & 	\mathbb{A}^1. 
		}
	\end{align*}
	Note that we have equivalences
	\begin{align*}
		&\Phi^{\lambda=0} \colon 
		\MF(\yY^{\lambda=0}, w^{\lambda=0}) \stackrel{\sim}{\to}
		\MF((\wW \oplus \wW^{\vee})^{\lambda=0}, w^{\lambda=0}+
		q^{\lambda=0}), \\
		&\Phi^{\lambda \ge 0} \colon 
		\MF(\yY^{\lambda \ge 0}, w^{\lambda \ge 0}) \stackrel{\sim}{\to}
		\MF((\wW^{\lambda \ge 0} \oplus (\wW^{\lambda \ge 0})^{\vee}, w^{\lambda\ge 0}+
		q^{\lambda\ge 0})
	\end{align*}
	by applying Theorem~\ref{thm:period} for $\wW^{\lambda=0} \to \yY^{\lambda=0}$,
	$\wW^{\lambda \ge 0} \to \yY^{\lambda \ge 0}$ respectively. 
	
	\begin{prop}\label{prop:Knoer}
		The following diagram commutes:
		\begin{align*}
			\xymatrix{
				\MF(\yY^{\lambda=0}, w^{\lambda=0}) \ar[r]^-{p_{\lambda\ast}q_{\lambda}^{\ast}} \ar[d]_-{\Phi^{\lambda=0} \circ \otimes (\det \wW^{\lambda>0})^{\vee}[\dim W^{\lambda>0}]} & 
				\MF(\yY, w) \ar[d]_-{\Phi} \\
				\MF((\wW\oplus \wW^{\vee})^{\lambda=0}, w^{\lambda=0}+q^{\lambda=0}) \ar[r]_-{p_{\lambda\ast}'q_{\lambda}^{'\ast}} & 
				\MF(\wW \oplus \wW^{\vee}, w+q). 
			}	
		\end{align*}
	\end{prop}
	\begin{proof}
		We have the following diagram
		\begin{align}\label{dia:YYWW}
			\xymatrix{
	\wW^{\lambda=0} \ar[d] & \ar[l] \wW^{\lambda \ge 0} \ar[r] \ar[d] & \wW \ar[d] \\
	\yY^{\lambda=0} & \ar[l]^-{q_{\lambda}} \yY^{\lambda \ge 0} \ar[r]_-{p_{\lambda}} & \yY. 		
	}
			\end{align}
			Here each horizontal diagrams are diagrams of attracting loci, 
			and vertical arrows are projections. 
			From the above diagram, 
			we construct the following diagram 
		\begin{align}\label{dia:comX}
			\xymatrix{	
				(\wW \oplus \wW^{\vee})^{\lambda \ge 0}
				\ar@/_50pt/[dd]_-{q'_{\lambda}} 
				\ar[r]_-{r_1} \ar[d]_-{r_2} \ar@/^20pt/[rr]^-{p'_{\lambda}}\diasquare & 
				\wW^{\lambda \ge 0} \oplus p_{\lambda}^{\ast}\wW^{\vee} \ar[r]_-{f_2} \ar[d]^-{g_2} & \wW \oplus \wW^{\vee} \ar[dd]^-{w+q}  \\
				\wW^{\lambda \ge 0} \oplus q_{\lambda}^{\ast}(\wW^{\lambda=0})^{\vee}	
				\ar[r]_-{g_1} \ar[d]_-{f_1} & \wW^{\lambda \ge 0} \oplus 
				(\wW^{\lambda \ge 0})^{\vee} 
				\ar[dr]^-{w^{\lambda \ge 0}+q^{\lambda \ge 0}} & \\
				\wW^{\lambda=0} \oplus (\wW^{\lambda=0})^{\vee}
				\ar[rr]_-{w^{\lambda=0}+q^{\lambda=0}}  & 
				& \mathbb{A}^1,
			} 
		\end{align}
	Here $(f_1, f_2)$ is induced by the top horizontal diagram in (\ref{dia:YYWW}), 
	$(g_1, g_2)$ is induced by the duals of
	the morphisms of vector bundles 
	\begin{align*}
		q_{\lambda}^{\ast}\wW^{\lambda=0} \leftarrow \wW^{\lambda \ge 0} 
	\to p_{\lambda}^{\ast}\wW
	\end{align*} 
on $\yY^{\lambda \ge 0}$,
	and $(r_1, r_2)$ is induced by the diagram of 
	attracting loci
	$(\wW^{\vee})^{\lambda=0} \leftarrow (\wW^{\vee})^{\lambda \ge 0} \to 
	\wW^{\vee}$ for $\wW^{\vee}$. 
	
	By applying Lemma~\ref{lem:commute1} for the right square of (\ref{dia:YYWW}) 
	 (and also noting that $p_{\lambda}$, $f_2$
	are proper), 
	we have the commutative diagram:
	\begin{align*}
		\xymatrix{
			\MF(\yY^{\lambda \ge 0}, w^{\lambda \ge 0}) \ar[r]^-{p_{\lambda\ast}} \ar[d]_-{\Phi^{\lambda \ge 0}} & 
			\MF(\yY, w) \ar[d]_-{\Phi} \\
			\MF(\wW^{\lambda \ge 0} \oplus (\wW^{\lambda \ge 0})^{\vee}, w^{\lambda \ge 0}+q^{\lambda \ge 0}) \ar[r]_-{f_{2\ast}g_2^{\ast}} & 
			\MF(\wW \oplus 
			\wW^{\vee}, w+q).
		}
	\end{align*}
Similarly by applying Lemma~\ref{lem:commute2} for the left square of (\ref{dia:YYWW}), 
	we have the commutative diagram: 
		\begin{align*}
			\xymatrix{
				\MF(\yY^{\lambda=0}, w^{\lambda=0}) \ar[r]^-{q_{\lambda}^{\ast}} \ar[d]_-{\Phi^{\lambda=0}} & 
				\MF(\yY^{\lambda \ge 0}, w^{\lambda \ge 0}) \ar[d]_-{\Phi^{\lambda \ge 0}} \\
				\MF((\wW\oplus \wW^{\vee})^{\lambda=0}, w^{\lambda=0}+q^{\lambda=0}) \ar[r]_-{g_{1!}f_1^{\ast}} & 
				\MF(\wW^{\lambda \ge 0} \oplus 
				(\wW^{\lambda \ge 0})^{\vee}, w^{\lambda \ge 0}+q^{\lambda \ge 0}).
			}	
		\end{align*}
	
		Note that we have 
		\begin{align*}
			g_{1!}(-)=g_{1\ast}(- \otimes f_1^{\ast}\mathrm{pr}^{\ast}\det(\wW^{\lambda >0})^{\vee}[\dim W^{\lambda >0}]).
		\end{align*}
	Here $\mathrm{pr} \colon (\wW \oplus \wW^{\vee})^{\lambda=0} \to \yY^{\lambda=0}$ is the 
	projection. 
		By the diagram (\ref{dia:comX}) and the base change, we have the isomorphism of 
		functors 
		\begin{align*}
			f_{2\ast}g_2^{\ast}g_{1\ast}f_1^{\ast} \cong 
			p_{\lambda \ast}' q_{\lambda}^{'\ast} \colon 
			\MF(\wW^{\lambda \ge 0} \oplus (\wW^{\lambda \ge 0})^{\vee}, w^{\lambda \ge 0}+q^{\lambda \ge 0})
			\to \MF(\wW \oplus 
			\wW^{\vee}, w+q).
		\end{align*}
		Therefore the proposition holds. 
	\end{proof}

\section{Categorified Hall products for quivers with super-potentials}
In this section, we review categorified Hall products 
for quivers with super-potentials introduced in~\cite{Tudor, Tudor1.5}. 
	\subsection{Moduli stacks of representations of quivers}\label{subsec:catH}
	A \textit{quiver} consists of data $Q=(Q_0, Q_1, s, t)$, 
where $Q_0$, $Q_1$ are finite sets and $s, t \colon Q_1 \to Q_0$
are maps. The set $Q_0$ is the set of vertices, $Q_1$ is the set of 
edges, and $s$, $t$ are maps which assign source and target of each edge. 
A
\textit{$Q$-representation} consists of data 
\begin{align*}
	\mathbb{V}=\{(V_i, u_e) : i \in Q_0, u_e  \in \Hom(V_{s(e)}, V_{t(e)})\}
	\end{align*}
where each $V_i$ is a finite dimensional vector space. 
The dimension vector $v(\mathbb{V})$ of $\mathbb{V}$
is $(\dim V_i)_{i \in Q_0}$. 
For $v=(v_i)_{i \in Q_0} \in \mathbb{Z}_{\ge 0}^{Q_0}$, 
let $R_Q(v)$ be the vector space 
\begin{align*}
	R_Q(v)=\bigoplus_{e \in Q_1}\Hom(V_{s(e)}, V_{t(e)})
	\end{align*}
where $\dim V_i=v_i$. 
The algebraic group $G(v) \cneq \prod_{i\in Q_0} \GL(V_i)$ acts on $R_Q(v)$ 
by conjugation. 
The stack of $Q$-representations of dimension vector $v$ is given by 
the quotient stack 
\begin{align*}
	\mM_Q(v) \cneq [R_Q(v)/G(v)]. 
	\end{align*}
We discuss King's $\theta$-stability condition on $Q^{\dag}$-representations~\cite{Kin}.
We take 
\begin{align*}
	\theta =(\theta_i)_{i \in Q_0} \in \mathbb{R}^{Q_0}. 
	\end{align*}
For a dimension vector $d \in \mathbb{Z}_{\ge 0}^{Q_0}$, 
we set $\theta(v) =\sum_{i \in Q_0} \theta_i v_i$. 
For a $Q$-representation $\mathbb{V}$, we set 
$\theta(\mathbb{V}) \cneq \theta(v(\mathbb{V}))$. 
\begin{defi}\label{def:thetastab}
	A $Q$-representation $\mathbb{V}$ is called $\theta$-(semi)stable 
	if $\theta(\mathbb{V})=0$ and for any non-zero subobject $\mathbb{V}' \subsetneq \mathbb{V}$
	we have $\theta(\mathbb{V}') <(\le) 0$. 
	\end{defi}
There is an open substack 
\begin{align*}
	\mM_Q^{\theta\sss}(v) \subset \mM_Q(v)
	\end{align*}
corresponding to 
$\theta$-semistable representations. 
By~\cite[Proposition~3.1]{Kin}, if each $\theta_i$ is an integer,   
the above open substack corresponds to the 
GIT semistable locus with respect to the character
\begin{align}\notag
	\chi_{\theta} \colon 
	G(v) \to \C, \ 
	(g_i)_{i \in Q_0} \mapsto \prod_{i\in Q_0} \det g_i^{-\theta_i}. 
	\end{align}
By taking the GIT quotient, it admits a 
good moduli space~\cite{MR3237451}
\begin{align}\label{gmoduli:M}
	\pi_{M} \colon \mM_Q^{\theta \sss}(v) \to M_Q^{\theta \sss}(v)
	\end{align}
such that each closed point of $M_Q^{\theta \sss}(v)$
corresponds to a $\theta$-polystable $Q$-representation. 

Let $(a_i, b_i) \in \mathbb{Z}_{\ge 0}^2$ be a pair of non-negative 
integers for each vertex $i \in Q_0$, 
and take $c \in \mathbb{Z}_{\ge 0}$. 
We define the extended quiver $Q^{\dag}$ 
so that its vertex set is $\{\infty\} \cup Q_0$, 
with edges consist of edges in $Q$ and 
\begin{align*}
	\sharp(\infty \to i)=a_i, \ 
	\sharp(i \to \infty)=b_i, \ 
	\sharp(\infty \to \infty)=c. 
	\end{align*}
The $\C$-rigidified moduli stack of $Q^{\dag}$-representations 
of dimension vector $(1, d)$ is given by 
\begin{align*}
	\mM_Q^{\dag}(v) \cneq 
	[R_{Q^{\dag}}(1, v)/G(v)]
	\end{align*}
where $1$ is the dimension vector at $\infty$. 
Note that there is a natural morphism 
$\mM_{Q^{\dag}}(1, v) \to \mM_Q^{\dag}(v)$
which is a trivial $\C$-gerbe. 
For $\theta=(\theta_{\infty}, \theta_i)_{i \in Q_0}$
with $\theta(1, v)=0$, 
the open substack of $\theta$-semistable representations 
\begin{align*}
	\mM_Q^{\dag, \theta\sss}(v) \subset \mM_Q^{\dag}(v)
	\end{align*}
is defined in a similar way. 
The condition $\theta(1, v)=0$ determines 
$\theta_{\infty}$ by 
$\theta_{\infty}=-\sum_{i\in Q_0} \theta_i v_i$, 
so we just write $\theta=(\theta_i)_{i \in Q_0}$.

\subsection{Categorified Hall products}
For a dimension vector $v \in \mathbb{Z}_{\ge 0}^{Q_0}$, let us take 
a decomposition 
\begin{align*}
	v=v^{(1)}+\cdots + v^{(l)}, \ v^{(j)} \in \mathbb{Z}_{\ge 0}^{Q_0}.
	\end{align*}
Let $V_i=\oplus_{j=1}^l V_i^{(j)}$ be a direct sum decomposition 
such that $\{V_i^{(j)}\}_{i\in Q_0}$ has dimension vector $v^{(j)}$. 
We take integers 
$\lambda^{(1)}>\cdots> \lambda^{(l)}$, 
and a one parameter subgroup $\lambda \colon \C \to G(v)$
which acts on $V_i^{(j)}$ by weight $\lambda^{(j)}$. 
We have the stack of attracting loci
\begin{align*}
	\mM_Q(v^{\bullet}) \cneq \left[ R_Q(v)^{\lambda \ge 0}/G(v)^{\lambda \ge 0} \right]. 
	\end{align*} 
The above stack is isomorphic to the stack of 
filtrations of $Q$-representations 
\begin{align}\label{filt:V}
	0=\mathbb{V}^{(0)} \subset \mathbb{V}^{(1)} \subset 
	\cdots \subset \mathbb{V}^{(v)} =\mathbb{V}
	\end{align}
such that each $\mathbb{V}^{(j)}/\mathbb{V}^{(j-1)}$ has dimension 
vector $v^{(j)}$. 
Moreover we have 
\begin{align*}
	\prod_{j=1}^{l} \mM_Q(v^{(j)})=
	\left[ R_Q(v)^{\lambda=0}/G(v)^{\lambda=0} \right]
	\end{align*}
and we have the diagram of attracting loci (see Subsection~\ref{subsec:exam:KN})
\begin{align}\label{dia:quiver}
	\xymatrix{
	\mM_Q(v^{\bullet}) \ar[r]^-{p_{\lambda}} \ar[d]_-{q_{\lambda}} & \mM_Q(v) \\
	\prod_{j=1}^l \mM_Q(v^{(j)}). & 
}
	\end{align}
Here $p_{\lambda}$ sends a filtration (\ref{filt:V}) to $\mathbb{V}$, and 
$q_{\lambda}$ sends a filtration (\ref{filt:V}) to its associated graded 
$Q$-representation. 
Since $p_{\lambda}$ is proper, we have the functor 
(called \textit{categorified Hall product})
	\begin{align}\label{Hproduct}
		p_{\lambda \ast}q_{\lambda}^{\ast} \colon 
		\boxtimes_{j=1}^l D^b(\mM_Q(v^{(j)})) \to 
		D^b(\mM_Q(v)). 
		\end{align}
	For $\eE^{(j)} \in D^b(\mM_Q(v^{(j)}))$, we set 
	\begin{align*}
		\eE^{(1)} \ast \cdots \ast \eE^{(l)} \cneq 
		p_{\lambda\ast}q_{\lambda}^{\ast}(\eE^{(1)} \boxtimes 
		\cdots \boxtimes \eE^{(l)}). 
		\end{align*}
	The above $\ast$-product is associative, i.e. 
	\begin{align*}
		(\eE^{(1)} \ast \eE^{(2)})\ast \eE^{(3)} \cong
		\eE^{(1)} \ast (\eE^{(2)}\ast \eE^{(3)}) \cong 
		\eE^{(1)} \ast \eE^{(2)}\ast \eE^{(3)}.
		\end{align*}
	
	We take $\theta=(\theta_i)_{i\in Q_0}$ such that 
	$\theta(v^{(j)})=0$ for all $j$. 
	Then 
	the diagram (\ref{dia:quiver}) restricts to the 
	diagram 
		\begin{align}\label{dia:quiver:theta}
			\xymatrix{
				\mM_Q^{\theta\sss}(v^{\bullet}) \ar[r]^-{p_{\lambda}} \ar[d]_-{q_{\lambda}} & 
				\mM_Q^{\theta\sss}(v) \\
				\prod_{j=1}^l \mM_Q^{\theta\sss}(v^{(j)}). & 
			}
		\end{align}
		which is a diagram of attracting loci for 
	$\mM_Q^{\theta\sss}(v)$. Similarly we have the functor 
		\begin{align*}
		p_{\lambda \ast}q_{\lambda}^{\ast} \colon 
		\boxtimes_{j=1}^l D^b(\mM_Q^{\theta \sss}(v^{(j)})) \to 
		D^b(\mM_Q^{\theta \sss}(v)),
	\end{align*}
which coincides with (\ref{Hproduct}) when $\theta=0$. 
		
	Similarly applying the above construction for the 
	extended quiver $Q^{\dag}$, for a decomposition 
	\begin{align}\label{decompose:dag}
		v=v^{(1)}+\cdots +v^{(l)}+v^{(\infty)}, \ 
		v^{(j)} \in \mathbb{Z}^{Q_0}
		\end{align}
	such that $\theta(v^{(j)})=0$ for $1\le j\le l$, 
		we have the functor
	\begin{align}\label{Hproduct:dag}
		\boxtimes_{j=1}^l D^b(\mM_Q^{\theta \sss}(v^{(j)})) \boxtimes 
		D^b(\mM_Q^{\dag, \theta \sss}(v^{(\infty)})) \to 
		D^b(\mM_Q^{\dag, \theta \sss}(v))
		\end{align}
	which gives a left 
	action of $\bigoplus_v D^b(\mM_Q^{\theta \sss}(v))$
	on $\bigoplus_v D^b(\mM_Q^{\dag, \theta \sss}(v))$. 
	
	\subsection{Categorified Hall products for quivers with super-potentials}\label{subsec:catHw}
	Let $W$ be a super-potential of a quiver $Q$, i.e. 
	$W \in \mathbb{C}[Q]/[\mathbb{C}[Q], \mathbb{C}[Q]]$
		where $\mathbb{C}[Q]$ is the path algebra of $Q$. 
	Then there is a function 
	\begin{align}\label{func:tr}
		w \cneq \mathrm{Tr}(W) \colon \mM_Q(v) \to \mathbb{A}^1
		\end{align}
	whose critical locus is identified with the 
	moduli stack of $(Q, W)$-representations $\mM_{(Q, W)}(v)$, i.e. 
	$Q$-representations satisfying the relation $\partial W$. 
	
	The diagram (\ref{dia:quiver:theta}) is extended to the diagram 
		\begin{align}\label{dia:quiver:theta:1.5}
		\xymatrix{
			\mM_Q^{\theta\sss}(v^{\bullet}) \ar[r]^-{p_{\lambda}} \ar[d]_-{q_{\lambda}} & 
			\mM_Q^{\theta\sss}(v) \ar[d]^-{w} \\
			\prod_{j=1}^l \mM_Q^{\theta\sss}(v^{(j)}) \ar[r]^-{\sum_{j=1}^l w^{(j)}} & \mathbb{A}^1. 
		}
	\end{align}
Here $w^{(j)}$ is the function (\ref{func:tr}) 
on $\mM_Q(v^{(j)})$. 
Similarly to (\ref{Hproduct}), we have the functor between triangulated categories of factorizations 
	\begin{align*}
		p_{\lambda \ast}q_{\lambda}^{\ast} \colon 
		\boxtimes_{j=1}^l \MF(\mM_Q^{\theta \sss}(v^{(j)}), w^{(j)}) \to 
		\MF(\mM_Q^{\theta \sss}(v), w), 
	\end{align*}
called the \textit{categorified Hall products} for 
representations of quivers with super-potentials. 

The super-potential naturally defines the super-potential of the 
extended quiver $Q^{\dag}$, 
so we have the regular function 
$w \colon \mM_Q^{\dag}(v) \to \mathbb{A}^1$
as in (\ref{func:tr}). 
Similarly to (\ref{Hproduct:dag}), for a decomposition (\ref{decompose:dag})
we have the left action  
	\begin{align}\label{act:quiver}
	\boxtimes_{j=1}^l \MF(\mM_Q^{\theta \sss}(v^{(j)}), w^{(j)}) \boxtimes 
	\MF(\mM_Q^{\dag, \theta \sss}(v^{(\infty)}), w^{(\infty)}) \to 
	\MF(\mM_Q^{\dag, \theta \sss}(v), w).
\end{align}
Note that we have the decomposition (\ref{decom:Y})
with respect to the diagonal torus $\C \subset G(v)$
\begin{align*}
	\MF(\mM^{\theta \sss}_Q(v), w)=\bigoplus_{j \in \mathbb{Z}} \MF(\mM^{\theta \sss}_Q(v), w)_{j}. 
	\end{align*} 
We will often restrict
the functor (\ref{act:quiver}) to the fixed weight spaces 
\begin{align}\notag
	\boxtimes_{j=1}^l \MF(\mM_Q^{\theta \sss}(v^{(j)}), w^{(j)})_{i_j} \boxtimes 
	\MF(\mM_Q^{\dag, \theta \sss}(v^{(\infty)}), w^{(\infty)}) \to 
	\MF(\mM_Q^{\dag, \theta \sss}(v), w).
\end{align}

\subsection{Base change to formal fibers}\label{subsec:bchange}
Later we will take a base change of the categorified Hall product
to a formal neighborhood of a point 
in the good moduli space (\ref{gmoduli:M}). 
Note that the diagram (\ref{dia:quiver:theta:1.5}) 
extends to the commutative diagram 
	\begin{align}\label{dia:quiver:theta2}
	\xymatrix{
		\mM_Q^{\theta\sss}(v^{\bullet}) \ar[r]^-{p_{\lambda}} \ar[d]_-{q_{\lambda}} & 
		\mM_Q^{\theta\sss}(v) \ar[dd]^-{\pi_M}  \ar[rdd]^-{w} & \\
		\prod_{j=1}^l \mM_Q^{\theta\sss}(v^{(j)}) \ar[d]  & & \\
	\prod_{j=1}^l M_Q^{\theta\sss}(v^{(j)})	 \ar[r]^-{\oplus} & 
	M_Q^{\theta \sss}(v) \ar[r] & \mathbb{A}^1. 
	}
\end{align}
Here the bottom arrow is the morphism taking the direct sum of 
$\theta$-polystable representations
which is a finite morphism (see~\cite[Lemma~2.1]{MeRe}), 
and the left bottom vertical arrow is the good moduli space morphism. 
For a closed point $p \in M_Q^{\theta \sss}(v)$, 
we consider the following formal fiber 
\begin{align*}
	\widehat{\mM}_Q^{\theta \sss}(v)_p \cneq 
	\mM_Q^{\theta \sss}(v) \times_{M_Q^{\theta \sss}(v)} 
	\widehat{M}_Q^{\theta \sss}(v)_p
	\to \widehat{M}_Q^{\theta \sss}(v)_p
	\cneq \Spec \widehat{\oO}_{M_Q^{\theta \sss}(v), p}.
	\end{align*}
Let $(p^{(1)}, \ldots, p^{(l)}) \in \prod_{j=1}^l M_Q^{\theta \sss}(v^{(j)})$
be a point such that 
$\oplus(p^{(1)}, \ldots, p^{(l)})=p$. 
By taking the fiber product of the diagram (\ref{dia:quiver:theta2})
by $\widehat{M}_Q^{\theta \sss}(v)_p \to M_Q^{\theta \sss}(v)$, 
we obtain the diagram 
	\begin{align}\label{dia:quiver:theta3}
	\xymatrix{
		\widehat{\mM}_Q^{\theta\sss}(v^{\bullet})_{p} \ar[r]^-{\widehat{p}_{\lambda}} \ar[d]_-{\widehat{q}_{\lambda}} & 
\widehat{\mM}_Q^{\theta\sss}(v)_p \\
		\coprod_{p^{(\bullet)} \in \oplus^{-1}(p)}		\prod_{j=1}^l \widehat{\mM}_Q^{\theta\sss}(v^{(j)})_{p^{(j)}}. & 
	}
\end{align}
The above diagram is a diagram of attracting loci for 
$\widehat{\mM}_Q^{\theta \sss}(v)_p$
(see~\cite[Lemma~4.11]{Totheta}).
By the derived base change, 
we have the commutative diagram 
\begin{align}\label{bchange:1}
	\xymatrix{
\boxtimes_{j=1}^l D^b(\mM_Q^{\theta \sss}(v^{(j)})) 
\ar[r]^-{p_{\lambda \ast}q_{\lambda}^{\ast}} \ar[d] 
& D^b(\mM_Q^{\theta \sss}(v)) \ar[d] \\
	\bigoplus_{p^{(\bullet)} \in \oplus^{-1}(p)}\boxtimes_{j=1}^l D^b(\widehat{\mM}_Q^{\theta \sss}(v^{(j)})_{p^{(j)}})
\ar[r]^-{\widehat{p}_{\lambda \ast}\widehat{q}_{\lambda}^{\ast}}
& D^b(\widehat{\mM}_Q^{\theta \sss}(v)_p). 
}	
	\end{align} 
Here the vertical arrows are
pull-backs to formal fibers. 

We denote by $\widehat{w}_p \colon \widehat{\mM}_Q^{\theta \sss}(v)
\to \mathbb{A}^1$ the pull-back of the function 
(\ref{func:tr}) to the formal fiber. 
Similarly to (\ref{bchange:1}), we have the commutative diagram 
\begin{align}\label{bchange:2}
	\xymatrix{
		\boxtimes_{j=1}^l \MF(\mM_Q^{\theta \sss}(v^{(j)}), w)
		\ar[r]^-{p_{\lambda \ast}q_{\lambda}^{\ast}} \ar[d] 
		&\MF(\mM_Q^{\theta \sss}(v)) \ar[d] \\
		\bigoplus_{p^{(\bullet)} \in \oplus^{-1}(p)}	\boxtimes_{j=1}^l 
		\MF(\widehat{\mM}_Q^{\theta \sss}(v^{(j)})_{p^{(j)}}, \widehat{w}^{(j)}_{p^{(j)}})
		\ar[r]^-{\widehat{p}_{\lambda \ast}\widehat{q}_{\lambda}^{\ast}}
		& \MF(\widehat{\mM}_Q^{\theta \sss}(v)_p, \widehat{w}_{p}). 
	}	
\end{align}

	\section{Derived categories of Grassmannian flips}\label{sec:DGflip}
	In this section, we use categorified Hall products
	to refine the result of~\cite[Theorem~5.4.4]{BNFV} on 
	variation of derived categories under Grassmannian flips. 
	\subsection{Grassmannian flips}\label{subsec:Gflip}
	Let $V$ be a vector space with dimension $d$, 
	and $A$, $B$ be another vector spaces such that 
	\begin{align*}
		a \cneq \dim A, \ b \cneq \dim B, \ a \ge b. 
		\end{align*}
	We form the following quotient stack 
	\begin{align}\label{Gabd}
		\gG_{a, b}(d) \cneq 
		\left[ (\Hom(A, V) \oplus \Hom(V, B))/\GL(V)   \right]. 
		\end{align}
	\begin{rmk}\label{rmk:Gab}
	The stack $\gG_{a, b}(d)$ is the 
	$\C$-rigidified moduli stack of 
	representations of the 
	quiver $Q_{a, b}$
	of dimension vector $(1, d)$, 
	where the vertex set is $\{\infty, 1\}$, the
	number of arrows from $\infty$ to $1$ is $a$, 
	that from $1$ to $\infty$ is $b$, and 
	there are no self loops (see Subsection~\ref{subsec:catH}), e.g. the quiver $Q_{3, 2}$ is described 
	below: 
\begin{align}\label{cy4 quiver}
	Q_{3, 2}=
	\xymatrix{ \bullet_{\infty}  \ar@/^2.5pc/[rr]^{}   \ar@/^1.5pc/[rr]^{}  \ar@/^0.5pc/[rr]^{} 
		&& \bullet_{1}\ar@/^0.5pc/[ll]^{}  \ar@/^1.5pc/[ll]^{}   }   \end{align}
	\end{rmk}
Below we fix a basis of $V$, and take the 
maximal torus $T \subset \GL(V)$ to be consisting of 
diagonal matrices. 
For a one parameter subgroup $\lambda \colon \C \to T$, 
we use the following notation for the diagram of 
attracting loci (\ref{dia:attract}) 
\begin{align}\label{dia:quiver2}
	\xymatrix{
\gG_{a, b}(d)^{\lambda \ge 0} \ar[r]^-{p_{\lambda}} \ar[d]_-{q_{\lambda}} & \gG_{a, b}(d) \\
\gG_{a, b}(d)^{\lambda=0}. & 	
}
	\end{align}
We use the following determinant character 
\begin{align}\label{ch:det}
	\chi_0 \colon \GL(V) \to \C, \ g \mapsto \det(g),
	\end{align}
and often regard it as a line bundle on $\gG_{a, b}(d)$. 
There exist two GIT quotients 
with respect to $\chi_0^{\pm 1}$ 
given by open substacks
\begin{align*}
	G_{a, b}^{\pm}(d) \subset \gG_{a, b}(d). 
	\end{align*}
Here $\chi_0$-semistable 
locus $G_{a, b}^{+}(d)$ 
consists of $(\alpha, \beta) \in \Hom(A, V) \oplus \Hom(V, B)$
such that $\alpha \colon A \to V$ is surjective, 
and $\chi_0^{-1}$-semistable locus 
$G_{a, b}^-(d)$ consists of $(\alpha, \beta)$
such that $\beta \colon V \to B$ is injective. 
We have the following diagram 
\begin{align}\label{dia:Grass}
	\xymatrix{
	G_{a, b}^+(d) \ar@<-0.3ex>@{^{(}->}[r] \ar[rd] & \gG_{a, b}(d) \ar[d]
	 & G_{a, b}^-(d) \ar@<0.3ex>@{_{(}->}[l] \ar[ld] \\
	& G_{a, b}^0(d). &
}
		\end{align}
Here the middle vertical arrow is the good moduli space for 
$\gG_{a, b}(d)$. 
\begin{rmk}\label{rmk:ad}
	When $a\ge d$ and $b=0$, then 
$G_{a, 0}^-(d)=\emptyset$ and 
$G_{a, 0}^+(d)$ is the Grassmannian 
parameterizing surjections $A \twoheadrightarrow V$. 
If $a\ge b\ge d$, then 
$G_{a, b}^{\pm}(d) \to G_{a, b}^0(d)$ are birational 
and 
$G_{a, b}^+(d) \dashrightarrow G_{a, b}^-(d)$
is a flip $(a>b)$, flop $(a=b)$. 
\end{rmk}

We have the KN stratifications with respect to $\chi_0^{\pm 1}$
\begin{align*}
	\gG_{a, b}(d)=
	\sS_0^{\pm} \sqcup \sS_1^{\pm} \sqcup \cdots \sqcup \sS_{d-1}^{\pm} 
	\sqcup G_{a, b}^{\pm}(d)
	\end{align*}
where $\sS_i^{+}$ consists of $(\alpha, \beta)$
such that the image of $\alpha \colon A \to V$
has dimension $i$, 
and $\sS_i^{-}$ consists of 
$(\alpha, \beta)$ such that the 
kernel of $\beta \colon V \to B$ has dimension $d-i$. 
The associated one parameter subgroups 
$\lambda_i^{\pm} \colon \C \to T$ are taken as (see~\cite[Example~4.12]{MR3327537})
\begin{align}\label{lambdai}
	\lambda_i^{+}(t)=(\overbrace{1, \ldots, 1}^i, \overbrace{t^{-1}, \ldots, t^{-1}}^{d-i}), \ 
	\lambda_i^-(t)=(\overbrace{t, \ldots, t}^{d-i}, \overbrace{1, \ldots, 1}^i). 
	\end{align}

\subsection{Window subcategories for Grassmannian flips}\label{subsec:Gflip:2}
We fix a Borel subgroup $B \subset \GL(V)$ 
to be consisting of upper triangular matrices, 
and set roots of $B$ to be negative roots. 
Let $M=\mathbb{Z}^d$ be the character lattice for $\GL(V)$, 
and $M^+_{\mathbb{R}} \subset M_{\mathbb{R}}$ the dominant chamber. 
By the above choice of negative roots, we have 
\begin{align*}
	M_{\mathbb{R}}^+=\{(x_1, x_2, \ldots, x_d) \in \mathbb{R}^d : 
	x_1 \le x_2 \le \cdots \le x_d\}. 
	\end{align*} 
We set $M^+ \cneq M_{\mathbb{R}}^+ \cap M$. 
For $c \in \mathbb{Z}$, we set 
\begin{align}\label{Bcd}
\mathbb{B}_{c}(d) \cneq \{(x_1, x_2, \ldots, x_d) \in M^+ : 
0 \le x_i \le c-d\}. 
	\end{align}
Here $\mathbb{B}_c(d)=\emptyset$ if $c<d$. 
For $\chi \in \mathbb{B}_c(d)$, we assign the Young diagram 
whose number of boxes at the $i$-th row is $x_{d-i+1}$.
The above assignment 
identifies $\mathbb{B}_c(d)$ with the set of Young diagrams 
with height less than or equal to $d$, 
width less than or equal to $c-d$. 
For example, the following 
picture illustrates the case of 
$(3, 7, 7, 10, 15) \in \mathbb{B}_{c}(d)$ for $d=5$
and $c\ge 20$:
\begin{figure}[H]
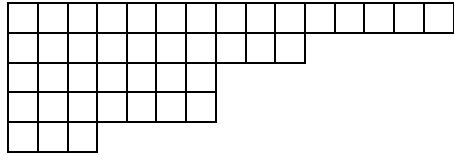
\caption{$(3, 7, 7, 10, 15) \in \mathbb{B}_{c}(d), d=5, c\ge 20$}
\begin{align*}
	\yng(15,10,7,7,3)
	\end{align*}
\end{figure}

We define the triangulated subcategory 
\begin{align}\label{window:Wc}
	\mathbb{W}_c(d) \subset D^b(\gG_{a, b}(d))
	\end{align}
to be the smallest thick triangulated subcategory 
which contains $V(\chi) \otimes \oO_{\gG_{a, b}(d)}$
for $\chi \in \mathbb{B}_c(d)$. 
Here $V(\chi)$ is the irreducible $\GL(V)$ representation with 
highest weight $\chi$, i.e. it is a Schur power of $V$ associated 
with the Young diagram corresponding to $\chi$. 
The following proposition is well-known (see~\cite[Proposition~3.6]{DoSe}), 
which give window subcategories for Grassmannian flips. 
We reprove it here using Theorem~\ref{thm:window}:
\begin{prop}\label{prop:WGD}
	The following composition functors are equivalences
	\begin{align}\label{compose:W}
		&\mathbb{W}_{b}(d) \subset D^b(\gG_{a, b}(d)) \twoheadrightarrow 
		D^b(G^-_{a, b}(d)), \\ 
		\notag &\mathbb{W}_{a}(d) \subset D^b(\gG_{a, b}(d)) \twoheadrightarrow 
		D^b(G_{a, b}^+(d)). 
		\end{align}
	\end{prop}
\begin{proof}
	We only prove the statement for $+$. 
	Let $\lambda_i^+$ be the one parameter 
	subgroup in (\ref{lambdai}). 
	Then $\eta_i^+$ given in (\ref{etai}) is 
	\begin{align*}
		\eta_i^+ &=\langle \lambda_i^+, 
		(\Hom(A, V)^{\vee} \oplus \Hom(V, B)^{\vee})^{\lambda_i^+>0}
		-\End(V)^{\lambda_i^+>0} \rangle \\
		&=(a-i)(d-i). 
		\end{align*}
	Let $\chi'=(x_1', \ldots, x_d')$ be a $T$-weight of $V(\chi)$ for $\chi \in \mathbb{B}_a(d)$. 
	Then we have $0\le x_j' \le a-d$ for $1\le j\le d$, 
so
\begin{align*}
	-\eta_i^+ <-(a-d)(d-i) \le 
	\langle \chi', \lambda_i^+ \rangle 
	=-\sum_{j=i+1}^d x_j' \le 0.
	\end{align*}
	Therefore by setting $m_i=-\eta_i^+ + \varepsilon$ for $0<\varepsilon \ll 1$
	and $l=\chi_0$ in (\ref{window:m}), we have
	\begin{align*}
		\mathbb{W}_a(d) \subset \mathbb{W}_{m_{\bullet}}^{\chi_0}(\gG_{a, b}(d)) \subset 
		D^b(\gG_{a, b}(d)). 
		\end{align*} 
	It follows that the second composition functor in (\ref{compose:W}) 
	is fully-faithful. 
	
	In order to show that it is essentially surjective, 
	note that 
	the projection to $\Hom(A, V)$
	defines a morphism 
	\begin{align}\label{Gr:ab}
		G_{a, b}^+(d) \to \mathrm{Gr}(a, d)
		\end{align}
	where $\mathrm{Gr}(a, d)$ is the Grassmannian 
	which parametrizes $d$-dimensional quotients of $A$. 
	By the above morphism, 
		$G_{a, b}^+(d)$ is the total space of a vector bundle 
	over $\mathrm{Gr}(a, d)$. 
	The objects $V(\chi) \otimes \oO_{\gG_{a, b}(d)}$
	for $\chi \in \mathbb{B}_a(d)$ restricted to 
	the zero section of (\ref{Gr:ab})
	forms Kapranov's exceptional collection~\cite{Kapranov}. 
	Since $D^b(G_{a, b}^+(d))$ is generated by pull-backs 
	of objects from $D^b(\mathrm{Gr}(a, d))$, the essentially 
	surjectivity of (\ref{compose:W}) holds. 
	\end{proof}

\subsection{Resolutions of categorified Hall products}
Let $d=d^{(1)}+\cdots +d^{(l)}+d^{(\infty)}$ be a decomposition of $d$. 
Note that from Subsection~\ref{subsec:catH}, we have the 
categorified Hall product
\begin{align*}
	\boxtimes_{j=1}^l D^b(B\GL(d^{(j)}))
	\boxtimes D^b(\gG_{a, b}(d^{(\infty)})) \to 
	D^b(\gG_{a, b}(d)). 
	\end{align*}
In particular by setting $d^{(1)}=1$ and $d^{(\infty)}=d-1$, 
we have the functor 
\begin{align}\label{def:ast}
	\ast \colon 
	D^b(B\C) \boxtimes D^b(\gG_{a, b}(d-1)) \to 
	D^b(\gG_{a, b}(d)). 
	\end{align}
It is explicitly given as follows. 
Let $\lambda \colon \C \to T$ be given by 
\begin{align}\label{lambdat}
	\lambda(t)=(t, 1, \ldots, 1). 
	\end{align}
Then we have the decomposition 
$V=V^{\lambda>0} \oplus V^{\lambda=0}$
where $V^{\lambda>0}$ is one dimensional. 
Then 
\begin{align*}
	\gG_{a, b}(d)^{\lambda=0}
	&=[B\GL(V^{\lambda>0})] \times 
	\left[(\Hom(A, V^{\lambda=0}) \oplus \Hom(V^{\lambda=0}, B))/\GL(V^{\lambda=0}))    \right] \\
	&=B\C \times \gG_{a, b}(d-1). 
	\end{align*}
The functor (\ref{def:ast}) is given by 
$p_{\lambda\ast}q_{\lambda}^{\ast}(-)$ in the diagram (\ref{dia:quiver2}). 
The stack $\gG_{a, b}(d)^{\lambda \ge 0}$ is the moduli stack of 
exact sequences of $Q_{a, b}$-representations 
\begin{align}\label{exact:lambda}
	0 \to \mathbb{V}^{\lambda>0} \to \mathbb{V} \to \mathbb{V}^{\lambda=0} \to 0
	\end{align}
such that $\mathbb{V}^{\lambda>0}$ has dimension vector $(0, 1)$. 
We will often use the following lemmas: 
\begin{lem}\label{lem:chi0}
	For $\eE_1 \in D^b(B\C)$ and $\eE_2 \in D^b(\gG_{a, b}(d-1))$, we have 
	\begin{align*}
		(\eE_1 \ast \eE_2) \otimes \chi_0^j =(\eE_1 \otimes \oO_{B\C}(j)) \ast (\eE_2 \otimes \chi_0^j).
		\end{align*}
	Here we have used the same symbol $\chi_0$ for the determinant 
	character of $\GL(V^{\lambda=0})$. 
	\end{lem}
\begin{proof}
	The lemma follows since $p_{\lambda}^{\ast}\chi_0=\oO_{B\C}(1) \boxtimes \chi_0$
	and the definition of the functor (\ref{def:ast}). 
	\end{proof}
\begin{lem}\label{lem:chi02}
	For $\eE \in \mathbb{W}_c(d)$ and $j\ge 0$, we have 
	$\eE \otimes \chi_0^j \in \mathbb{W}_{c+j}(d)$. 
	\end{lem}
\begin{proof}
	The lemma follows since $V(\chi) \otimes \chi_0^j=V(\chi')$
	where $\chi'=\chi+(j, j, \ldots, j)$. 
	\end{proof}
The following proposition is essentially proved in~\cite{DoSe, BNFV}, 
which we interpret in terms of categorified Hall products: 
\begin{prop}\emph{(\cite[Theorem~A.7]{DoSe}, \cite[Proposition~5.4.6]{BNFV})}\label{prop:resol}
For $\chi \in \mathbb{B}_{c}(d-1)$ with $c \ge b$, let $\delta$ be the 
corresponding Young diagram. 
Then the object 
$\oO_{B\C}\ast (V(\chi) \otimes \oO_{\gG_{a, b}(d-1)})$ is a sheaf 
which fits into an exact sequence 
\begin{align}\notag
	0 \to V(\chi_K) \otimes \oO_{\gG_{a, b}(d)}^{\oplus m_K} \to 
	\cdots &\to   V(\chi_1) \otimes \oO_{\gG_{a, b}(d)}^{\oplus m_1} \\
	\label{seq:Vchi} &\to   V(\chi) \otimes \oO_{\gG_{a, b}(d)} \to 
	\oO_{B\C}\ast (V(\chi) \otimes \oO_{\gG_{a, b}(d-1)}) \to 0. 
	\end{align}
Here $\chi \in \mathbb{B}_c(d-1)$ is regarded 
as an element of $\mathbb{B}_{c+1}(d)$ by 
$(x_2, \ldots, x_d) \mapsto (0, x_2, \ldots, x_d)$,
and each $\chi_i \in \mathbb{B}_{c+1}(d)$ in (\ref{seq:Vchi}) corresponds to a Young diagram $\delta_i$
obtained from $\delta$ by the 
following algorithm (see Figure~\ref{fig:delta}): 
\begin{itemize}
	\item The Young diagram $\delta_1$ is obtained from $\delta$ by adding boxes to the first column
	until it reaches to height $d$. 
	\item $\delta_i$ is obtained from $\delta_{i-1}$ by adding boxes to the $i$-th 
	column until its height is one more than the height of the $(i-1)$-th column of $\delta$. 
	\end{itemize}
Moreover $m_i=\dim \bigwedge^{s_i} B$ for $s_i=\lvert \delta_i \rvert -\lvert \delta \rvert$, 
and the sequence (\ref{seq:Vchi}) terminates when we reach a positive integer $K$ 
such that $s_{K+1} >b$. 
	\end{prop}
\begin{proof}
	Let 
	$\lambda$ be the one parameter subgroup (\ref{lambdat}). 
		Then we have 
	\begin{align*}
		(\Hom(A, V) \oplus \Hom(V, B))^{\lambda \ge 0}
		&= \Hom(A, V) \oplus \Hom(V^{\lambda=0}, B) \\
		&\cong \Hom(V^{\vee}, A^{\vee}) \oplus 
		\Hom(B^{\vee}, (V^{\lambda=0})^{\vee}).
		\end{align*}
	The parabolic subgroup 
	$\GL(V)^{\lambda \ge 0}$
	is the subgroup of $\GL(V)$ which preserves
	$V^{\lambda>0}\subset V$. 
	Therefore there is an isomorphism of quotient stacks 
	\begin{align*}
		&\left[(\Hom(A, V) \oplus \Hom(V, B))^{\lambda \ge 0}/
		\GL(V)^{\lambda \ge 0} \right] \\
		&\stackrel{\cong}{\to} \left[\Hom(V^{\vee}, A^{\vee}) \oplus 
		\Hom(B^{\vee}, (V^{\lambda=0})^{\vee}) \oplus \Hom^{\rm{inj}}((V^{\lambda=0})^{\vee}, V^{\vee})/
		\GL(V) \times \GL(V^{\lambda=0})\right]. 
		\end{align*}
	Here $\Hom^{\rm{inj}}((V^{\lambda=0})^{\vee}, V^{\vee})\subset \Hom((V^{\lambda=0})^{\vee}, V^{\vee})$
	is the subset consisting of injective homomorphisms. 
	The above isomorphism is induced by the 
	embedding into the direct summand $(V^{\lambda=0})^{\vee} 
	\hookrightarrow V^{\vee}$ together with the natural 
	inclusion $\GL(V)^{\lambda \ge 0} \hookrightarrow 
	\GL(V) \times \GL(V^{\lambda=0})$. 
	Under the above isomorphism, 
	the morphism 
	\begin{align*}
		p_{\lambda} \colon 
		\left[(\Hom(A, V) \oplus \Hom(V, B))^{\lambda \ge 0}/
		\GL(V)^{\lambda \ge 0} \right] \to \gG_{a, b}(d)
		\end{align*}
	from the diagram (\ref{dia:quiver}) is identified with the one 
	\begin{align*}
	&\left[\Hom(V^{\vee}, A^{\vee}) \oplus 
		\Hom(B^{\vee}, (V^{\lambda=0})^{\vee}) \oplus \Hom^{\rm{inj}}((V^{\lambda=0})^{\vee}, V^{\vee})/
		\GL(V) \times \GL(V^{\lambda=0})\right] \\
	&\stackrel{p_{\lambda}}{\to} \left[\Hom(V^{\vee}, A^{\vee}) \oplus \Hom(B^{\vee}, V^{\vee})/\GL(V)    \right]	
	\end{align*}
induced by the composition of maps. 
The above morphism is nothing but the one 
considered in~\cite[Theorem~A.7]{DoSe}, \cite[Proposition~5.4.6]{BNFV}. 
We then directly apply the computation of $p_{\lambda\ast}(-)$ 
for 
vector bundles given by Schur powers 
in~\cite[Theorem~A.7]{DoSe}, \cite[Proposition~5.4.6]{BNFV}
to obtain the resolution (\ref{seq:Vchi}). 
	
	We also check that each $\chi_i$ in (\ref{seq:Vchi}) is an element of $\mathbb{B}_{c+1}(d)$, i.e. 
	$\delta_i$ has at most height $d$ and width $c-d+1$. 
	It is obvious that $\delta_i$ has at most height $d$. 
	Let $\mu_j$ be the number of boxes of $\delta$ at the $j$-th column. 
	Then from the algorithm defining $\delta_i$, we have 
	\begin{align*}
		s_i&
=(d-\mu_{1})+(\mu_1+1-\mu_2)+\cdots (\mu_{i-1}+1-\mu_i) \\
&=d+i-1-\mu_i. 		
		\end{align*}
Since $\chi \in \mathbb{B}_c(d-1)$, we have 
$\mu_{c-d+2}=0$, 
so $s_{c-d+2}=c+1>b$. 
Therefore we have $K \le c-d+1$. Since 
$\delta$ has width at most $c-d+1$, 
it follows that $\delta_i$ also has width at most $c-d+1$. 
	\end{proof}

\begin{figure}
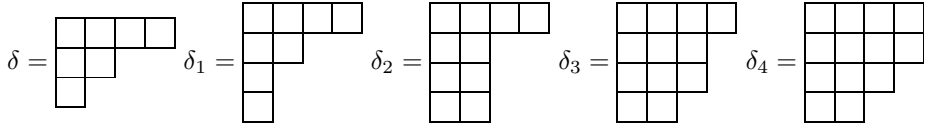
\caption{Algorithm for $\chi=(4, 2, 1)$, $d=4$, $c=b=7$}\label{fig:delta}
	\begin{align*}
\Yvcentermath1 \delta=\yng(4,2,1) \ 
\delta_1=\yng(4,2,1,1) \ 
\delta_2=\yng(4,2,2,2) \ 
\delta_3=\yng(4,3,3,2) \ 
\delta_4=\yng(4,4,3,2)
		\end{align*}
	\end{figure}

Using the above proposition, we have the following lemma: 
\begin{lem}\label{lem:inclu}
For $0\le j\le c-1$, we have 
\begin{align}\label{OB:W}
	\oO_{B\C}(j) \ast (\mathbb{W}_{c-1-j}(d-1) \otimes \chi_0^j) \subset 
	\mathbb{W}_c(d).
	\end{align}
	\end{lem}
\begin{proof}
We have 
	\begin{align*}
		\oO_{B\C}(j) \ast (\mathbb{W}_{c-1-j}(d-1) \otimes \chi_0^j)
		&= 	(\oO_{B\C} \ast \mathbb{W}_{c-1-j}(d-1)) \otimes \chi_0^j \\
		& \subset \mathbb{W}_{c-j}(d) \otimes \chi_0^j \\
		& \subset \mathbb{W}_c(d). 
	\end{align*}
Here we have used Lemma~\ref{lem:chi0} for the first 
identity, Proposition~\ref{prop:resol} for the first inclusion
and Lemma~\ref{lem:chi02} for the last inclusion. 
	\end{proof}
\subsection{Generation of window subcategories}
We show that for $c\ge b$ the category $\mathbb{W}_c(d)$ is generated 
by its subcategory 
$\mathbb{W}_b(d)$ and subcategories (\ref{OB:W}) for $0\le j\le c-b-1$.
We first prove the case of $c=b+1$, which is a
 variant of~\cite[Lemma~5.4.9]{BNFV}. 
 \begin{lem}\label{lem:genen0}
 	The subcategory $\mathbb{W}_{b+1}(d) \subset D^b(\gG_{a, b}(d))$
 	is generated by $\mathbb{W}_b(d)$ and $\oO_{B\C} \ast \mathbb{W}_{b}(d-1)$.
 	\end{lem}
 \begin{proof}
 	 	For $\chi \in \mathbb{B}_{b+1}(d)$, it 
 	is enough to show that $V(\chi) \otimes \oO_{\gG_{a, b}(d)}$ is generated by 
 	$\mathbb{W}_b(d)$ and $\oO_{B\C} \ast \mathbb{W}_{b}(d-1)$. 
 	Let $\delta$ be the Young diagram corresponding to $\chi$, and 
 	we denote by $\mu_j$ the number of boxes of $\delta$ at the $j$-th 
 	column. We may assume that the width of $\delta$ is exactly 
 	$b-d+1$, i.e. $\mu_j\ge 1$ for $1\le j\le b-d+1$ and $\mu_{b-d+2}=0$. 
 	
 	Suppose that the height of $\delta$ is exactly $d$, i.e. 
 	$\mu_1=d$. 
 	We define another Young diagram $\delta'$ 
 	whose number of boxes at the $j$-th column is $\mu_{j+1}-1$. 
 	Then the height of $\delta'$ is at most $d-1$, 
 	and the width of $\delta'$ is at most $b-d$ (see Figure~\ref{fig:delta2}). 
 	Let $\chi' \in \mathbb{B}_{b-1}(d-1)$ be the 
 	character corresponding to $\delta'$. 
 	As $\mathbb{B}_{b-1}(d-1) \subset \mathbb{B}_{b}(d-1)$, we 
 	apply Proposition~\ref{prop:resol} to obtain a resolution
 	 	\begin{align}\notag
 		0 \to V(\chi_K') \otimes \oO_{\gG_{a, b}(d)}^{\oplus m_K} \to 
 		\cdots &\to   V(\chi_1') \otimes \oO_{\gG_{a, b}(d)}^{\oplus m_1} \\
 		\label{seq:Vchi2} &\to   V(\chi') \otimes \oO_{\gG_{a, b}(d)} \to 
 		\oO_{B\C}\ast (V(\chi') \otimes \oO_{\gG_{a, b}(d-1)}) \to 0. 
 	\end{align}
  	for $\chi_i' \in \mathbb{B}_{b+1}(d)$. 
 	Note that we have  
 	\begin{align*}
 		\lvert \delta \rvert -\lvert \delta' \rvert
 		&=(d-\mu_2+1)+(\mu_2-\mu_3+1)+ \cdots + (\mu_{b-d}-\mu_{b-d+1}+1)+\mu_{b-d+1} \\
 		&= b.
 	\end{align*}
 	From the construction of $\delta'$, the Young diagram $\delta$ is 
 	reconstructed from $\delta'$ by 
 	the algorithm in Proposition~\ref{prop:resol}
 	at the $(b-d+1)$-th step. 
 	Therefore 
 	from the above identity, 
 	it follows that there are exactly $(b-d+1)$-terms of the resolution (\ref{seq:Vchi2}), i.e. 
 	$K=b-d+1$, and also $m_K=1$, $\chi_K'=\chi$. Moreover since the width of $\delta'$ is as most
 	$b-d$, we also have 
 	$\chi_i' \in \mathbb{B}_{b}(d)$ for $0\le i<b-d+1$. Therefore $V(\chi) \otimes \oO_{\gG_{a, b}(d)}$
 	is generated by objects in $\mathbb{W}_{b}(d)$ and
 	$\oO_{B\C} \ast (V(\chi') \otimes \oO_{\gG_{a, b}(d-1)}) \in \oO_{B\C} \ast \mathbb{W}_{b}(d-1)$. 
 	
 	Suppose that the height of $\delta$ is less than $d$. 
 	Then we have $\chi \in \mathbb{B}_{b}(d-1)$. 
 	By applying Proposition~\ref{prop:resol}, 
 	we see that $V(\chi) \otimes \oO_{\gG_{a, b}(d)}$ is generated by 
 	$\oO_{B\C} \ast  (V(\chi) \otimes \oO_{\gG_{a, b}(d-1)}) \in \oO_{B\C} \ast \mathbb{W}_{b}(d-1)$
 	and $V(\chi_i) \otimes \oO_{\gG_{a, b}(d)}$ for $\chi_i \in \mathbb{B}_{b+1}(d)$. 
 	By the algorithm in Proposition~\ref{prop:resol}, 
 	the Young diagram corresponding to $\chi_i$ has a full column, i.e. 
 	the height of $\chi_i$ is exactly $d$. 
 	Therefore by the above argument, each $V(\chi_i) \otimes \oO_{\gG_{a, b}(d)}$ is generated by 
 	$\mathbb{W}_b(d)$ and 
 	$\oO_{B\C} \ast \mathbb{W}_{b}(d-1)$. 
 	\end{proof}
 
 \begin{figure}
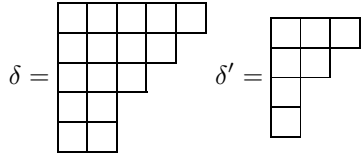
\caption{$\delta$ and $\delta'$ for 
 	$\chi=(2,2,3,4,5) \in \mathbb{B}_{10}(5)$}\label{fig:delta2}
 	\begin{align*}
 		\Yvcentermath1 \delta=\yng(5,4,3,2,2) \ 
 		\delta'=\yng(3,2,1,1)
 	\end{align*}
 \end{figure}
 
 We then show the generation for $\mathbb{W}_c(d)$: 
 
\begin{lem}\label{lem:genen}
	For $c\ge b$, the subcategory $\mathbb{W}_c(d) \subset D^b(\gG_{a, b}(d))$ is generated by 
	$\mathbb{W}_b(d)$ and $\oO_{B\C}(j) \ast (\mathbb{W}_{c-1-j}(d-1) \otimes \chi_0^j)$
	for $0\le j\le c-b-1$. 
	\end{lem}
\begin{proof}
The case of $c=b+1$ is proved in Lemma~\ref{lem:genen0}.
We prove the lemma for $c>b+1$ by the induction 
of $c$. 
For $\chi \in \mathbb{B}_{c}(d)$, suppose that the corresponding Young diagram $\delta$
has a full column. 
Let $\delta''$ be the Young diagram obtained by removing the first column, 
and $\chi''$ the corresponding character. 
Then $\chi'' \in \mathbb{B}_{c-1}(d)$, so 
by the induction hypothesis 
$V(\chi'') \otimes \oO_{\gG_{a, b}(d)}$ is generated by 
$\mathbb{W}_b(d)$ and $\oO_{B\C}(j) \ast (\mathbb{W}_{c-2-j}(d-1) \otimes \chi_0^j)$
for $0\le j \le c-b-2$. 
By taking the tensor product with $\chi_0$ and setting $j'=j+1$, 
we see that $V(\chi) \otimes \oO_{\gG_{a, b}(d)}$ is generated by 
$\mathbb{W}_b(d) \otimes \chi_0$ and $\oO_{B\C}(j') \ast (\mathbb{W}_{c-1-j'}(d-1) \otimes \chi_0^{j'})$
for $1\le j' \le c-b-1$.
Since $\mathbb{W}_b(d) \otimes \chi_0 \subset \mathbb{W}_{b+1}(d)$, and the latter is generated by 
$\mathbb{W}_b(d)$ and $\oO_{B\C} \ast \mathbb{W}_{b}(d-1) \subset \oO_{B\C} \ast \mathbb{W}_{c-1}(d-1)$ by Lemma~\ref{lem:genen0}, 
we have the desired generation for $V(\chi) \otimes \oO_{\gG_{a, b}(d)}$ when $\delta$ has a 
full column. 

If $\delta$ does not have a full column, then 
$\chi \in \mathbb{B}_{c-1}(d-1)$. 
By applying Proposition~\ref{prop:resol}, we see that 
$V(\chi) \otimes \oO_{\gG_{a, b}(d)}$ is generated by 
$\oO_{B\C} \ast (V(\chi) \otimes \oO_{\gG_{a, b}(d-1)})$
and $V(\chi_i) \otimes \oO_{\gG_{a, b}(d)}$ for $\chi_i \in \mathbb{B}_c(d)$. 
Since each Young diagram corresponding to $\chi_i$ has a full column, 
the desired generation of $V(\chi) \otimes \oO_{\gG_{a, b}(d)}$ is 
reduced to the case of the existence of full column which is proved above. 
	\end{proof}

The above generation result is stated in terms of 
iterated Hall products as follows: 
\begin{prop}\label{prop:Cgen}
	For $c\ge b$, the subcategory $\mathbb{W}_c(d) \subset 
	D^b(\gG_{a, b}(d))$ is generated by the 
	subcategories
	\begin{align}\label{subcat:product}
		\cC_{j_{\bullet}} \cneq 
		\oO_{B\C}(j_1) \ast \cdots \ast \oO_{B\C}(j_l)
		\ast (\mathbb{W}_b(d-l) \otimes \chi_0^{j_l})	\subset D^b(\gG_{a, b}(d))	
		\end{align}
	for $0\le l\le d$ and 
	$0\le j_1 \le \cdots \le j_l \le c-b-l$. 
	Here when $l=0$, the above subcategory is 
	set to be $\mathbb{W}_b(d)$. 
	\end{prop}
\begin{proof}
	We first show that (\ref{subcat:product}) 
	are subcategories of $\mathbb{W}_{c}(d)$
	by the induction on $c$. 
	By Lemma~\ref{lem:chi0}, 
	the subcategory (\ref{subcat:product}) is written as 
	\begin{align*}
		\oO_{B\C}(j_1) \ast \left(\left(
		\oO_{B\C}(j_2-j_1) \ast \cdots \ast \oO_{B\C}(j_l-j_1)
		\ast (\mathbb{W}_b((d-1)-(l-1)) \otimes \chi_0^{j_l-j_1})
		 \right)\otimes \chi_0^{j_1} \right) .
		\end{align*}
	Since $j_l-j_1 \le (c-1-j_1)-b-(l-1)$, by the induction hypothesis
	we have 
	\begin{align*}
		\oO_{B\C}(j_2-j_1) \ast \cdots \ast \oO_{B\C}(j_l-j_1)
		\ast \left(\mathbb{W}_b((d-1)-(l-1)) \otimes \chi_0^{j_l-j_1})
		\right) \in \mathbb{W}_{c-1-j_1}(d-1). 
		\end{align*}
	 Therefore (\ref{subcat:product}) is a subcategory of $\mathbb{W}_c(d)$
	 by Lemma~\ref{lem:inclu}. 
	
	We next show that $\mathbb{W}_c(d)$ is generated by 
	subcategories (\ref{subcat:product}) by the induction on $c$. 
	By Lemma~\ref{lem:genen}, 
	the subcategory $\mathbb{W}_c(d)$ is generated by 
	$\mathbb{W}_b(d)$ and $\oO_{B\C}(j) \ast (\mathbb{W}_{c-1-j}(d-1) \otimes \chi_0^j)$
	for $0\le j\le c-b-1$. 
	By the induction hypothesis and Lemma~\ref{lem:chi0}, 
	$\oO_{B\C}(j) \ast (\mathbb{W}_{c-1-j}(d-1) \otimes \chi_0^j)$
	is generated by 
	\begin{align*}
		&\oO_{B\C}(j) \ast \left(
		\left( \oO_{B\C}(j_1) \ast \cdots \ast \oO_{B\C}(j_{l'})
		\ast (\mathbb{W}_b(d-1-l') \otimes \chi_0^{j_{l'}})   \right) \otimes \chi_0^j
		\right) \\
		&=\oO_{B\C}(j) \ast \oO_{B\C}(j+j_1) \ast \cdots \ast \oO_{B\C}(j+j_{l'})
		\ast (\mathbb{W}_b(d-1-l') \otimes \chi_0^{j+j_{l'}})
			\end{align*}
		for $0\le l'\le d-1$ and $0\le j_1 \le \cdots \le j_{l'}
		 \le (c-1-j)-b-l'$. 
		Since $j+j_{l'} \le c-b-(l'+1)$, the above subcategory is 
		of the form (\ref{subcat:product}) for $l=l'+1$. 
		Therefore we obtain the desired generation. 
	\end{proof}

\begin{rmk}\label{rmk:Cj}
	Let $F_j(-) \cneq (\oO_{B\C} \ast (-)) \otimes \chi_0^j
	=\oO_{B\C}(j) \ast ((-) \otimes \chi_0^j)$. Then 
	the repeated use of Lemma~\ref{lem:chi0} implies that 
	\begin{align*}
		\cC_{j_{\bullet}}=
		F_{j_1} \circ F_{j_2-j_1} \circ \cdots \circ 
		F_{j_l-j_{l-1}}(\mathbb{W}_{b}(d-l)). 
		\end{align*}
	Similarly for an intermediate step, we have 
	\begin{align*}
		\oO_{B\C}(j_i) \ast \cdots \ast \oO_{B\C}(j_l)
		\ast (\mathbb{W}_b(d-l) \otimes \chi_0^{j_l}) 
		=F_{j_i} \circ F_{j_{i+1}-j_i} \circ \cdots \circ F_{j_{l}-j_{l-1}}(\mathbb{W}_b(d-l)).
		\end{align*}
	By the repeated use of Lemma~\ref{lem:inclu}, the above 
	category is a subcategory of $\mathbb{W}_{b+l-i+1+j_l}(d-i+1)$. 
	\end{rmk}

\subsection{Semiorthogonal decompositions under Grassmannian flips}
We show that the subcategories in Proposition~\ref{prop:Cgen} form a 
semiorthogonal decomposition. 
We prepare some lemmas:

\begin{lem}\label{lem:ortho1}
For any $\chi \in \mathbb{B}_{b}(d)$ and $\chi' \in \mathbb{B}_{c}(d-1)$ for some 
$c\ge 0$, we have the vanishing for $j\ge 0$
\begin{align}\label{Hom:Gab}
	\Hom_{\gG_{a, b}(d)}(\oO_{B\C}(j) \ast (V(\chi') \otimes \oO_{\gG_{a, b}(d-1)}), V(\chi)
	\otimes \oO_{\gG_{a, b}(d)})=0. 	
	\end{align}	
	\end{lem}
\begin{proof}
	Let $\lambda \colon \C \to T$
	be the one parameter subgroup 
	given by (\ref{lambdat}). 
	Using the notation of the diagram (\ref{dia:quiver2}), 
	the LHS 
	of (\ref{Hom:Gab}) is 
	\begin{align}\notag
		&\Hom(p_{\lambda \ast}q_{\lambda}^{\ast}(\oO_{B\C}(j) \boxtimes
		 (V(\chi') \otimes \oO_{\gG_{a, b}(d-1)}), V(\chi) \otimes \oO_{\gG_{a, b}(d)}) \\
		 \label{Hom:qp}
		 &=\Hom(q_{\lambda}^{\ast}(\oO_{B\C}(j) \boxtimes
		 (V(\chi') \otimes \oO_{\gG_{a, b}(d-1)}), 
		 p_{\lambda}^{!}(V(\chi) \otimes \oO_{\gG_{a, b}(d)})). 
		\end{align}
	We have the formula for $p_{\lambda}^!$ (cf.~\cite[Section~A.1]{DoSe}, \cite[(5.8)]{BNFV})
	\begin{align*}
		p_{\lambda}^{!}(-)=(-) \otimes (\det V^{\lambda>0})^{d-b-1}
		\otimes (\det V^{\lambda=0})^{-1}[d-b-1]. 
		\end{align*}
	Since $\chi \in \mathbb{B}_b(d)$
	and it is a highest weight of $V(\chi)$, 
	any $T$-weight $\chi''=(x_1'', \ldots, x_d'')$ 
	of $V(\chi)$ satisfies $x_i'' \le b-d$. 
	Therefore any $T$-weight of 
	$V(\chi)  \otimes (\det V^{\lambda>0})^{d-b-1}
	\otimes (\det V^{\lambda=0})^{-1}$
	pair negatively with $\lambda$. 
	On the other hand, a pairing of $\lambda$ with any $T$-weight of 
	the $\GL(V)^{\lambda=0}$-representation 
	 $(\det V^{\lambda >0})^j \boxtimes V(\chi')$
	 is $j\ge 0$. 
	 Therefore we have the vanishing of (\ref{Hom:qp}) by Lemma~\ref{lem:vanish}. 
		\end{proof}

\begin{lem}\label{lem:FF}
	For $\chi, \chi' \in \mathbb{B}_c(d-1)$ for some $c \ge 0$,
	we have the vanishing for $j>j'$
	\begin{align}\label{vanish:j}
	\Hom_{\gG_{a, b}(d)}(\oO_{B\C}(j) \ast (V(\chi) \otimes \chi_0^j \otimes 
	 \oO_{\gG_{a, b}(d-1)}), \oO_{B\C}(j') \ast (V(\chi) \otimes \chi_0^{j'} \otimes 
	 \oO_{\gG_{a, b}(d-1)}))=0.  		
		\end{align}
		\end{lem}
	\begin{proof}
		By Lemma~\ref{lem:chi0}, we may assume that $j'=0$. 
		We use the notation in the proof of Lemma~\ref{lem:ortho1}. 
		Using Lemma~\ref{lem:chi0} and the adjunction, the 
		LHS of (\ref{vanish:j}) is 
		\begin{align*}
			&\Hom((p_{\lambda\ast}q_{\lambda}^{\ast}(\oO_{B\C} 
			\boxtimes (V(\chi) \otimes \oO_{\gG_{a, b}(d-1)}))) \otimes \chi_0^j, 
			p_{\lambda\ast}q_{\lambda}^{\ast}(\oO_{B\C} 
			\boxtimes (V(\chi') \otimes \oO_{\gG_{a, b}(d-1)}))) \\
			&\cong
			\Hom(p_{\lambda}^{\ast}((p_{\lambda\ast}q_{\lambda}^{\ast}(\oO_{B\C} 
			\boxtimes (V(\chi) \otimes \oO_{\gG_{a, b}(d-1)}))) \otimes \chi_0^j), 
			q_{\lambda}^{\ast}(\oO_{B\C} 
			\boxtimes (V(\chi') \otimes \oO_{\gG_{a, b}(d-1)}))). 
			\end{align*}
		By Proposition~\ref{prop:resol},
		the object 
		\begin{align*}
			p_{\lambda\ast}q_{\lambda}^{\ast}(\oO_{B\C} 
		\boxtimes (V(\chi) \otimes \oO_{\gG_{a, b}(d-1)}))
		\in D^b(\gG_{a, b}(d))
		\end{align*}
		is resolved by vector bundles of the form 
		$V(\chi'') \otimes \oO_{\gG_{a, b}(d)}$
		where $\chi''$ is either $\chi''=\chi$, 
		or $\chi'' \in \mathbb{B}_{c+1}(d)$ whose corresponding Young 
		diagram has a full column. 
		In the latter case, any $T$-weight of 
		$V(\chi'')$ pair positively 
		with $\lambda$. 
		Therefore in both cases, any $T$-weight 
		of $V(\chi'') \otimes \chi_0^j$ for $j>0$
		pair positively with $\lambda$. 
		On the other hand the $\lambda$-weight 
		of $\oO_{B\C} 
		\boxtimes (V(\chi') \otimes \oO_{\gG_{a, b}(d-1)})$ is 
		zero so the desired vanishing (\ref{vanish:j})
		follows from Lemma~\ref{lem:vanish}. 
		\end{proof}

	\begin{lem}\label{lem:FF2}
		In the situation of Lemma~\ref{lem:FF}, 
		we have the isomorphism for $j \in \mathbb{Z}$
		\begin{align}\label{isom:j}
			&\Hom_{\gG_{a, b}(d-1)}(V(\chi) \otimes \chi_0^j \otimes \oO_{\gG_{a, b}(d-1)}, 
			V(\chi') \otimes \chi_0^j \otimes \oO_{\gG_{a, b}(d-1)}) \\
			\notag &\stackrel{\cong}{\to}
			\Hom_{\gG_{a, b}(d)}(\oO_{B\C}(j) \ast (V(\chi)\otimes \chi_0^j \otimes \oO_{\gG_{a, b}(d-1)}), 
			\oO_{B\C}(j) \ast (V(\chi')\otimes \chi_0^j \otimes \oO_{\gG_{a, b}(d-1)})). 
			\end{align}		
		\end{lem}
	\begin{proof}
			By Lemma~\ref{lem:chi0}, we may assume that $j=0$. 
			Let $\chi''$ be a weight which appeared in the proof of Lemma~\ref{lem:FF}. 
			Note that we observed that 
		any $T$-weight of 
		$V(\chi'')$ pair positively with $\lambda$
		except $\chi''=\chi$. 
		Therefore by Lemma~\ref{lem:vanish} (i), the RHS of (\ref{isom:j}) is isomorphic to 
		\begin{align}\label{isom:j2}
			\Hom(p_{\lambda}^{\ast}(V(\chi) \otimes \oO_{\gG_{a, b}(d)}), 
			q_{\lambda}^{\ast}(\oO_{B\C} 
			\boxtimes (V(\chi') \otimes \oO_{\gG_{a, b}(d-1)}))).
			\end{align}
		Since $\gG_{a, b}(d)^{\lambda \ge 0}$
		parametrizes exact sequences (\ref{exact:lambda}), 
		the object $p_{\lambda}^{\ast}(V(\chi) \otimes \oO_{\gG_{a, b}(d)})$
		admits a filtration whose associated graded is of the form 
		$q_{\lambda}^{\ast}(\oO_{B\C}(j) 
		\boxtimes (V(\chi''') \otimes \oO_{\gG_{a, b}(d-1)}))$
		for $j \ge 0$ and $\chi''' \in \mathbb{B}_c(d-1)$, and $j=0$ if and only if $\chi'''=\chi$. 
		Therefore by Lemma~\ref{lem:vanish} (i), (ii), the above (\ref{isom:j2}) is isomorphic to 
		\begin{align}\label{isom:j3}
			&	\Hom(q_{\lambda}^{\ast}(\oO_{B\C} 
				\boxtimes (V(\chi) \otimes \oO_{\gG_{a, b}(d-1)})), 
			q_{\lambda}^{\ast}(\oO_{B\C} 
			\boxtimes (V(\chi') \otimes \oO_{\gG_{a, b}(d-1)}))) \\
		\notag	&\cong \Hom_{\gG_{a, b}(d-1)}(V(\chi) \otimes \oO_{\gG_{a, b}(d-1)}, 
			V(\chi') \otimes \oO_{\gG_{a, b}(d-1)}).
			\end{align}
		\end{proof}
	
	In order to state the order of semiorthogonal decompositions, 
	we take a 
	lexicographical order on $\mathbb{Z}^d$, i.e. 
	for $m_{\bullet}=(m_1, \ldots, m_d) \in \mathbb{Z}^d$
	and $m_{\bullet}'=(m_1', \ldots, m_d') \in \mathbb{Z}^d$, 
	we write $m_{\bullet} \succ m_{\bullet}'$ if 
	$m_i=m_i'$ for $1\le i\le k$ for some $k \ge 0$ and 
	$m_{k+1}>m_{k+1}'$. 
	\begin{defi}\label{defi:orderj}	
		For
	$j_{\bullet}=(j_1, j_2, \ldots, j_l)$ 
	and 
	$j_{\bullet}'=(j_1', j_2' \ldots, j_l')$
	with $l, l' \le d$, 
	we 
	define $j_{\bullet} \succ j_{\bullet}'$ if 
	we have $\widetilde{j}_{\bullet} \succ \widetilde{j}_{\bullet}'$, 
	where 
	$\widetilde{j}_{\bullet}$ is defined by 
		\begin{align}\label{jtilde}
		\widetilde{j}_{\bullet}=(j_1, j_2, \ldots, j_l, -1, \ldots, -1) \in \mathbb{Z}^d. 
		\end{align}
	\end{defi}
The following proposition shows the semiorthogonality of 
subcategories (\ref{subcat:product})
with respect to the above order. 
	
	\begin{prop}\label{prop:sod}
		For $j_{\bullet}=(j_1, j_2, \ldots, j_l)$
		with $0\le l\le d$ and $0\le j_1 \le j_2 \le \cdots \le j_l$, 
		and $j_{\bullet}'=(j_1', j_2', \ldots, j_{l'}')$
		with $0\le l'\le d$ and $0\le j_1' \le j_2' \le \cdots \le j_{l'}'$, 
		suppose that $j_{\bullet} \succ j_{\bullet}'$. 
		Then we have 
		$\Hom(\cC_{j_{\bullet}}, \cC_{j_{\bullet}'})=0$. 
		Here $\cC_{\bullet}$ is defined in (\ref{subcat:product}). 
		\end{prop}
	\begin{proof}
		Let us take $P \in \mathbb{W}_b(d-l)$ and $P' \in \mathbb{W}_b(d-l')$. 
		We need to show the vanishing of 
		\begin{align}
			\label{vanish:AB}
			\Hom(\oO_{B\C}(j_1) \ast \cdots \ast \oO_{B\C}(j_l) \ast (P \otimes \chi_0^{j_l}), \oO_{B\C}(j_1) \ast \cdots \ast \oO_{B\C}(j_{l'}) \ast P' \otimes \chi_0^{j_l'})).  
			\end{align}
		We note that, by Remark~\ref{rmk:Cj}, for each $i \le  l, l'$ the objects
		\begin{align}\label{note:W}
			\oO_{B\C}(j_{i+1}) \ast \cdots \ast \oO_{B\C}(j_l) \ast 
			(P \otimes \chi_0^{j_l}), \ 
			\oO_{B\C}(j_{i+1}) \ast \cdots \ast \oO_{B\C}(j_{l'}) \ast 
			(P' \otimes \chi_0^{j_l'})
		\end{align}
		are objects in $\mathbb{W}_{c'}(d-i)$ for some $c' \ge 0$. 
				
			From $j_{\bullet}\succ j_{\bullet}'$, 
			we have two cases:
		\begin{enumerate}
			\item $l>l'$ and $j_i=j_i'$ for $1\le i\le l'$; 
			\item there is $1\le m<l, l'$ such that 
			$j_i=j_i'$ for $1\le i\le m$ and $j_{m+1}>j_{m+1}'$. 
			\end{enumerate}		
				In the first case, 
			we have  		
		\begin{align*}
		(\ref{vanish:AB})&=	\Hom(\oO_{B\C}(j_1) \ast \cdots \ast \oO_{B\C}(j_{l'}) \ast (\oO_{B\C}(j_{l'+1})\ast \cdots \ast \oO_{B\C}(j_l)
			\ast (P \otimes \chi_0^{j_l})), \\
			& \qquad \qquad \qquad \qquad \qquad \qquad \qquad \oO_{B\C}(j_1) \ast \cdots \ast \oO_{B\C}(j_{l'}) \ast (P'
			\otimes \chi_0^{j_{l'}})) \\
			& \cong 	\Hom(\oO_{B\C}(j_{l'+1})\ast \cdots \ast \oO_{B\C}(j_l)
			\ast (P \otimes \chi_0^{j_l}),  P'
			\otimes \chi_0^{j_{l'}}) \\
			&\cong 0. 
			\end{align*}
		Here the first isomorphism 
		follows from the 
		repeated use of Lemma~\ref{lem:FF2} (noting that (\ref{note:W}) 
		are objects in $\mathbb{W}_{c'}(d-i)$), 
		and the second isomorphism follows from 
		Lemma~\ref{lem:ortho1}, 
				In the second case, 
		a similar argument as above shows that  
		\begin{align*}
			(\ref{vanish:AB})&= \Hom(\oO_{B\C}(j_1) \ast \cdots \ast \oO_{B\C}(j_{m}) \ast (\oO_{B\C}(j_{m+1})\ast \cdots \ast \oO_{B\C}(j_l)
			\ast (P \otimes \chi_0^{j_l})), \\
			& \qquad \qquad 	\oO_{B\C}(j_1) \ast \cdots \ast \oO_{B\C}(j_{m}) \ast (\oO_{B\C}(j'_{m+1})\ast \cdots \ast \oO_{B\C}(j'_{l'})
			\ast (P' \otimes \chi_0^{j'_{l'}})) \\
			& \cong 	\Hom(\oO_{B\C}(j_{m+1})\ast \cdots \ast \oO_{B\C}(j_l)
			\ast (P \otimes \chi_0^{j_l}),  \oO_{B\C}(j_{m+1}')\ast \cdots \ast \oO_{B\C}(j'_{l'})
			\ast (P' \otimes \chi_0^{j'_{l'}})) \\
			&\cong 0. 
			\end{align*}
			Here the first isomorphism follows from 
	the	repeated use of Lemma~\ref{lem:FF2}, 
		and the second 
		isomorphism follows from Lemma~\ref{lem:FF}. 
		\end{proof}
	
	The following is the main result in this section: 
	\begin{thm}\label{thm:sod}
		For $c\ge b$, there exists a semi-orthogonal decomposition 
		\begin{align*}
			\mathbb{W}_c(d)=\left \langle  \cC_{j_{\bullet}} : 0\le j\le d, 
			j_{\bullet}=(0\le j_1 \le \cdots \le j_l \le c-b-l)  \right \rangle
			\end{align*}
		where 
			$\Hom(\cC_{j_{\bullet}}, \cC_{j_{\bullet}'})=0$	
			for 
			$j_{\bullet} \succ j_{\bullet}'$, and 
			for each $j_{\bullet}$ we have an equivalence 
			$\mathbb{W}_{b}(d-l) \stackrel{\sim}{\to} \cC_{j_{\bullet}}$. 
		\end{thm}
	\begin{proof}
		The generation of $\mathbb{W}_c(d)$ by $\cC_{j_{\bullet}}$ is proved in 
		Proposition~\ref{prop:Cgen}, and the semi-orthogonality 
		is proved in Proposition~\ref{prop:sod}. 
		The equivalence $\mathbb{W}_{b}(d-l) \stackrel{\sim}{\to} \cC_{j_{\bullet}}$
		follows from repeated use of Lemma~\ref{lem:FF2}. 
		\end{proof}
	
	By applying the above theorem to $c=a$ and using Proposition~\ref{compose:W}, 
	we obtain the following corollary 
	which relates derived categories under Grassmannian flips: 
	\begin{cor}\label{cor:DG}
		There exists a a semiorthogonal decomposition 
		\begin{align*}
			D^b(G_{a, b}^+(d))=
			\left \langle D^b(G_{a, b}^-(d-l))_{j_1, \ldots, j_l} : 
			0\le l\le d, 0\le j_1 \le \cdots \le j_l \le a-b-l \right \rangle. 
			\end{align*}
		Here $D^b(G_{a, b}^-(d-l))_{j_1, \ldots, j_l}$ is a copy of $D^b(G_{a, b}^-(d-l))$. 
		\end{cor}
	
	\begin{rmk}\label{rmk:exceptional}
		When $b=0$,
		from Remark~\ref{rmk:ad}
		the semiorthogonal decomposition in Corollary~\ref{cor:DG}
		is 
		\begin{align*}
			D^b(G_{a, 0}^+(d))=
				\left \langle D^b(\Spec \mathbb{C})_{j_1, \ldots, j_d} : 
			0\le j_1 \le \cdots \le j_d \le a-d \right \rangle. 
			\end{align*}
		Each factor $D^b(\Spec \mathbb{C})_{j_1, \ldots, j_d}$ is generated by a 
		vector bundle which forms Kapranov's exceptional collection~\cite{Kapranov}
		of the Grassmannian $G_{a, 0}^+(d)$. 
		\end{rmk}
	
	\begin{rmk}\label{rmk:d=1}
		When $d=1$, the birational map 
		$G_{a, b}^+(1) \dashrightarrow 
		G_{a, b}^-(1)$ is a standard toric flip. 
		In this case, the semiorthogonal decomposition in Corollary~\ref{cor:DG}
		is 
		\begin{align*}
			D^b(G_{a, b}^+(1))
			=\langle D^b(G_{a, b}^-(1)), D^b(\mathrm{pt})_{(0)}, 
			\ldots, D^b(\mathrm{pt})_{(a-b-1)} \rangle. 
			\end{align*}
		The above semiorthogonal decomposition is a (mutation of) 
		well-known 
		semiorthogonal decomposition for a standard 
		flip (see~\cite[Example~8.8 (2)]{KawBir}). 
		\end{rmk}
	
	\begin{rmk}
		For a fixed $(a, b, l)$, the set of 
		sequences of integers $(j_1, \ldots, j_l)$
		satisfying $0\le j_1 \le \cdots \le j_l \le a-b-l$
		consists of $\binom{a-b}{l}$ elements. 
		Therefore Corollary~\ref{cor:DG} implies (\ref{sod:grass}). 
		The same applies to Corollary~\ref{cor:sod}, Corollary~\ref{cor:sod2.5} below
		so that they imply (\ref{intro:catWCF}), (\ref{intro:PT:sod}) respectively. 
		\end{rmk}
	
	\subsection{Applications to categories of factorizations}
	We will use the following variant of Corollary~\ref{cor:DG}. 
	Let $Z$ be a smooth scheme with a closed point $z \in Z$. 
	Let us take the formal completion of $G_{a, b}^0(d) \times Z$ 
	where $G_{a, b}^0(d)$ is the good moduli space for $\gG_{a, b}(d)$, 
	\begin{align*}
		\widehat{G}^0_{a, b}(d)_Z \cneq \Spec \widehat{\oO}_{G^0_{a, b}(d) \times Z, (0, z)}. 
		\end{align*}
	We also take a regular function $w$ on it 
	\begin{align*} 
		w \colon \widehat{G}^0_{a, b}(d)_Z \to \mathbb{A}^1, \ 
	w(0, z)=0.
	\end{align*}
	By taking the product of the diagram (\ref{dia:Grass}) with $Z$
	and 
	pulling it back via 
	$\widehat{G}^0_{a, b}(d)_Z \to G^0_{a, b}(d) \times Z$, 
	we obtain the diagram  
	\begin{align}\label{dia:Grass2}
		\xymatrix{
			\widehat{G}_{a, b}^+(d)_Z \ar@<-0.3ex>@{^{(}->}[r] \ar[rd]
			 \ar@/_1.5pc/[rdd]_-{w} & \widehat{\gG}_{a, b}(d)_Z \ar[d]
			& \widehat{G}_{a, b}^-(d)_Z \ar@<0.3ex>@{_{(}->}[l] \ar[ld] \ar@/^1.5pc/[ldd]^-{w} \\
			& \widehat{G}_{a, b}^0(d)_Z \ar[d]_-{w} & \\
			& \mathbb{A}^1. &
		}
	\end{align}
	Similarly to (\ref{def:ast}), we have the categorified Hall product 
	for formal fibers
	(see Subsection~\ref{subsec:bchange})
	\begin{align*}
		\ast \colon \MF(B\C, 0) \boxtimes \MF(\widehat{\gG}_{a, b}(d-1)_Z, w)
		\to \MF(\widehat{\gG}_{a, b}(d)_Z, w). 
		\end{align*}
		The subcategory 
	\begin{align*}
		\widehat{\mathbb{W}}_c(d) \subset \MF(\widehat{\gG}_{a, b}(d)_Z, w)
		\end{align*}
	is also defined
	similarly to (\ref{window:Wc})
	 to be the smallest thick 
	triangulated subcategory which contains factorizations 
	with entries $V(\chi) \otimes \oO$
	for $\chi \in \mathbb{B}_c(d)$. 
		Note that we have the decomposition (\ref{decom:Y})
		\begin{align*}
		\MF(B\C, 0)=\bigoplus_{j \in \mathbb{Z}} \MF(\Spec \mathbb{C}, 0)_j
	\end{align*}
	such that
	$\MF(\Spec \mathbb{C}, 0)_j$ is equivalent to $\MF(\Spec \mathbb{C}, 0)$. 
	We then define 
		\begin{align}\label{subcat:product2}
		\widehat{\cC}_{j_{\bullet}} \cneq 
	\MF(\Spec \mathbb{C}, 0)_{j_1} \ast \cdots \ast 	\MF(\Spec \mathbb{C}, 0)_{j_l}
		\ast (\widehat{\mathbb{W}}_b(d-l) \otimes \chi_0^{j_l})	\subset
		\MF(\widehat{\gG}_{a, b}(d)_Z, w))	
	\end{align}
	for $0\le l\le d$ and 
	$0\le j_1 \le \cdots \le j_l \le c-b-l$. 
		We have the following variant of 
	Theorem~\ref{thm:sod}: 
	\begin{cor}\label{cor:MF}
			For $c\ge b$, there exists a semi-orthogonal decomposition 
		\begin{align*}
			\widehat{\mathbb{W}}_c(d)=\left \langle  \widehat{\cC}_{j_{\bullet}} : 0\le j\le d, 
			j_{\bullet}=(0\le j_1 \le \cdots \le j_l \le c-b-l)  \right \rangle
		\end{align*}
		where 
		$\Hom(\widehat{\cC}_{j_{\bullet}}, \widehat{\cC}_{j_{\bullet}'})=0$	
		for 
		$j_{\bullet} \succ j_{\bullet}'$, and 
		for each $j_{\bullet}$ we have an equivalence 
			$\widehat{\mathbb{W}}_{b}(d-l) \stackrel{\sim}{\to} \widehat{\cC}_{j_{\bullet}}$. 
		\end{cor}
	\begin{proof}
		The argument of Theorem~\ref{thm:sod}
		implies an analogous semiorthogonal decomposition for $D^b(\widehat{\gG}_{a, b}(d)_Z)$. 
		Then it is well-known that 
		the above semiorthogonal decomposition
		induces the one for categories of factorizations 
		(cf.~\cite[Lemma~1.17, 1.18]{HPHodge}, \cite[Proposition~1.10]{OrLG}, 
		\cite[Proposition~2.7]{Tudor}, \cite[Proposition~2.1]{Tudor1.5}). 
		\end{proof}
	
	\section{Categorical Donaldson-Thomas theory for the resolved conifold}
	In this section, we use the result in the previous section to prove Theorem~\ref{intro:thm2}. 
	\subsection{Geometry and algebras for the resolved conifold}
		Let $X$ be the resolved conifold 
	\begin{align*}
		X \cneq \mathrm{Tot}_{\mathbb{P}^1}(\oO_{\mathbb{P}}(-1)^{\oplus 2}). 
		\end{align*}
	Here we recall some well-known geometry and algebras
	for the resolved conifold (see~\cite{MR2057015, NN} for details). 
	There is a birational contraction 
	\begin{align*}
		f \colon X \to Y \cneq \{xy+zw=0\} \subset \mathbb{C}^4
		\end{align*}
	which contracts 
	the zero section $C=\mathbb{P}^1 \subset X$
	to the conifold singularity $0 \in Y$. 
		Let $\eE \cneq \oO_X \oplus \oO_X(1)$, and 
	$A \cneq \End(\eE)$. 
	Then there is an equivalence by Van den Bergh~\cite{MR2057015}
	\begin{align}\label{equiv:Phi}
		\Phi \cneq \RHom(\eE, -) \colon 
		D^b(X) \stackrel{\sim}{\to} D^b(\modu A). 
		\end{align}
	Here $\modu A$ is the abelian category 
	 finitely generated right $A$-modules. 
	 The non-commutative algebra $A$ is isomorphic to
	 the path algebra associated with a quiver 
	 with super-potential $(Q, W)$, given below
	 	\[
	 Q=
	 \begin{tikzcd}
	 	 	\bullet_{0}
	 	\arrow[rr, bend left,  "a_2"]
	 	\arrow[rr,bend left=70,  "a_1"]
	 	& &
	 	\bullet_1
	 	\arrow[ll, bend left, "b_1" ]
	 	\arrow[ll, bend left=70, "b_2"]
	 \end{tikzcd}
 \quad 
	 W=a_1 b_1 a_2 b_2-a_1 b_2 a_2 b_1.	
	 \]

   The equivalence (\ref{equiv:Phi}) restricts to the equivalences of
   abelian subcategories
   \begin{align*}
   	\Phi \colon \PPer(X/Y) \stackrel{\sim}{\to} \modu A, \ 
   	\Phi \colon \PPer_{c}(X/Y) \stackrel{\sim}{\to} \modu_{\rm{fd}}(A). 
   	\end{align*}
   Here $\PPer(X/Y)$ is the abelian category of 
   Bridgeland's perverse coherent sheaves~\cite{Br1}, 
   explicitly given by 
   \begin{align*}
   	\PPer(X/Y)=\left\{
   		E \in D^b(X):  \begin{array}{ll}
   	\hH^i(E)=0 \mbox{ for } i \neq -1, 0, R^1 f_{\ast}\hH^0(E)=f_{\ast}\hH^{-1}(E)=0 \\
   	\Hom(\hH^0(E), \oO_C(-1))=0
   	\end{array}
   	  \right\}. 
   	\end{align*}
   The subcategory $\PPer_c(X/Y) \subset \PPer(X/Y)$ consists of compactly 
   supported objects, and 
   $\modu_{\rm{fd}}(A) \subset \modu(A)$ consists of finite dimensional 
   $A$-modules. 
   The simple $(Q, W)$-representations corresponding to 
   the vertex $\{0, 1\}$
   are given by 
 \begin{align*}
 	\{\oO_C, \oO_C(-1)[1]\} \subset \PPer_c(X/Y). 
  	\end{align*}
 
   An object $F \in \PPer_c(X/Y)$ is supported on $C$ 
 or zero dimensional subscheme in $X$. 
 For $F \in \PPer_c(X/Y)$, 
 we set 
 \begin{align*}
 	\cl(F) \cneq (\beta, n) \in \mathbb{Z}^{\oplus 2}, \ [F]=\beta[C], \chi(F)=n
 \end{align*}
 where $[F]$ is the fundamental one cycle of $F$. 
 Under the equivalence $\Phi$, 
 an object $F \in \PPer_c(X/Y)$ with $\cl(F)=(\beta, n)$
 corresponds to a $(Q, W)$-representation with 
 dimension vector $(n, n-\beta)$. 
 
  Following~\cite[Section~1]{NN}, a \textit{perverse coherent system} is defined to be 
 a pair 
 \begin{align}\label{coh:sys}
 	(F, s), \ F \in \PPer_{c}(X/Y), \ s \colon \oO_X \to F. 
 \end{align}
  Let $(Q^{\dag}, W)$ be a quiver with super-potential, 
 given below 
 	\[
 Q^{\dag}=
 \begin{tikzcd}
 	\bullet_{\infty} \arrow[d] & & \\
 	\bullet_{0}
 	\arrow[rr, bend left,  "a_2"]
 	\arrow[rr,bend left=70,  "a_1"]
 	& &
 	\bullet_1
 	\arrow[ll, bend left, "b_1" ]
 	\arrow[ll, bend left=70, "b_2"]
 \end{tikzcd}
\quad 
 W=a_1 b_1 a_2 b_2-a_1 b_2 a_2 b_1.	
 \]
  Note that $Q^{\dag}$ is an extended quiver obtained from $Q$
 as in Subsection~\ref{subsec:catH}.  
 By the equivalence (\ref{equiv:Phi}),
 giving a perverse coherent system with $\cl(F)=(\beta, n)$
 is equivalent to
 giving a representation of $(Q^{\dag}, W)$
 with dimension vector $(v_{\infty}, v_0, v_1)=(1, n, n-\beta)$.  
 
 \subsection{Categorical DT invariants for the resolved conifold}
For a dimension vector $v=(v_0, v_1)$ of $Q$, let $V_0$, $V_1$ be 
vector spaces with dimension $v_0$, $v_1$ respectively. 
The $\C$-rigidified moduli stack of 
$Q^{\dag}$-representations of dimension vector $(1, v)$
in Subsection~\ref{subsec:catH} is explicitly written as  
\begin{align*}
	\mM_Q^{\dag}(v) &= [R_{Q^{\dag}}(v)/G(v)] \\
	&=	
	\left[V_0 \oplus \Hom(V_0, V_1)^{\oplus 2} \oplus 
	\Hom(V_1, V_0)^{\oplus 2}/\GL(V_0) \times \GL(V_1)    \right]. 
	\end{align*}
Let $w$ be the function 
\begin{align}\label{func:w}
	w=\mathrm{Tr}(W) \colon \mM_Q^{\dag}(v) \to \mathbb{A}^1, \ 
	w(v, A_1, A_2, B_1, B_2)=\mathrm{Tr}(A_1 B_1 A_2 B_2-A_1 B_2 A_2 B_1). 
	\end{align}
Then its critical locus 
\begin{align}\label{w-10}
	\mM_{(Q, W)}^{\dag}(v) \cneq \mathrm{Crit}(w)
	\subset w^{-1}(0)
	\subset \mM_{Q}^{\dag}(v)
\end{align}
is 
the $\C$-rigidified 
moduli stack of $(Q^{\dag}, W)$-representations
of dimension vector $(1, v)$. Here  
the first inclusion follows from the fact that $w$ is a 
homogeneous function on $R_{Q^{\dag}}(v)$ of degree four. 
By the equivalence (\ref{equiv:Phi}), 
$\mM_{(Q, W)}^{\dag}(v)$ is 
isomorphic to the moduli stack of perverse 
coherent systems (\ref{coh:sys})
satisfying $\cl(F)=(v_0-v_1, v_0)$. 

For $\theta=(\theta_0, \theta_1) \in \mathbb{R}^2$, 
we denote by 
 \begin{align*}
 	\mM_Q^{\dag, \theta \sss}(v)=[R_{Q^{\dag}}^{\theta \sss}(v)/G(v)] \subset 
 	\mM_Q^{\dag}(v)
 	\end{align*}
 the open substack of $\theta$-semistable $Q^{\dag}$-representations. 
 We also have the open substack 
 \begin{align*}
 	\mM_{(Q, W)}^{\dag, \theta \sss}(v) \cneq 
 	\mM_Q^{\dag, \theta \sss}(v) \cap \mM_{(Q, W)}^{\dag}(v)
 	\subset \mM_{(Q, W)}^{\dag}(v)
 	\end{align*}
 corresponding to $\theta$-semistable $(Q^{\dag}, W)$-representations. 
 If $\theta_i \in \mathbb{Z}$, as mentioned in Subsection~\ref{subsec:catH}
 these open substacks are 
 GIT semistable locus with respect to the character 
 \begin{align}\label{chi:theta}
 	\chi_{\theta} \colon G(v)=\GL(V_0) \times \GL(V_1) \to \C, \
 	(g_0, g_1) \mapsto \det(g_0)^{-\theta_0} \det(g_1)^{-\theta_1}. 
 	\end{align}
 We have the 
 good moduli spaces by taking GIT quotients 
 \begin{align}\label{gmoduli}
 	\pi_Q^{\dag} \colon \mM_Q^{\dag, \theta \sss}(v) \to M_Q^{\dag, \theta \sss}(v), \ 
 	\pi_{(Q, W)}^{\dag} \colon \mM_{(Q, W)}^{\dag, \theta \sss}(v) \to M_{(Q, W)}^{\dag, \theta \sss}(v).
 	\end{align}
 
 We will consider the triangulated category 
 \begin{align*}
 	\MF(\mM_Q^{\dag, \theta \sss}(v), w)
 	\end{align*}
 and call it the \textit{categorical DT invariant} for the conifold quiver 
 $(Q^{\dag}, W)$. 
 The above triangulated category (or more precisely its dg-enhancement)
 recovers the numerical DT invariant considered in~\cite{NN}: 
 \begin{lem}\label{lem:pcyc}
 	For a generic $\theta \in \mathbb{R}^2$, there is an equality
 	\begin{align*}
 		e_{\mathbb{C}\lgakko u \rgakko}(\mathrm{HP}_{\ast}(\MF(\mM_Q^{\dag, \theta \sss}(v), w))
 		=(-1)^{v_1}\DT^{\theta}(v). 
 		\end{align*}
 	Here $\mathrm{HP}_{\ast}(-)$ is the periodic cyclic homology
 	which is a $\mathbb{Z}/2$-graded $\mathbb{C}\lgakko u \rgakko$-vector 
 	space (see~\cite{MR1667558}), $e_{\mathbb{C}\lgakko u \rgakko}(-)$
 	is the Euler characteristic of
 	$\mathbb{Z}/2$-graded $\mathbb{C}\lgakko u \rgakko$-vector 
 	space, 
 	and $\DT^{\theta}(v) \in \mathbb{Z}$ is the numerical 
 	DT invariant counting 
 	$(Q^{\dag}, w)$-representations with dimension vector $(1, v)$. 
 	\end{lem}
 \begin{proof}
 	Since $\theta$ is generic and the dimension vector $(1, v)$ 
 	of $Q^{\dag}$ is primitive,
 	the stack $M=\mM_Q^{\dag, \theta \sss}(v)$ consists of 
 	only $\theta$-stable objects 
 	and it is a smooth quasi-projective scheme. 
 	By~\cite[Theorem~5.4]{Eff},
 	there is an isomorphism of $\mathbb{Z}/2$-graded vector spaces
 	over $\mathbb{C} \lgakko u \rgakko$. 
 	\begin{align*}
 		\mathrm{HP}_{\ast}(\MF(M, w)) &\cong 
 		H^{\ast}(M, \phi_{w}(\mathbb{Q}_M)) \otimes_{\mathbb{Q}}
 		\mathbb{C}\lgakko u \rgakko \\
 		& \cong H^{\ast+\dim M}(M, \phi_{w}(\IC_M)) \otimes_{\mathbb{Q}}
 		\mathbb{C}\lgakko u \rgakko. 
 		\end{align*} 
 	Here $\phi_{w}(-)$ is the vanishing cycle functor
 	and $u$ has degree two, and $\IC_M=\mathbb{Q}_M[\dim M]$. 
We take the Euler characteristics of both sides as $\mathbb{Z}/2$-graded 
vector spaces over $\mathbb{C}\lgakko u \rgakko$. 
Since we have 
\begin{align*}
	e(H^{\ast}(M, \phi_w(\IC_M))) =
	\int_M \chi_B \ de =: \DT^{\theta}(v), 
\end{align*}
where $\chi_B$ is the Behrend function~\cite{MR2600874} on $M$, 
it is enough to show that $(-1)^{v_1}=(-1)^{\dim M}$. 
Let $E$ be a $Q^{\dag}$-representation with dimension vector $(1, v_0, v_1)$. 
Then we have 
\begin{align*}
	\dim M &=1+\dim \Ext_{Q^{\dag}}^1(E, E)-\dim \Hom_{Q^{\dag}}(E, E) \\
	&=v_0-v_0^2-v_1^2+4v_0 v_1. 
	\end{align*}
Here we have used Lemma~\ref{lem:Euler} below for the second identity. 
Therefore $(-1)^{v_1}=(-1)^{\dim M}$ holds. 
		 	\end{proof}
 
 The following lemma follows immediately from the Euler pairing computations of 
 quiver representations (see~\cite[Corollary~1.4.3]{Brion}): 
 \begin{lem}\label{lem:Euler}
 	For $Q^{\dag}$-representations $E$, $E'$ with dimension 
 	vector $(v_{\infty}, v_0, v_1)$, $(v_{\infty}', v_0', v_1')$, 
 	we have 
 	\begin{align*}
 		\dim \Hom_{Q^{\dag}}(E, E')-
 		\dim \Ext^1_{Q^{\dag}}(E, E')=
 		v_{\infty}v'_{\infty}-v_{\infty}v_0'+v_0 v_0'-2v_0 v_1' -2v_1 v_0'+v_1 v_1'. 
 	\end{align*}
 \end{lem}

We have the following unstable locus 
 \begin{align*}
 	\mM_{(Q, W)}^{\dag, \theta \us}(v) \cneq 
 	\mM_{(Q, W)}^{\dag}(v) \setminus \mM_{(Q, W)}^{\dag, \theta \sss}(v). 
 	\end{align*}
 Then we have the open immersion 
 \begin{align*}
 	\mM_{Q}^{\dag, \theta \sss}(v) \subset \mM_{Q}^{\dag}(v) \setminus \mM_{(Q, W)}^{\dag, \theta\us}(v).
 	\end{align*}
 The following lemma shows that the categorical DT invariant can be also 
 defined on a bigger ambient space: 
 \begin{lem}\label{lem:bigger}
 	The following restriction functor is an equivalence 
 	\begin{align*}
 		\MF(\mM_Q^{\dag}(v) \setminus \mM_{(Q, W)}^{\dag, \theta \us}(v), w)
 		\stackrel{\sim}{\to} 
 		\MF(\mM_{Q}^{\dag, \theta \sss}(v), w). 
 		\end{align*}
 	\end{lem}
\begin{proof}
	The lemma follows since the category of factorizations only 
	depends on an open neighborhood of the critical locus (see the equivalence (\ref{Pre:rest})). 
	\end{proof}
 
 \subsection{Wall-chamber structure}
 There is a wall-chamber structure for the $\theta$-stability 
as in Figure~\ref{figure4} (see~\cite[Figure~1]{NN}): 
 \begin{figure}[H]
 	\centering
 	\begin{tikzpicture}[scale=0.6][node distance=1cm]
 		\draw[thick] (-4.6,0)--(4.6,0)  node [pos=0, anchor=east]{\tiny{$\theta_1=0$}} ;\draw[thick](0,4.6)--(0,-4.6)  node [pos=0, anchor=east]{\tiny{$\theta_0=0$}} ;
 		\draw[thick] (-4,4)--(4,-4)  node [pos=0, anchor=east]{\tiny{$\theta_0+\theta_1=0$}} ;
 		\draw[thick] (-3.5,4.3)--(3.5,-4.3)  node[pos=0, anchor=east]{\tiny{$m\theta_0+(m-1)\theta_1=0$} } ;
 		\draw[thick] (-2.3,4.6)--(2.3,-4.6)  node[pos=0, anchor=east]{\tiny{$2\theta_0+\theta_1=0$} } ;
 		\draw[thick] (-4.3,3.5)--(4.3,-3.5)  node[pos=0, anchor=east]{\tiny{$m\theta_0+(m+1)\theta_1=0$} } ;
 		\draw[thick] (-4.6,2.3)--(4.6,-2.3)  node[pos=0, anchor=east]{\tiny{$\theta_0+2\theta_1=0$} } ;
 		\draw[fill] (-3.5,3.6) circle [radius=0.025];
 		\node [thick, right] at (-3.5,3.6) {\tiny{PT}};
 		\draw[fill] (-3.6,3.5) circle [radius=0.025];
 		\node [thick, below] at (-3.6,3.5) {\tiny{DT}};
 		\draw[fill] (3.8,-3.7) circle [radius=0.025];
 		\node [thick, above] at (3.8,-3.7) {\tiny{PT}};
 		
 		\draw[fill] (3.7,-3.8) circle [radius=0.025];
 		\node [thick, left] at (3.7,-3.8) {\tiny{DT}};
 		
 		\node at (4.2,-3.7) {\small{$X^+$}};
 		\node at (-4,3.7) {\small{$X$}};
 		
 		\draw[fill] (-3.46,2.5) circle [radius=0.015];
 		\draw[fill] (-3.51,2.4) circle [radius=0.015];
 		\draw[fill] (-3.56,2.3) circle [radius=0.015];
 		
 		\draw[fill] (-2.5,3.46) circle [radius=0.015];
 		\draw[fill] (-2.4,3.51) circle [radius=0.015];
 		\draw[fill] (-2.3,3.56) circle [radius=0.015];
 		
 		\draw[fill] (-3.15, 2.97) circle [radius=0.015];
 		\draw[fill] (-3.2, 2.9) circle [radius=0.015];
 		\draw[fill] (-3.25, 2.83) circle [radius=0.015];
 		
 		\draw[fill] (-2.97, 3.15) circle [radius=0.015];
 		\draw[fill] (-2.9, 3.22) circle [radius=0.015];
 		\draw[fill] (-2.83, 3.28 ) circle [radius=0.015];
 		
 		\node at (2.5,2) {$\begin{subarray}{c}\mathrm{empty\,\, chamber} \end{subarray}$ } ;
 		\node at (-3,-2.3) {$\begin{subarray}{c}\mathrm{non-commutative}  \\ \mathrm{chamber} \end{subarray}$ } ;
 	\end{tikzpicture} 
 	\caption{Wall-chamber structures}
 	\label{figure4}
 \end{figure}
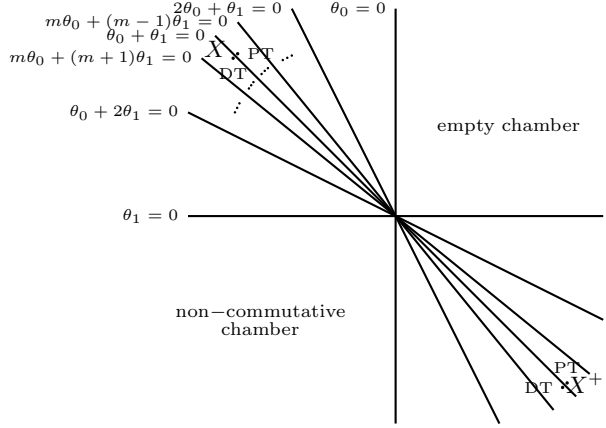
In Figure~\ref{figure4}, 
if $\theta$ lies in the first quadrant
then $\mM_{(Q, W)}^{\dag, \theta \sss}(v)=\emptyset$
unless $v=0$, 
so it is called an \textit{empty chamber}. 
In this case, the categorical DT invariants are given 
in the following lemma: 
\begin{lem}\label{lem:empty}
	Let $\theta_{\rm{en}} \in \mathbb{R}^2$ lies
	in an empty chamber. 
	Then 
	\begin{align*}
		\MF(\mM_Q^{\dag, \theta_{\rm{en}} \sss}(v), w)
		=\begin{cases}  \MF(\Spec \mathbb{C}, 0), & v=0 \\
		0, & v\neq 0.
		\end{cases}
		\end{align*}	
	\end{lem}
\begin{proof}
	If $v \neq 0$, then $\MF(\mM_Q^{\dag, \theta_{\rm{en}} \sss}(v), w)=0$
	by the equivalence (\ref{Pre:rest}), 
	since the critical locus of $w$ is empty. 
	If $v=0$, then $\mM_Q^{\dag, \theta_{\rm{en}}\sss}(v)=\Spec \mathbb{C}$
	and $w=0$. 
	\end{proof}

  We focus on
  the wall in the second quadrant, classified by $m \in \mathbb{Z}_{\ge 1}$
 \begin{align*}
 	W_m \cneq \mathbb{R}_{>0} \cdot (1-m, m) \subset \mathbb{R}^2.  
 	\end{align*}
 If $\theta$ lies between $W_{m}$ and $W_{m+1}$, 
 then the moduli stack $\mM_{(Q, W)}^{\dag, \theta \sss}(v)$
 is constant, 
 consisting of 
$\theta$-stable objects. 
So $\mM_{(Q, W)}^{\dag, \theta \sss}(v)$
 is a quasi-projective scheme, 
 and the good moduli space morphism $\pi_{(Q, W)}^{\dag}$ in (\ref{gmoduli}) is an isomorphism. 
 If $\theta$ is also sufficiently close to the wall $W_m$, then 
 $\mM_Q^{\dag, \theta \sss}(v)$ also consists of 
 $\theta$-stable objects and the morphism $\pi_Q^{\dag}$ in (\ref{gmoduli})
 is an isomorphism. 
 The categorical DT invariant is also constant when $\theta$ deforms 
 inside a chamber: 
 \begin{lem}\label{lem:catwall}
 	The triangulated category $\MF(\mM_{Q}^{\dag, \theta \sss}(v), w)$ is 
 	constant (up to equivalence) when $\theta$ deforms inside a chamber in Figure~\ref{figure4}.
 	\end{lem}
 \begin{proof}
 	Suppose that $\theta$ lies in a chamber in Figure~\ref{figure4}.
 	Although $\theta$ does not lie in a wall for 
 	$(Q^{\dag}, W)$-representations, it may lie on a wall for 
 	$Q^{\dag}$-representations. 
 	However the destabilizing locus 
 	in $\mM_Q^{\dag, \theta \sss}(v)$ 
 	is disjoint from $\Crit(w)=\mM_{(Q, W)}^{\dag, \theta \sss}(v)$, 
 	so by (\ref{Pre:rest})
 	the triangulated categories
 	 $\MF(\mM_{Q}^{\dag, \theta \sss}(v), w)$
 	 are equivalent under wall-crossing inside a 
 	 chamber of Figure~\ref{figure4}.  	
 	\end{proof}
 
 For $\theta \in W_m$, there is a unique (up to isomorphism)
 $\theta$-stable $(Q, W)$-representation $S_m$ (i.e. $(Q^{\dag}, W)$-representation whose dimension vector at $\infty$ is zero), 
 given by (see~\cite[Theorem~3.5]{NN})
 \begin{align}\label{def:Tm}
 	S_m	\cneq \left(\xymatrix{ \mathbb{C}^m  \ar@/^1.5pc/[rr]^{0}  \ar@/^0.5pc/[rr]^{0} 
 		&& \mathbb{C}^{m-1} \ar@/^0.5pc/[ll]^{B_1^0}  \ar@/^1.5pc/[ll]^{B_2^0}  }
 	\right), \
  	B_1^0(\mathbf{f}_i)=\mathbf{e}_i, \ 
 	B_2^0(\mathbf{f}_i)=\mathbf{e}_{i+1}. 
 	\end{align}
 Here $\{\mathbf{e}_1, \ldots, \mathbf{e}_m\}$, 
 $\{\mathbf{f}_1, \ldots, \mathbf{f}_{m-1}\}$ are basis of 
 $\mathbb{C}^m$, $\mathbb{C}^{m-1}$ respectively. 
 Note that $S_m$ has dimension vector $s_m=(m, m-1)$ so that 
 $\theta(S_m)=0$ when $\theta \in W_m$. 
 Under the equivalence $\Phi$ in (\ref{equiv:Phi}), we have 
 the following relation (see~\cite[Remark~3.6]{NN})
 \begin{align}\label{Phi:T}
 	\Phi(\oO_C(m-1))=S_m. 
 	\end{align}
 
 Since $s_m=(m, m-1)$ is primitive, the moduli stack 
 $\mM_{Q}^{\theta \sss}(s_m)$ consists of $\theta$-stable 
 $Q$-representations, and the good moduli space morphism 
 \begin{align}\label{MQ:sm}
 	\mM_Q^{\theta \sss}(s_m) \to M_Q^{\theta \sss}(s_m)
 	\end{align} 
 is a $\C$-gerbe. 
 There is a function defined similarly to (\ref{func:w})
 \begin{align}\notag
 	w=\mathrm{Tr}(W) \colon 
 	\mM_Q^{\theta \sss}(s_m) \to \mathbb{A}^1
 	\end{align}
 whose critical locus $\mM_{(Q, W)}^{\theta \sss}(s_m)$
 is the moduli stack of 
 $\theta$-stable $(Q, W)$-representation.
 Note that $\mM_{(Q, W)}^{\theta \sss}(s_m)$
 consists of a one point
 corresponding to the unique $\theta$-stable
 $(Q, W)$-representation $S_m$. 
 \begin{lem}\label{lem:crit:simple}
 	For any $j\in \mathbb{Z}$, there is an equivalence 
 	\begin{align*}
 		\MF(\mM_Q^{\theta \sss}(s_m), w)_j \simeq \MF(\Spec \mathbb{C}, 0). 
 		\end{align*} 	
 	\end{lem}
 \begin{proof}
 	Let $V_0=\mathbb{C}^{m}$, $V_1=\mathbb{C}^{m-1}$ and 
 	$B_1^0, B_2^0 \colon V_1 \to V_0$ be maps as in (\ref{def:Tm}). 
 	Note that we have 
 	\begin{align*}
 \mM_Q^{\theta \sss}(s_m)=\left[ (\Hom(V_0, V_1)^{\oplus 2} \oplus \Hom(V_1, V_0)^{\oplus 2})/\GL(V_0) \times \GL(V_1)   \right]. 
 \end{align*}	
It admits a projection 
 \begin{align}\label{project:Q}
 \mM_Q^{\theta \sss}(s_m) \to 	
 	[\Hom(V_1, V_0)^{\oplus 2}/\GL(V_0) \times \GL(V_1)].
 	\end{align}
 The target of the above morphism is identified with the moduli stack of 
 representations of the Kronecker quiver $Q_K$ (i.e. two vertices $\{0, 1\}$ with 
 two arrows from $1$ to $0$), so it contains an open 
 substack corresponding to stable coherent sheaves on $\mathbb{P}^1$
 under the Beilison equivalence 
 \begin{align*}
 	\RHom(\oO_{\mathbb{P}^1} \oplus \oO_{\mathbb{P}^1}(1), -) \colon 
 	D^b(\mathbb{P}^1) \stackrel{\sim}{\to} D^b(\mathrm{Rep}(Q_K)). 
 	\end{align*} 
 Under the above equivalence, 
 $\oO_{\mathbb{P}^1}(m-1)$ corresponds to $(B_1^0, B_2^0)$. 
 Since $\oO_{\mathbb{P}^1}(m-1)$ is rigid in $\mathbb{P}^1$, there is a 
 $\GL(V_0) \times \GL(V_1)$-invariant open neighborhood 
 \begin{align*}
 	(B_1^0, B_2^0) \in \uU \subset \Hom(V_1, V_0)^{\oplus 2}
 	\end{align*}
 such that $
 	[\uU/\GL(V_0) \times \GL(V_1)]$ is isomorphic to $B\C$. 
 By pulling it back by the projection (\ref{project:Q}), we see that 
 there is an open immersion 
 \begin{align}\label{open:V}
 	\Hom(V_0, V_1)^{\oplus 2} \times B\C \subset 
 	\mM_Q^{\dag, \theta \sss}(s_m), \ 
 	(A_1, A_2) \mapsto (A_1, A_2, B_1^0, B_2^0),
 	\end{align}
 whose image contains $\{S_m\}=\Crit(w)$
 as $0 \in \Hom(V_0, V_1)^{\oplus 2}$. 
 By the equivalence (\ref{Pre:rest}), the restriction functor gives an equivalence 
 \begin{align}\label{equiv:BCast}
 	\MF(\mM_Q^{\dag, \theta \sss}(s_m), w) \stackrel{\sim}{\to}
 	\MF(\Hom(V_0, V_1)^{\oplus 2} \times B\C, w). 
 	\end{align}
 The function $w$ restricted (\ref{open:V}) is a quadratic function by the definition of $w$, 
 which must be non-degenerate as its critical locus is one point. 
 Since $\Hom(V_0, V_1)^{\oplus 2}$ is even dimensional, the 
 RHS of (\ref{equiv:BCast}) is equivalent to $\MF(B\C, w)$
 by the Kn\"{o}rrer periodicity in Theorem~\ref{Phi:period}.  
 	\end{proof}
 
 \subsection{Descriptions of formal fibers}\label{subsec:formal}
 By the above classification of $\theta$-stable $(Q, W)$-representations, 
 a $\theta$-polystable $(Q^{\dag}, W)$-representation of 
 dimension vector $(1, v_0, v_1)$ at the wall $\theta \in W_m$
 is of the form 
 \begin{align}\label{pstableR}
 	R=R_{\infty} \oplus (V \otimes S_m)
 	\end{align}
 where $V$ is a finite dimensional vector space
 and $R_{\infty}$ is a $\theta$-stable $(Q^{\dag}, W)$-representation. 
 By setting $d \cneq \dim V$, 
 the dimension vector of $R_{\infty}$ is 
 $(1, v_0-dm, v_1-d(m-1))$. 
 By regarding $R$ as a $\theta$-polystable $Q^{\dag}$-representation, 
 it determines a point $p \in M_{Q}^{\dag, \theta\sss}(v)$. 
 \begin{rmk}\label{rmk:V}
 	The vector space $V$ in (\ref{pstableR}) will play the same
 	role of the vector space $V$ in Section~\ref{sec:DGflip}. 
 Below we fix a basis of $V$ 
 	and use the same convention of the dominant chamber in Subsection~\ref{subsec:Gflip:2}. 
 	\end{rmk}
 
 Below we fix a $\theta$-polystable object (\ref{pstableR}), and
 $p \in M_{Q}^{\dag, \theta\sss}(v)$ is the corresponding 
 point as above. We will give a description of the formal fiber 
 of the good moduli space morphism $\pi_Q^{\dag} \colon \mM_Q^{\dag, \theta \sss}(v) \to 
 M_Q^{\dag, \theta \sss}(v)$ at $p$. 
 We set 
 \begin{align*}
 	G_p \cneq \Aut(R) =\GL(V).
 	\end{align*} 
 It acts on $\Ext_{Q^{\dag}}^1(R, R)$
 by the conjugation, 
 and we have the good moduli space morphism
  \begin{align}\label{ExtQ}
 	[\Ext_{Q^{\dag}}^1(R, R)/G_p] \to \Ext_{Q^{\dag}}^1(R, R)\ssslash G_p. 
 \end{align} 
Let $q \in R_{Q^{\dag}}(v)$ be a point corresponding to the polystable object (\ref{pstableR}). 
Note that $\Ext_{Q^{\dag}}^1(R, R)$ is the tangent space of the stack 
$\mM_{Q}^{\dag, \theta \sss}(v)$ at $q$. 
By Luna's \'{e}tale slice theorem, there exists a 
$G_p$-invariant locally closed subset $q \in W_p \subset R_{Q^{\dag}}(v)$
and a commutative diagram 
\begin{align}\label{dia:Luna}
	\xymatrix{
	([\Ext_{Q^{\dag}}^1(R, R)/G_p], 0)  \ar[d] \diasquare & \ar[l]  ([W_p/G_p], q)
	 \ar[r] \ar[d] \diasquare
	 & (\mM_{Q}^{\dag, \theta \sss}(v), q) \ar[d] \\
	(\Ext_{Q^{\dag}}^1(R, R) \ssslash G_p, 0)  & (W_p \ssslash G_p, q) \ar[r] \ar[l] &
	(M_Q^{\dag, \theta \sss}(v), p)
}
	\end{align}
such that each horizontal arrows are \'{e}tale. 

 We have the following decomposition of $\Ext_{Q^{\dag}}^1(R, R)$ as $G_p$-representations, 
 \begin{align} \label{Ext:decom}	
 	\Ext_{Q^{\dag}}^1(R, R)&=
 	\Ext_{Q^{\dag}}^1(R_{\infty}, R_{\infty}) \oplus 
 	(V \otimes \Ext_{Q^{\dag}}^1(R_{\infty}, S_m)) \\
 	&\quad \oplus \notag
 	(V^{\vee} \otimes \Ext_{Q^{\dag}}(S_m, R_{\infty}))
 	\oplus (\End(V) \otimes \Ext_{Q}^1(S_m, S_m)) \\
 \notag	&=(\Ext_{Q^{\dag}}^1(R_{\infty}, R_{\infty}) \oplus \Ext_Q^1(S_m, S_m)) \oplus 
 	(V \otimes \Ext_{Q^{\dag}}^1(R_{\infty}, S_m)) \\
\notag &\quad \oplus 
 	(V^{\vee} \otimes \Ext_{Q^{\dag}}^1(S_m, R_{\infty}))
 	\oplus (\End_0(V) \otimes \Ext_{Q}^1(S_m, S_m)). 	 
 	\end{align}
 Here $\End_0(V)$ is the kernel of 
 the trace map $\mathrm{Tr} \colon \End(V) \to \mathbb{C}$
 which is an irreducible $G_p$-representation, 
  and the last identity gives a direct sum 
  decomposition of 
 $\Ext_{Q^{\dag}}^1(R, R)$ into 
 its irreducible $G_p$-representations whose 
 irreducible factors are $\mathbb{C}$ (trivial representation), 
 $V$, $V^{\vee}$ and $\End_0(V)$. 
 The number of summands 
 is calculated as follows: 
 \begin{lem}\label{lem:computea}
We have the following identities: 
 \begin{align}
 \label{compute:abc}&a_{v, m, d} \cneq \ext_{Q^{\dag}}^1(R_{\infty}, S_m)=
 C_{v, m}+m+d(-2m^2+2m+1), \\ 
 \notag&b_{v, m, d} \cneq \ext_{Q^{\dag}}^1(S_m, R_{\infty})=C_{v, m}+d(-2m^2+2m+1), \\
 \notag&C_{v, m} \cneq (m-2)v_0 +(m+1)v_1, \\
 \notag&c_m \cneq \ext_{Q^{\dag}}^1(S_m, S_m)=2m^2-2m. 
 	\end{align}
 \end{lem}
\begin{proof}
	The lemma easily follows from 
	Lemma~\ref{lem:Euler} noting that 
	\begin{align*}
		\Hom(R_{\infty}, S_m)=\Hom(S_m, R_{\infty})=0, \ 
		\Hom(T_{m}, S_m)=\mathbb{C}. 
	\end{align*}
For example, 
since the dimension vectors of $R_{\infty}$, $S_m$
are $(1, v_0-md, v_1-(m-1)d)$, $(0, m, m-1)$ respectively, 
we have 
\begin{align*}
	-a_{v, m, d}&=\hom(R_{\infty}, S_m)-\ext^1(R_{\infty}, S_m) \\
	&=-m+(v_0-md)m-2(v_0-md)(m-1)-2(v_1-(m-1)d)m+(m-1)(v_1-(m-1)d) \\
	&=-(m-2)v_0-(m+1)v_1-m-d(-2m^2+2m+1). 
	\end{align*}
	\end{proof}

The left vertical arrow in (\ref{dia:Luna}) is also identified with 
a moduli stack of some quiver representations and its good moduli space. 
We define $Q_p$ to be the Ext-quiver for $\{S_m\}$
and $Q_p^{\dag}$ to be the Ext-quiver for $\{R_{\infty}, S_m\}$. 
Namely  
 $Q_p$ is the quiver with one vertex $\{1\}$
 and the number of loops at $1$ is $c_m$. 
The quiver 
 $Q_p^{\dag}$ consists of two vertices $\{\infty, 1\}$, 
 the number of arrows from $\infty$ to $1$ is $a_{v,m,d}$, 
 from $1$ to $\infty$ is $b_{v,m,d}$, and the number of loops at 
 $\infty$ (resp.~$1$) is $\ext_{Q^{\dag}}^1(R_{\infty}, R_{\infty})$ (resp.~$c_m$).  
 From (\ref{Ext:decom}), we have the identification  
 \begin{align}\label{isom:Ext1}
 	\xymatrix{\mM_{Q_p}^{\dag}(d) \ar@{=}[r] \ar[d] & \left[\Ext^1_{Q^{\dag}}(R, R)/G_p \right] \ar[d] \\
 		 		M_{Q_p}^{\dag}(d) \ar@{=}[r] & \Ext^1_{Q^{\dag}}(R, R)\ssslash G_p. 
 	 }
 	\end{align}
   By combining the diagrams (\ref{dia:Luna}), (\ref{isom:Ext1}) and 
   taking the formal fibers,  
 we have a commutative diagram 
 \begin{align}\label{ffiber}
 	\xymatrix{\mM_{Q_p}^{\dag}(d) \ar[d] \diasquare &  \ar[l]
 		\widehat{\mM}_{Q_p}^{\dag}(d) \ar[d] \ar@{=}[r]\ar@/^2.0pc/[rr]^-{\eta_p} & 
 		 		\left[\widehat{\Ext}^1_{Q^{\dag}}(R, R)/G_p \right] 
 		\ar[r]^-{\cong} \ar[d] & \widehat{\mM}_{Q}^{\dag, \theta \sss}(v)_{p} \ar[d] \ar[r] \diasquare & \mM_Q^{\dag, \theta \sss}(v) \ar[d]  \\
 		M_{Q_p}^{\dag}(d) & \ar[l]
 		\widehat{M}_{Q_p}^{\dag}(d) \ar@{=}[r] & 
 		\widehat{\Ext}^1_{Q^{\dag}}(R, R) \ssslash G_p
 		\ar[r]^-{\cong}	 & \widehat{M}_{Q}^{\dag, \theta\sss}(v)_{p} \ar[r] & 
 		M_Q^{\dag, \theta \sss}(v).  
 	} 	
 \end{align}
Here each vertical arrow is a good moduli space morphism, 
the vertical arrow second from the right (resp.~left)
is the formal fiber of the right (resp.~left) one at $p$ (resp.~origin), 
the middle vertical arrow is the formal fiber of the morphism (\ref{ExtQ})
at the origin, and the square second from the right 
is 
obtained by the formal completions 
of good moduli spaces in the diagram (\ref{dia:Luna}).

We then compare the semistable loci under the isomorphism $\eta_p$ in 
the diagram (\ref{ffiber}).  
 We take $\theta=(\theta_0, \theta_1) \in W_m$ and 
 $\theta_{\pm}$ of the form 
 \begin{align}\label{thetapm}
 	\theta_{\pm}=(\theta_0 \mp\varepsilon, \theta_1\pm\varepsilon), \ \varepsilon>0. 
 	\end{align}
 We take $(\theta_0, \theta_1)$ and $\varepsilon$ to be integers and 
 $\theta_{\pm}$ lie on chambers adjacent to $W_m$ which are sufficiently 
 close to $W_m$, e.g. 
 take $\varepsilon=1$ and $(\theta_0, \theta_1)=N \cdot (1-m, m)$ for a sufficiently 
 large integer $N$. 
 We have the open substacks 
 \begin{align*}
 	\mM_Q^{\dag, \theta_{\pm} \sss}(v) \subset 
 	\mM_Q^{\dag, \theta \sss}(v), \ 
 	\widehat{\mM}_Q^{\dag, \theta_{\pm} \sss}(v)_p
 	\subset \widehat{\mM}_Q^{\dag, \theta \sss}(v)_p
 	\end{align*}
 corresponding to $\theta_{\pm}$-semistable representations.  
 
 On the other hand, as in (\ref{ch:det})
 we set $\chi_0 \colon \GL(V) \to \C$ to be the 
 determinant character $g \mapsto \det(g)$. 
 We have the open substacks 
 \begin{align*}
 	\mM_{Q_p}^{\dag, \chi_0^{\pm 1} \sss}(d) \subset 
 	\mM_{Q_p}^{\dag}(d), \ 
 		\widehat{\mM}_{Q_p}^{\dag, \chi_0^{\pm 1} \sss}(d) \subset 
 	\widehat{\mM}_{Q_p}^{\dag}(d)
 	\end{align*}
 corresponding to $\chi_0^{\pm 1}$-semistable $Q_p^{\dag}$-representations. 
 We have the following lemma: 
 \begin{lem}\label{lem:chitheta}
 The isomorphism 	
  $\eta_p$ in (\ref{ffiber}) restricts to the isomorphisms 
 	\begin{align}\label{isom:thetap}
 		\eta_p \colon \widehat{\mM}_{Q_p}^{\dag, \chi_0^{\pm 1} \sss}(d)
 		\stackrel{\cong}{\to}
 		\widehat{\mM}_Q^{\dag, \theta_{\pm} \sss}(v)_p. 
 	\end{align}
 	\end{lem}
 \begin{proof}
Let us consider the composition 
 \begin{align}\label{compose:Gp}
 	G_p=\GL(V) \hookrightarrow \GL(V_0) \times \GL(V_1) \stackrel{\chi_{\theta_{\pm}}}{\to}
 	\C.
 	\end{align}
 We see that the above composition is given by 
 $g \mapsto \det(g)^{\pm \varepsilon}$, where 
 $\chi_{\theta_{\pm}}$ is the character 
 (\ref{chi:theta}) applied to $\theta_{\pm}$. 
Indeed we have 
	\begin{align*}
		V_0=(V\otimes \mathbb{C}^{m}) \oplus \mathbb{C}^{v_0-dm}, \ 
		V_1=(V \otimes \mathbb{C}^{m-1}) \oplus \mathbb{C}^{v_1-d(m-1)}. 
		\end{align*}
	The embedding $\GL(V) \hookrightarrow \GL(V_0) \times \GL(V_1)$
	is given by 
	\begin{align*}
		g \mapsto \left( (g \otimes 1_{\mathbb{C}^m}) \oplus 1_{\mathbb{C}^{v_0-dm}}, (g \otimes 1_{\mathbb{C}^{m-1}})
		\oplus 1_{\mathbb{C}^{v_1-d(m-1)}}   \right). 
		\end{align*}
	By composing it with $\chi_{\theta_{\pm}}$, we see that the composition (\ref{compose:Gp})
	is given by $g \mapsto \det(g)^{\pm \varepsilon}$. 
	Therefore under the isomorphism $\eta_p$ in (\ref{ffiber}) the line bundle on $\widehat{\mM}_Q^{\dag, \theta \sss}(v)$ 
	determined by $\chi_{\theta_{\pm}}$ corresponds to 
	that on $\widehat{\mM}_{Q_p}^{\dag}(d)$
	determined by $\chi_0^{\pm \varepsilon}$. 
	Therefore the lemma holds. 
	\end{proof}

 \subsection{Reduced Ext-quiver}
  We define the \textit{reduced Ext-quiver}
 $Q_{p}^{\rm{red}, \dag}$
 to be 
 the quiver obtained from $Q_p^{\dag}$ by removing 
 all the loops at the vertex $\{1\}$, 
 and adding $c_m$-loops at the vertex $\{\infty\}$, 
 where $c_m$ is given in (\ref{compute:abc}). 
 It contains the full sub quiver
 \begin{align}\label{def:Qpred}
 	Q_p^{\rm{red}} \subset Q_p^{\rm{red}, \dag}
 \end{align}
 consisting of the vertex $\{1\}$ and no loops. 
 See the following picture: 
 \[
 Q_p^{\dag}=
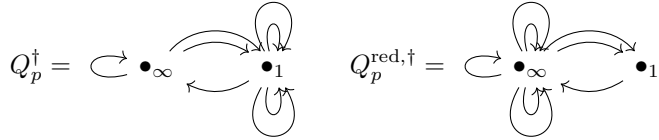
 \begin{tikzcd}
 	\bullet_{\infty}
 	\arrow[r,bend left]
 	\arrow[r,bend left=50]
 	\arrow[out=195, in=165, loop] & 
 	\bullet_1
 	\arrow[l, bend left]
 	\arrow[
 	out=105,
 	in=75,
 	loop,
 	distance=0.5cm]
 	\arrow[
 	out=120,
 	in=60,
 	loop,
 	distance=1cm]
 	\arrow[
 	out=-105,
 	in=-75,
 	loop,
 	distance=0.5cm]
 	\arrow[
 	out=-120,
 	in=-60,
 	loop,
 	distance=1cm]
 \end{tikzcd}
 \quad 
 Q_p^{\rm{red}, \dag}=
 \begin{tikzcd}
 	\bullet_{\infty}
 	\arrow[r,bend left]
 	\arrow[r,bend left=50]
 	\arrow[out=195, in=165, loop] 
 	\arrow[
 	out=105,
 	in=75,
 	loop,
 	distance=0.5cm]
 	\arrow[
 	out=120,
 	in=60,
 	loop,
 	distance=1cm]
 	\arrow[
 	out=-105,
 	in=-75,
 	loop,
 	distance=0.5cm]
 	\arrow[
 	out=-120,
 	in=-60,
 	loop,
 	distance=1cm]
 	& 
 	\bullet_1
 	\arrow[l, bend left]
 \end{tikzcd}
 \]
  
 Let $\mM_{Q^{\rm{red}}_{p}}^{\dag}(d)$ be the 
 $\C$-rigidified moduli stack of $Q_{p}^{\rm{red}, \dag}$-representations 
 with dimension vector $(1, d)$. It is described as 
 \begin{align}\label{M:reduced}
 	\mM_{Q_{p}^{\rm{red}}}^{\dag}(d)&=
 	\left[\left(\mathbb{C}^{\ext_{Q^{\dag}}^1(R_{\infty}, R_{\infty})+c_m}
 	\oplus V^{\oplus a_{v, m, d}}
 	\oplus (V^{\vee})^{\oplus b_{v, m, d}} \right)
 	/\GL(V) \right], \\
 \notag	&=\mathbb{C}^{\ext_{Q^{\dag}}^1(R_{\infty}, R_{\infty})+c_m} \times 
 	\gG_{a_{v, m, d}, b_{v, m, d}}(d), 
 \end{align}
 i.e. it is obtained from (\ref{Ext:decom}) by removing 
 the last factor $\End_0(V) \otimes \Ext_Q^1(S_m, S_m)$, and taking 
 the quotient by $\GL(V)$. 
 Here $\gG_{a, b}(d)$ is the quotient stack (\ref{Gabd}) studied in 
 Section~\ref{sec:DGflip}. 
 We also denote by 
 $\widehat{\mM}_{Q_{p}^{\rm{red}}}^{\dag}(d)$
 the formal fiber for the 
 good moduli space morphism 
 \begin{align*}
 	\mM_{Q_{p}^{\rm{red}}}^{\dag}(d) \to M_{Q_{p}^{\rm{red}}}^{\dag}(d)
 	\end{align*}
 at the origin.

   By restricting the function (\ref{func:w}) to the formal fiber
  of the good moduli space morphism (\ref{gmoduli}) and 
  pulling it back by 
  the isomorphism $\eta_p$ in (\ref{ffiber}), we have the 
function  
 \begin{align*}
 	w_p \colon
 	\widehat{\mM}_{Q_p}^{\dag}(d)= [\widehat{\Ext}_{Q^{\dag}}^1(R, R)/G_p] \to \mathbb{A}^1. 
 	\end{align*}
 We see that the above function is a sum of 
 a function from $\widehat{\mM}_{Q_p^{\rm{red}}}^{\dag}(d)$
 and some non-degenerate quadratic form. 
 Let us take a (non-canonical) isomorphism 
 of $\mathbb{C}$-vector spaces 
  \begin{align}\label{isom:W}
 \Ext_Q^1(S_m, S_m) \cong H \oplus H^{\vee}
 \end{align}
where the dimension of $H$ is $m^2-m$. 
There is also an isomorphism 
$\End_0(V) \cong \End_0(V)^{\vee}$ of $G_p$-representations, 
so 
we have 
  an isomorphism of $G_p$-representations
 \begin{align}\notag
 	\End_0(V) \otimes \Ext_Q^1(S_m, S_m) \cong W \oplus W^{\vee}
 \end{align}
 where $W=\End_0(V) \otimes H$.
 In particular we have the non-degenerate symmetric 
 quadratic form
 \begin{align}\label{q:W}
 	q=\langle -, -\rangle \colon 
 	\End_0(V) \otimes \Ext_Q^1(S_m, S_m) \to \mathbb{A}^1
 \end{align}
 defined to be the natural pairing on 
 $W$ and $W^{\vee}$. Note that (\ref{isom:W}) is a summand of 
 $\Ext_{Q^{\dag}}^1(R, R)$ by the decomposition (\ref{Ext:decom}). 
 We will use the following proposition, whose proof will be given 
 in Subsection~\ref{subsec:prop:func}: 
 \begin{prop}\label{prop:func}
 	By replacing the isomorphisms in (\ref{ffiber}) and (\ref{isom:W}) if necessary, 
 	the function $w_p$ is written as 
 	\begin{align}\label{func:wp}
 		w_p=w_{p}^{\rm{red}} + q. 
 		\end{align}
 	Here $w_{p}^{\rm{red}}$ is non-zero and 
 	does not contain variables from 
 	$\End_0(V) \otimes  \Ext_Q^1(S_m, S_m)$
 	under the decomposition (\ref{Ext:decom}).  
 	\end{prop}

 The $\GL(V)$-representation $W$ determines the vector bundle
 $\wW$ on $\widehat{\mM}_{Q_{p}^{\rm{red}}}^{\dag}(d)$. 
 By Proposition~\ref{prop:func}, we have the commutative diagram  
 \begin{align}\label{dia:WW}
 	\xymatrix{
\wW^{\vee} \ar@<-0.3ex>@{^{(}->}[r]^-{i} \ar[d]_-{\mathrm{pr}} & \wW \oplus \wW^{\vee} \ar[rd]_-{w_p^{\rm{red}}+q}& \ar[l]_-{\iota} \widehat{\mM}_{Q_{p}}^{\dag}(d) \ar[d]^-{w_p} \ar[r]^-{\eta_p}_-{\cong} & \widehat{\mM}^{\dag, \theta \sss}_Q(v)_p \ar[ld]^-{w_p}  \\
\widehat{\mM}_{Q_{p}^{\rm{red}}}^{\dag}(d) \ar[rr]_-{w_{p}^{\rm{red}}} & &
\mathbb{A}^1.   &
}
 	\end{align}
 Here $\mathrm{pr}$ is the projection, $i$ is given by $i(x)=(0, x)$, 
 $\iota$ is the natural morphism by the formal completion 
 (see Lemma~\ref{lem:formW})
 and $\eta_p$ is the isomorphism in (\ref{ffiber}). 
 We have the following proposition: 
 \begin{prop}\label{prop:funct}
There is an equivalence
\begin{align}\label{funct:Knoer}
	\Phi_p \cneq 
	\iota^{\ast}i_{\ast}\mathrm{pr}^{\ast} \colon 
	\MF(\widehat{\mM}_{Q_{p}^{\rm{red}}}^{\dag}(d), w_p^{\rm{red}})
	\stackrel{\sim}{\to}
	\MF(\widehat{\mM}_{Q_p}^{\dag}(d), w_p). 
	\end{align}	
 	\end{prop}
 \begin{proof}
  	The composition functor 
 	\begin{align}\label{func:Kos}
 		i_{\ast}\mathrm{pr}^{\ast} \colon 
 		\MF(\widehat{\mM}_{Q^{\rm{red}}_{p}}^{\dag}(d), w_p^{\rm{red}})
 		\stackrel{\mathrm{pr}^{\ast}}{\to}
 		\MF(\wW^{\vee}, w_p^{\rm{red}})
 		\stackrel{i_{\ast}}{\to} 
 		\MF(\wW\oplus \wW^{\vee}, w_p^{\rm{red}}+q)
 		\end{align}
 	is an equivalence by Theorem~\ref{thm:period}. 
By Lemma~\ref{lem:formW}, 
 	the functor 
 	\begin{align}\label{iota:ast2}
 		\iota^{\ast} \colon 
 		\MF(\wW\oplus \wW^{\vee}, w_p^{\rm{red}}+q)
 		\to
 		\MF(\widehat{\mM}_{Q_{p}}^{\dag}(d), w_p)
 		\end{align}
 	is fully-faithful with dense image. 
 By Lemma~\ref{lem:idem} and the equivalence (\ref{func:Kos}), 
 the LHS of (\ref{iota:ast2}) is idempotent complete, 
 so the functor (\ref{iota:ast2}) is an equivalence. 
 	Therefore we obtain the proposition. 
 	\end{proof}

 \subsection{Window subcategories}
 In this subsection, we define several window subcategories 
 for moduli stacks of representations of quivers and their 
 formal fibers discussed in the previous subsection. 
 The notation is summarized in Table~1. 
 
 \begin{table}[htb]\label{table1}
 	\caption{Notation of moduli spaces and windows}
 	\begin{tabular}{|c||c|c|c|} \hline
 		& moduli stack & formal fiber &
 		windows  \\ \hline \hline
 		conifold quiver $Q^{\dag}$ & $\mM_Q^{\dag, \theta \sss}(v)$ &
 		$\widehat{\mM}_Q^{\dag, \theta \sss}(v)_p$	
 		& $\mathbb{W}_{\rm{glob}}^{\theta_{\pm}}(v), \mathbb{W}_{\rm{loc}}^{\theta_{\pm}}(v)_p$
 		\\ \hline
 		Ext-quiver $Q_p^{\dag}$ & 	$\mM^{\dag}_{Q_p}(d)$ &
 		$\widehat{\mM}^{\dag}_{Q_p}(d)$ & $\mathbb{W}^{\pm}(d)_p$
 		\\  \hline  
 		reduced Ext-quiver $Q_p^{\rm{red}, \dag}$ &
 		$\mM^{\dag}_{Q^{\rm{red}}_p}(d)$  &	$\widehat{\mM}^{\dag}_{Q^{\rm{red}}_p}(d)$
 		& $\mathbb{W}_c(d)_p$
 		\\
 		\hline
 	\end{tabular}
 \end{table}

 \subsubsection*{Global window subcategory $\mathbb{W}_{\rm{glob}}^{\theta_{\pm}}(v)$}
  We take $\theta \in W_m$ and 
 $\theta_{\pm}$ as in (\ref{thetapm}) which are sufficiently close 
 to the wall $W_m$.
 Then 
 the KN stratification of 
 $\mM^{\dag}_{Q}(v)$ for $\chi_{\theta_{\pm}}$ is finer than those 
 for $\chi_{\theta}$. 
 So we have KN stratifications 
 for $\mM_{Q}^{\dag, \theta \sss}(v)$ with respect to 
 $\chi_{\theta_{\pm}}$
 \begin{align}\label{KN:Spm}
 	\mM_Q^{\dag, \theta \sss}(v)=
 	\sS_1^{\pm} \sqcup \cdots \sqcup \sS_{N^{\pm}}^{\pm} \sqcup 
 	\mM_Q^{\dag, \theta_{\pm} \sss}(v)
 	\end{align}
 with associated one parameter subgroups $\lambda_i^{\pm} \colon \C \to \GL(V_0) \times \GL(V_1)$
 and the associated number $\eta_i^{\pm} \in \mathbb{Z}$ as in (\ref{etai}). 
 By Theorem~\ref{thm:window:MF} (and also noting Lemma~\ref{lem:bigger}), 
 for each choice of real numbers 
 $m_{\bullet}^{\pm}=\{(m_i^{\pm})\}_{1\le i\le N^{\pm}}$
 we have the subcategories
 \begin{align}\label{window:Q}
 	\mathbb{W}_{m_{\bullet}^{\pm}}^{\theta_{\pm}}(v) \subset 
 	\MF(\mM_Q^{\dag, \theta \sss}(v), w)
 	\end{align}
 such that the compositions 
 \begin{align}\label{equiv:glob}
 		\mathbb{W}_{m_{\bullet}^{\pm}}^{\theta_{\pm}}(v) \hookrightarrow 
 	\MF(\mM_Q^{\dag, \theta \sss}(v), w)
 	\twoheadrightarrow \MF(\mM_Q^{\dag, \theta_{\pm} \sss}(v), w)
 	\end{align}
 are equivalences. 
 The subcategory (\ref{window:Q}) consists 
 of objects whose 
 $\lambda_{i}^{\pm}$-weights at 
 each center of $\sS_i^{\pm}$ 
 are contained in $[m_i^{\pm}, m_i^{\pm}+\eta_i^{\pm})$. 
 
 We define the following character
 \begin{align*}
 	\chi_0 \colon \GL(V_0) \times \GL(V_1) \to \C, \ (g_0, g_1) \mapsto \det(g_0) \cdot \det(g_1)^{-1}
 	\end{align*}
 i.e. $\chi_0=\chi_{(-1, 1)}$ in the notation (\ref{chi:theta}). 
 As we discussed in (\ref{compose:Gp}), the composition 
 \begin{align*}
 	G_p=\GL(V) \hookrightarrow \GL(V_0) \times \GL(V_1) \stackrel{\chi_0}{\to}
 	\mathbb{C}^{\ast}
 	\end{align*}
 coincides with the determinant character 
 $\chi_0 \colon \GL(V) \to \mathbb{C}^{\ast}$. 
 We use the following special choices for $m_{\bullet}^{\pm}$
 \begin{align}\label{choice:m}
 	m_i^{+} =-\frac{1}{2} \eta_i^+ +\left(\frac{1}{2}C_{v, m}+
 	\frac{m}{2}\right)
 	\langle \lambda_i^+, \chi_0 \rangle, \
 		m_i^{-} =-\frac{1}{2} \eta_i^- +\frac{1}{2}C_{v, m}\langle \lambda_i^-, \chi_0 \rangle. 
 	\end{align}
 Here $C_{v, m}$ is given in (\ref{compute:abc}). 
 We then define 
 \begin{align}\label{window:glob}
 	\mathbb{W}_{\mathrm{glob}}^{\theta_{\pm}}(v) \subset \MF(\mM_Q^{\dag, \theta \sss}(v), w)
 	\end{align}
  to be 
  the window subcategories (\ref{window:Q}) for the choices of $m_{\bullet}^{\pm}$ 
  as (\ref{choice:m}). 
  
  \subsubsection*{Local window subcategories $\mathbb{W}_{\rm{loc}}^{\theta_{\pm}}(v)_p$}
  Let us take a $\theta$-polystable object $R$ as in (\ref{pstableR}), 
  and the corresponding closed point 
  $p \in M_Q^{\dag, \theta \sss}(v)$. 
  Then we have the diagram of formal fibers (\ref{ffiber}). 
  By restricting the KN stratification (\ref{KN:Spm})
  to the formal fiber, 
  we obtain the KN stratification 
  of $\widehat{\mM}_{Q}^{\dag, \theta \sss}(v)_p$
  \begin{align}\label{KN:Spm0}
  	\widehat{\mM}_{Q}^{\dag, \theta \sss}(v)_p=
  	\widehat{\sS}_{1, p}^{\pm} \sqcup \cdots \sqcup \widehat{\sS}_{N^{\pm}, p}^{\pm} \sqcup 
  		\widehat{\mM}_{Q}^{\dag, \theta_{\pm} \sss}(v)_p. 
  	\end{align}
  We define local window subcategories
  \begin{align}\notag
  	\mathbb{W}_{\rm{loc}}^{\theta_{\pm}}(v)_p \subset 
  	\MF(\widehat{\mM}_{Q}^{\dag, \theta \sss}(v)_p, w_p)
  \end{align}
  similarly to (\ref{window:glob}) as in Theorem~\ref{thm:window:MF}, 
  with respect to the KN stratifications (\ref{KN:Spm0})
  and the choices of $m_{\bullet}^{\pm}$ in (\ref{choice:m}). 
  The following lemma follows immediately from the definition of window 
subcategories: 
\begin{lem}\label{lem:globloc}
	An object $\eE \in \MF(\mM_Q^{\dag, \theta \sss}(v), w)$ is an object 
	in $\mathbb{W}_{\rm{glob}}^{\theta_{\pm}}(v)$ if and only if 
	for any closed point $p \in M_Q^{\dag, \theta \sss}(v)$ represented by 
	an object of the form (\ref{pstableR})
	we have $\eE|_{\widehat{\mM}_Q^{\dag, \theta \sss}(v)_p} \in 
	\mathbb{W}_{\rm{loc}}^{\theta_{\pm}}(v)_p$. 
	\end{lem}
\begin{proof}
	The defining conditions of window subcategories $\mathbb{W}_{\rm{glob}}^{\theta_{\pm}}(v)$ is local on 
	the good moduli space, so $\eE$ is an object in 
	$\mathbb{W}_{\rm{glob}}^{\theta_{\pm}}(v)$
	if and only if $\eE|_{\widehat{\mM}_Q^{\dag, \theta \sss}(v)_p} \in 
	\mathbb{W}_{\rm{loc}}^{\theta_{\pm}}(v)_p$
	for any $p \in M_{Q}^{\dag, \theta \sss}(v)$. 
		If $p$ is not represented by an object of the form (\ref{pstableR}), 
	then the formal fiber $\widehat{\mM}_Q^{\dag, \theta \sss}(v)_p$
	does not intersect with the critical locus of $w$, 
	so $\MF(\widehat{\mM}_Q^{\dag, \theta \sss}(v)_p, w_p)=0$.
	\end{proof}

\subsubsection*{Window subcategories
$\mathbb{W}^{\pm}(d)_p$ for the Ext-quiver}
  By pulling the KN stratification (\ref{KN:Spm0}) 
back to $\widehat{\mM}_{Q_p}^{\dag}(d)$ 
by the isomorphism $\eta_p$ in (\ref{ffiber}),  
we have the KN stratification of $\widehat{\mM}_{Q_p}^{\dag}(d)$
with respect to $\chi_0^{\pm 1}$
\begin{align}\label{KN:Spm2}
	\widehat{\mM}_{Q_p}^{\dag}(d)=
	\widetilde{\sS}_{1, p}^{\pm} \sqcup \cdots \sqcup \widetilde{\sS}_{N^{\pm}, p}^{\pm} \sqcup 
	\widehat{\mM}_{Q_p}^{\dag, \chi_0^{\pm 1}}(d). 
\end{align}
We define window subcategories
\begin{align}\notag
	\mathbb{W}^{\pm}(d)_p \subset 
	\MF(\widehat{\mM}_{Q_p}^{\dag}(d), w_p)
\end{align}
as in Theorem~\ref{thm:window:MF}, 
with respect to the KN stratifications (\ref{KN:Spm2})
and the choices of $m_{\bullet}^{\pm}$ in (\ref{choice:m}). 
By the isomorphism $\eta_p$ in (\ref{ffiber}), 
we have the equivalence 
\begin{align}\label{etap:equiv}
	\eta_p^{\ast} \colon \mathbb{W}_{\rm{loc}}^{\theta_{\pm}}(v)_p
	\stackrel{\sim}{\to} \mathbb{W}^{\pm}(d)_p. 
\end{align}

\subsubsection*{Window subcategories $\mathbb{W}_c(d)_p$
for the reduced Ext-quiver}
For $c \in \mathbb{Z}_{\ge 0}$, we also define 
\begin{align*}
	\mathbb{W}_c(d)_p \subset 
	\MF(\widehat{\mM}^{\dag}_{Q_p^{\rm{red}}}(d), w_p^{\rm{red}})
	\end{align*}
to be the thick closure of matrix factorizations whose entries are
of the form $V(\chi) \otimes \oO$ for $\chi \in \mathbb{B}_c(d)$, 
where $\mathbb{B}_c(d)$ is defined in (\ref{Bcd}). 
By the description (\ref{M:reduced}) of $\mM^{\dag}_{Q_p^{\rm{red}}}(d)$
in terms of the stack $\gG_{a_{v, m, d}, b_{v, m, d}}(d)$, 
the argument of Proposition~\ref{prop:WGD} (also see the argument of Corollary~\ref{cor:MF}) implies that the  
following composition functors are equivalences
\begin{align}\label{wind:localExt}
	&\mathbb{W}_{a_{v, m, d}}(d)_p
	\hookrightarrow \MF(\widehat{\mM}^{\dag}_{Q_p^{\rm{red}}}(d), w_p^{\rm{red}})
	\twoheadrightarrow \MF(\widehat{\mM}^{\dag, \chi_0 \sss}_{Q_p^{\rm{red}}}(d), w_p^{\rm{red}}), \\
	\notag&\mathbb{W}_{b_{v, m, d}}(d)_p
	\hookrightarrow \MF(\widehat{\mM}^{\dag}_{Q_p^{\rm{red}}}(d), w_p^{\rm{red}})
	\twoheadrightarrow \MF(\widehat{\mM}^{\dag, \chi_0^{-1}\sss}_{Q_p^{\rm{red}}}(d), w_p^{\rm{red}}). 
	\end{align}

\subsection{Comparison of window subcategories}
We compare the window subcategories in the previous subsection 
under the Kn\"{o}rrer periodicity: 
\begin{prop}\label{prop:Koszul}
	The equivalence (\ref{funct:Knoer})
restricts to the equivalences
\begin{align*}
&\Phi_p \colon \mathbb{W}_{a_{v, m, d}}(d)_p \otimes \chi_0^{d(m^2-m)}
\stackrel{\sim}{\to} \mathbb{W}^{+}(d)_p, \\
&	\Phi_p \colon \mathbb{W}_{b_{v, m, d}}(d)_p \otimes \chi_0^{d(m^2-m)}
\stackrel{\sim}{\to} \mathbb{W}^{-}(d)_p.
	\end{align*}
	\end{prop}
\begin{proof}
	We only give a proof for the $+$ part. 
	Let $\wW \to \widehat{\mM}_{Q_p^{\rm{red}}}^{\dag}(d)$
	be the vector bundle as in the diagram (\ref{dia:WW}). 
	The KN stratifications (\ref{KN:Spm}) are pull-back 
	of the KN stratifications
		\begin{align}\label{W:strata}
		\wW \oplus \wW^{\vee}=\overline{\sS}_1^{\pm} \sqcup \cdots 
		\sqcup \overline{\sS}_{N^{\pm}}^{\pm} \sqcup (\wW \oplus \wW^{\vee})^{\chi_0^{\pm 1} \sss} 
		\end{align}
	 of $\wW \oplus \wW^{\vee}$ 
	with respect to $\chi_0^{\pm 1}$
	by the morphism $\iota$ in (\ref{dia:WW}).
	We denote by 
	\begin{align*}
		\overline{\mathbb{W}}^{\pm}(d)_p \subset 
		\MF(\wW \oplus \wW^{\vee}, w_p^{\rm{red}}+q)
		\end{align*}
	the window subcategories with respect to the above 
	stratifications (\ref{W:strata}) and $m_{\bullet}^{\pm} \in \mathbb{R}$ 
	given by (\ref{choice:m}). 
	By the definition of the above window subcategories, 
	the equivalence (\ref{iota:ast2}) 
	restricts to the equivalence 
	\begin{align*}
		\iota^{\ast} \colon 
		\overline{\mathbb{W}}^{\pm}(d)_p 
		\stackrel{\sim}{\to} \mathbb{W}^{\pm}(d)_p.
		\end{align*}
	Therefore it is enough to show that 
	the equivalence (\ref{func:Kos}) restricts to the 
	equivalence 
	\begin{align*}
		i_{\ast}\mathrm{pr}^{\ast} \colon 
		\mathbb{W}_{a_{v, m, d}}(d)_p \otimes \chi_0^{d(m^2-m)} \stackrel{\sim}{\to}
		\overline{\mathbb{W}}^{+}(d)_p. 
		\end{align*}
	We have the following commutative diagram 
	\begin{align}\label{dia:descend}
		\xymatrix{
			\mathbb{W}_{a_{v, m, d}}(d)_p \otimes \chi_0^{d(m^2-m)}
			\ar@<-0.3ex>@{^{(}->}[r] \ar@{.>}[d] & \MF(\widehat{\mM}^{\dag}_{Q_p^{\rm{red}}}(d), w_p^{\rm{red}})
			\ar[r] \ar[d]_-{i_{\ast}\mathrm{pr}^{\ast}}^-{\sim} & 
			\MF(\widehat{\mM}^{\dag, \chi_0\sss}_{Q_p^{\rm{red}}}(d), w_p^{\rm{red}}) \ar@{.>}[d]^-{\sim}\\
			\overline{\mathbb{W}}^{+}(d)_p \ar@<-0.3ex>@{^{(}->}[r] & 
			\MF(\wW \oplus \wW^{\vee}, w_p^{\rm{red}}+q) \ar[r] & 
			\MF((\wW \oplus \wW^{\vee})^{\chi_0 \sss}, w_p^{\rm{red}}+q).	
		}
	\end{align}
The composition of top 
arrows is an equivalence by the equivalence in (\ref{wind:localExt}), 
and that of bottom arrows is also an equivalence by Theorem~\ref{thm:window:MF}.
We see that the middle vertical arrow descends to an equivalence of right vertical dotted 
arrow. 
Note that we have the isomorphism 
\begin{align}\notag
	\Crit(w_p^{\rm{red}}) \cap \widehat{\mM}_{Q_p^{\rm{red}}}^{\dag, \chi_0 \sss}
	\stackrel{\cong}{\hookrightarrow}
	\mathrm{Crit}(w_p^{\rm{red}}+q)
	\cap (\wW \oplus \wW^{\vee})^{\chi_0 \sss}
	\end{align}
induced by the zero section 
$\widehat{\mM}_{Q_p^{\rm{red}}}^{\dag}(d) \hookrightarrow 
\wW \oplus \wW^{\vee}$. 
In particular we have the inclusion 
\begin{align}
\label{inc:W}
\mathrm{Crit}(w_p^{\rm{red}}+q)
\cap (\wW \oplus \wW^{\vee})^{\chi_0 \sss}
	\subset (\wW \oplus \wW^{\vee}) \times_{\widehat{\mM}^{\dag}_{Q_p^{\rm{red}}}(d)}\widehat{\mM}^{\dag, \chi_0\sss}_{Q_p^{\rm{red}}}(d).
\end{align}
The desired equivalence is given by 
the composition 
\begin{align*}
	\MF(\widehat{\mM}^{\dag, \chi_0\sss}_{Q_p^{\rm{red}}}(d), w_p^{\rm{red}})
	&\stackrel{\sim}{\to} 
	\MF(
	 (\wW \oplus \wW^{\vee}) \times_{\widehat{\mM}^{\dag}_{Q_p^{\rm{red}}}(d)}\widehat{\mM}^{\dag, \chi_0\sss}_{Q_p^{\rm{red}}}(d), w_p^{\rm{red}}+q) \\
	 &\stackrel{\sim}{\to} 
	 \MF((\wW \oplus \wW^{\vee})^{\chi_0 \sss}, w_p^{\rm{red}}+q).
	\end{align*}
Here the first equivalence is Kn\"{o}rrer periodicity in Theorem~\ref{thm:period},
and the second equivalence follows from (\ref{inc:W}) and the 
equivalence (\ref{Pre:rest}). 

Therefore it is enough to show 
that the middle vertical arrow in (\ref{dia:descend}) restricts to the 
left dotted arrow, i.e. for $\pP \in \mathbb{W}_{a_{v, m, d}}(d)_p \otimes \chi_0^{d(m^2-m)}$, 
we show that
the object $i_{\ast}\mathrm{pr}^{\ast}(\pP)$ 
lies in $\overline{\mathbb{W}}^+(d)_p$. 
Note that the critical locus of $w_p^{\rm{red}}+q$ 
lies in the zero section 
$\widehat{\mM}^{\dag}_{Q_p^{\rm{red}}}(d) \subset \wW \oplus \wW^{\vee}$. 
From Theorem~\ref{thm:window:MF}, 
it is enough to that $i_{\ast}\mathrm{pr}^{\ast}(\pP)$ 
satisfies the condition (\ref{cond:P})
for one parameter subgroups 
which appear in the KN stratification of  
 $\widehat{\mM}^{\dag}_{Q_p^{\rm{red}}}(d)$. 
From the
description (\ref{M:reduced}) of $\mM^{\dag}_{Q_p^{\rm{red}}}(d)$, 
its KN stratifications with respect to $\chi_0^{\pm 1}$
are KN stratifications of $\gG_{a_{v, m, d}, b_{v, m, d}}(d)$
discussed in Subsection~\ref{subsec:Gflip}, 
 up to a 
product with a trivial factor. 
Therefore they are of the form  
\begin{align*}
	\widehat{\mM}^{\dag}_{Q_p^{\rm{red}}}(d)=\sS_0^{\pm}
	\sqcup \cdots \sqcup \sS_{d-1}^{\pm} \sqcup 
	\widehat{\mM}^{\dag, \chi_0^{\pm 1}\sss}_{Q_p^{\rm{red}}}(d)
	\end{align*}
such that each associated one parameter subgroup $\lambda_i^{\pm} \colon \C \to G_p=\GL(V)$ is 
given by (\ref{lambdai}), i.e. $\lambda_i^+$ is 
\begin{align}\label{lambdai2}
	\lambda_i^{+}(t)=(\overbrace{1, \ldots, 1}^i, \overbrace{t^{-1}, \ldots, t^{-1}}^{d-i}). 
\end{align}
Therefore in order to show 
that the object $i_{\ast}\mathrm{pr}^{\ast}(\pP)$ 
lies in $\overline{\mathbb{W}}^{+}(d)_p$, 
	it is enough to check the weight conditions (\ref{cond:P})
for the above $\lambda_i^{+}$. 

Since the object $i_{\ast}\mathrm{pr}^{\ast}(\pP)$ is 
given by taking the tensor product with the Koszul factorization (\ref{Koszul}), 
it is 
isomorphic 
to a direct summand of a matrix factorization whose entries 
are of the form 
\begin{align*}
	V(\chi) \otimes \bigwedge^k W \otimes 
\chi_0^{d(m^2-m)} \otimes \oO, \ 
\chi  \in \mathbb{B}_{a_{v, m, d}}(d), \ 0\le k\le \dim W. 
\end{align*}
For each one parameter subgroup $\lambda \colon \C \to G_p$, 
we set 
\begin{align*}
	\gamma_{\lambda} \cneq \langle \lambda, W^{\lambda>0} \rangle 
=-\langle \lambda, W^{\lambda<0} \rangle,
\end{align*} where 
the second identity holds as $W=\End_0(V) \otimes \mathbb{C}^{m^2-m}$
 is a self-dual $G_p$-representation. 
Then 
we have the following inclusions of the 
set of $\lambda_i^+$-weights of $V(\chi) \otimes \bigwedge^k W \otimes \chi_0^{d(m^2-m)}$:
\begin{align}\notag
	&\mathrm{wt}_{\lambda_i^+}(V(\chi) \otimes \bigwedge^k W\otimes 
	 \chi_0^{d(m^2-m)}) \\
	&\subset 
	\notag
	\bigcup_{\chi' \in \mathrm{wt}(V(\chi))}
	\left[ -\sum_{j=i+1}^d x_j'-(d-i) \cdot d(m^2-m)-\gamma_{\lambda_i^+}, 
	  -\sum_{j=i+1}^d x_j'-(d-i) \cdot d(m^2-m)+\gamma_{\lambda_i^+} \right] \\
	  &\notag \subset 
	  \left[ (d-i)(-a_{v, m, d}+d-dm^2+dm)-\gamma_{\lambda_{i}^+}, 
	  (d-i)(-dm^2+dm)+\gamma_{\lambda_i^+} \right]. 
	\end{align}
Here $\mathrm{wt}(V(\chi))$ is the set of $T$-weights 
of $V(\chi)$ for the maximal torus $T \subset G_p$, 
and we have written $\chi' \in \mathrm{wt}(V(\chi))$ as 
$\chi'=(x_1', \ldots, x_d')$
satisfying $0\le x_j' \le a_{v, m, d}-d$. 

We show that the above set of weights is contained in 
$[m_i^+, m_i^+ +\eta_i^+)$. 
From the decomposition (\ref{Ext:decom}), 
the $\eta_i^{+} \in \mathbb{Z}$ which appears in (\ref{choice:m}) for 
the one parameter subgroup (\ref{lambdai2})
is calculated as in the proof of Proposition~\ref{prop:WGD}:
\begin{align*}
	\eta_i^+ &= \langle \lambda_i^+, (\Ext_{Q^{\dag}}^1(R, R)^{\vee})^{\lambda_i^+>0} - (\mathfrak{g}_p^{\vee})^{\lambda_i^+>0}
	\rangle \\
	&=\langle \lambda_i^+, ((V^{\vee})^{\oplus a_{v, m, d}} \oplus 
	V^{\oplus b_{b, m, d}} \oplus W \oplus W^{\vee})^{\lambda_i^+ > 0}
	-\End(V)^{\lambda_i^+>0} \rangle \\
	&=(a_{v, m, d}-i)(d-i)+2\gamma_{\lambda_i^+}. 
	\end{align*}
Here $\mathfrak{g}_p=\End(V)$ is the Lie algebra of $G_p=\GL(V)$. 
Therefore we have 
\begin{align*}
	& \left[ m_i^+, m_i^+ +\eta_i^{+}   \right) \\
	&=\left[-\frac{1}{2}\eta_i^+ +\left(\frac{1}{2}C_{v, m}+\frac{m}{2}\right) \langle \lambda_i^+, \chi_0 \rangle, \frac{1}{2}\eta_i^+ +\left(\frac{1}{2}C_{v, m}+\frac{m}{2}\right) \langle \lambda_i^+, \chi_0 \rangle\right) \\
	&=\left[(d-i)\left(-a_{v, m, d}+\frac{i}{2}+\frac{d}{2}-dm^2+dm \right)-\gamma_{\lambda_i^+},  (d-i)\left(-dm^2+dm+\frac{d}{2}-\frac{i}{2} \right) + \gamma_{\lambda_i^+}\right).
	\end{align*}
Since $0\le i\le d-1$, 
we conclude the inclusion 
\begin{align*}
	\mathrm{wt}_{\lambda_i^+}(V(\chi) \otimes \bigwedge^k W\otimes 
	\chi_0^{d(m^2-m)})
	 \subset [m_i^+, m_i^+ +\eta_i^+).
	\end{align*}
Therefore 
the weight condition (\ref{cond:P})
for $i_{\ast}\mathrm{pr}^{\ast}\pP$  
with respect to $\lambda_i^{+}$ is satisfied. 
\end{proof}

Let $s_m=(m, m-1)$ be the dimension vector of 
the stable $Q$-representation $S_m$, defined in (\ref{def:Tm}). 
Let $q_m \in M_Q^{\theta \sss}(s_m)$ be the corresponding 
closed point. 
We consider the formal fiber of the 
good moduli space morphism (\ref{MQ:sm}) at $q_m$
\begin{align*}
	\widehat{\mM}_Q^{\theta \sss}(s_m) \to \widehat{M}_Q^{\theta \sss}(s_m). 
	\end{align*}
Similarly to (\ref{dia:Luna}), the \'{e}tale slice theorem 
implies an isomorphism 
\begin{align}\label{eslice:M}
	\widehat{\mM}_{Q_p}^{\dag}(1)=\left[\widehat{\Ext}^1_{Q}(S_m, S_m)/\Aut(S_m)\right] \stackrel{\cong}{\to} \widehat{\mM}_{Q}^{\theta \sss}(s_m). 	
	\end{align} 
Here $\Aut(S_m)=\C$ acts on $\Ext_Q^1(S_m, S_m)$ trivially. 
We will also use the following lemma, which compares 
window subcategories for quivers without framings: 
\begin{lem}\label{lem:Kn}
	For any $j \in \mathbb{Z}$, we have equivalences
	\begin{align}\label{equiv:sm}
				\MF(\widehat{\mM}_{Q_p^{\rm{red}}}(1), w_p^{\rm{red}})_j
		\stackrel{\sim}{\to} \MF(\widehat{\mM}_{Q_p}(1), w_p)_j
		\stackrel{\sim}{\to}\MF(\widehat{\mM}_Q^{\theta\sss}(s_m), w)_{j}, 
		\end{align}
	and all of them are equivalent to $\MF(\Spec \mathbb{C}, 0)$. 
	Here the first equivalence 
	is given by the Kn\"{o}rrer periodicity in Theorem~\ref{thm:period}. 
		\end{lem}
\begin{proof}
By the definition of $Q_p^{\rm{red}}$ in (\ref{def:Qpred}), we have 	
$(\widehat{\mM}_{Q_p^{\rm{red}}}(1), w_p^{\rm{red}})=
	(B\C, 0)$.
	On the other hand, 
	the isomorphisms (\ref{isom:W}), (\ref{eslice:M}) and an argument of Proposition~\ref{prop:func}
	imply an isomorphism 
\begin{align}\label{isom:Msm}
		(\widehat{\mM}_{Q_p}(1), w_p)\cong 
		\left(\left[\widehat{(H \oplus H^{\vee})}/\C \right], q \right)
		\end{align}
where $\C$ acts on $H=\mathbb{C}^{m^2-m}$ trivially
and $q$ is the natural paring on $H$ and its dual. 
By the 
Kn\"{o}rrer periodicity in Theorem~\ref{thm:period}, we have 
an equivalence 
\begin{align*}
	\MF(\widehat{\mM}_{Q_p^{\rm{red}}}(1), w_p^{\rm{red}})_j 
	=\MF(\Spec \mathbb{C}, 0)
	\stackrel{\sim}{\to} \MF(H\oplus H^{\vee}, q). 
	\end{align*}
The natural functor by the formal completion 
\begin{align*}
	\MF(H \oplus H^{\vee}, q)
	\to \MF(\widehat{H \oplus H^{\vee}}, q)
	\end{align*}
is an equivalence (see~\cite[Remark~2.18]{Brown}). 
Therefore we obtain the desired equivalences (\ref{equiv:sm}). 
	\end{proof}

\subsection{Comparison of Hall products}
As in the previous subsections, we take 
a stability condition on the wall 
$\theta \in W_m$ for $m \ge 1$. 
As in Subsection~\ref{subsec:catHw}, 
we have the categorified Hall product 
 \begin{align}\label{Hahat-1}
\MF(\mM_Q^{\theta \sss}(s_m), w)_{j_1}
	\boxtimes \cdots \boxtimes 	\MF(\mM_Q^{\theta \sss}(s_m), w)_{j_l}
	\boxtimes &\MF(\mM_Q^{\dag, \theta \sss}(v-l s_m), w) \\
\notag	&\to \MF(\mM_Q^{\dag, \theta \sss}(v), w). 
	\end{align}
Here $s_m=(m, m-1)$ is the dimension vector of $S_m$. 
We take a $\theta$-polystable representation $Q^{\dag}$-representation
$R$ of the form (\ref{pstableR}), 
i.e. $R=R_{\infty} \oplus (V \otimes S_m)$
with $\dim V=d$, 
and the corresponding closed point $p \in M_{Q}^{\dag, \theta \sss}(v)$. 
By taking the base change of the above 
categorified Hall product 
to the formal completion at $p$ (see Subsection~\ref{subsec:bchange}),   
we obtain the functor 
\begin{align}\label{Hahat0}
	\MF(\widehat{\mM}_Q^{\theta \sss}(s_m), w)_{j_1}
	\boxtimes \cdots \boxtimes 	\MF(\widehat{\mM}_Q^{\theta \sss}(s_m), w)_{j_l}
	\boxtimes &\MF(\widehat{\mM}_Q^{\dag, \theta \sss}(v-l s_m)_{p_l}, w) \\
\notag	&\to \MF(\widehat{\mM}_Q^{\dag, \theta \sss}(v)_p, w). 
\end{align}
Here $p_l \in M_{Q}^{\dag, \theta \sss}(v-l s_m)$ 
corresponds to the $\theta$-polystable representation 
$R_{\infty} \oplus (V'\otimes S_m)$
with $\dim V'=d-l$. 
We note that by the isomorphism $\eta_p$ in (\ref{ffiber})
and the isomorphism (\ref{eslice:M}), 
the above functor is identified with 
the functor 
\begin{align}\label{Hahat1}
	\MF(\widehat{\mM}_{Q_p}(1), w_p)_{j_1} \boxtimes \cdots \boxtimes 
	\MF(\widehat{\mM}_{Q_p}(1), w_p)_{j_l} \boxtimes 
	&\MF(\widehat{\mM}_{Q_p}^{\dag}(d-l), w_p) \\
\notag	& \to 
	\MF(\widehat{\mM}_{Q_p}^{\dag}(d), w_p)
	\end{align}
obtained by the categorified Hall products for $Q_p^{\dag}$-representations 
and the completions at the origins. 
A similar construction also gives the categorified Hall product 
for $Q_p^{\rm{red}, \dag}$-representations
\begin{align}\label{Hahat2}
	\MF(\widehat{\mM}_{Q_p^{\rm{red}}}(1), 0)_{j_1} \boxtimes \cdots \boxtimes 
	\MF(\widehat{\mM}_{Q_p^{\rm{red}}}(1), 0)_{j_l} \boxtimes 
	&\MF(\widehat{\mM}_{Q_p^{\rm{red}}}^{\dag}(d-l), w_p^{\rm{red}}) \\
\notag	 &\to 
	\MF(\widehat{\mM}_{Q_p^{\rm{red}}}^{\dag}(d), w_p^{\rm{red}}).
\end{align}
We compare the above categorified Hall products under the
Kn\"{o}rrer periodicity: 
\begin{prop}\label{prop:compare}
	The following diagram commutes
	\begin{align}\label{dia:Hcompare}
		\xymatrix{
\boxtimes_{i=1}^l 	\MF(\widehat{\mM}_{Q_p^{\rm{red}}}(1), 0)_{j_i} \boxtimes
\MF(\widehat{\mM}_{Q_p^{\rm{red}}}^{\dag}(d-l), w_p^{\rm{red}}) \ar[r] \ar[d] & 
\MF(\widehat{\mM}_{Q_p^{\rm{red}}}^{\dag}(d), w_p^{\rm{red}}) \ar[d] \\
\boxtimes_{i=1}^l	\MF(\widehat{\mM}_{Q_p}(1), w_p)_{j_i+(2i-d-1)(m^2-m)}
\boxtimes \MF(\widehat{\mM}_{Q_p}^{\dag}(d-l), w_p) \ar[r] &
\MF(\widehat{\mM}_{Q_p}^{\dag}(d), w_p).
}
		\end{align}
	Here the horizontal arrows are given by 
	categorized Hall products (\ref{Hahat1}), (\ref{Hahat2}), 
	 the right vertical arrow is 
	given in Proposition~\ref{prop:funct}, and 
	the left vertical arrow is a composition of 
	the functors in Proposition~\ref{prop:funct}, Lemma~\ref{lem:Kn} 
	with the equivalence 
	\begin{align}\label{factor:tensor}
		\boxtimes_{i=1}^l &\otimes \oO_{B\C}((2i-d-1)(m^2-m)) \boxtimes \otimes \chi_0^{l(m^2-m)}
		\left[\left(dl-\frac{l}{2}-\frac{l^2}{2}\right)(m^2-m) \right]
		\colon \\
		\notag&\boxtimes_{i=1}^l 	\MF(\widehat{\mM}_{Q_p^{\rm{red}}}(1), 0)_{j_i} \boxtimes
		\MF(\widehat{\mM}_{Q_p^{\rm{red}}}^{\dag}(d-l), w_p^{\rm{red}}) \\
		\notag&\stackrel{\sim}{\to}
		 \boxtimes_{i=1}^l 	\MF(\widehat{\mM}_{Q_p^{\rm{red}}}(1), 0)_{j_i+(2i-d-1)(m^2-m)}\boxtimes
		\MF(\widehat{\mM}_{Q_p^{\rm{red}}}^{\dag}(d-l), w_p^{\rm{red}}).
		\end{align}
	\end{prop}
\begin{proof}
We take $\lambda \colon \C \to G_p=\GL(V)$
	by 
	\begin{align*}
		\lambda(t)=(t^l, t^{l-1}, \ldots, t, \overbrace{1, \ldots, 1}^{d-l}). 
		\end{align*}
	Then the top arrow in the diagram (\ref{dia:Hcompare})
	is obtained from the diagram of attracting loci 
	for $\widehat{\mM}_{Q_p^{\rm{red}}}^{\dag}(d)$
	with respect to the above $\lambda$. 
	Recall that in the diagram (\ref{dia:WW}), the vector bundle 
	 $\wW \to \widehat{\mM}_{Q_p^{\rm{red}}}^{\dag}(d)$ is 
	induced by the $\GL(V)$-representation $W=\End_0(V) \otimes H$
	for $H=\mathbb{C}^{m^2-m}$ by its definition. 
By Proposition~\ref{prop:Knoer}, the categorified Hall products
	in (\ref{dia:Hcompare})
	commute with Kn\"{o}rrer periodicity equivalences
	up to the following equivalence 
	\begin{align}\label{factor:tensor2}
		\otimes  \det (\End_0(V)^{\lambda >0} & \otimes H)^{\vee}
		[\dim (\End_0(V)^{\lambda >0} \otimes H)]
		 \colon \\
		 	\notag \boxtimes_{i=1}^l &	\MF(\widehat{\mM}_{Q_p^{\rm{red}}}(1), 0) \boxtimes
		 \MF(\widehat{\mM}_{Q_p^{\rm{red}}}^{\dag}(d-l), w_p^{\rm{red}}) \\
		 \notag & \stackrel{\sim}{\to}
		 \boxtimes_{i=1}^l 	\MF(\widehat{\mM}_{Q_p^{\rm{red}}}(1), 0)\boxtimes
		 \MF(\widehat{\mM}_{Q_p^{\rm{red}}}^{\dag}(d-l), w_p^{\rm{red}}).		 
		\end{align}
	
It is enough to show that the equivalence (\ref{factor:tensor2}) 
	restricts to the equivalence (\ref{factor:tensor}). 	
	Let 
	$V=\oplus_{i=0}^l V_i$ be the decomposition 
	into $\lambda$-weight part, i.e. 
	$V_i$ has $\lambda$-weight $i$ so that 
	$\dim V_i=1$ for $1\le i\le l$ and 
	$\dim V_{0}=d-l$. 
	We have 
	\begin{align*}
		\End_0(V, V)^{\lambda>0} =
		\left(\bigoplus_{0\le i<j\le l} V_i^{\vee} \otimes V_j\right). 
		\end{align*}
		We compute that 
	\begin{align*}
		\det \left(\End_0(V, V)^{\lambda >0}\right)^{\vee} &=
		\bigotimes_{0\le i<j\le l} \det(V_i \otimes V_j^{\vee}) \\
		&=\bigotimes_{1\le j\le l} \det(V_0 \otimes V_j^{\vee})
		\otimes \bigotimes_{1\le i<j\le l} \det(V_i \otimes V_j^{\vee}) \\
		&=(\det V_0)^l \otimes 
		\bigotimes_{i=1}^l (\det V_i)^{2l-2i+1-d}. 
		\end{align*}
	We note that 
	$\otimes \det V_i=\otimes \oO_{B\C}(1)$ on 
	the factor 
	$\MF(\widehat{\mM}_{Q_p^{\rm{red}}}(1), 0)_{j_{l-i+1}}$. 
	We also have 
	\begin{align*}
		\dim \End_0(V, V)^{\lambda>0}=dl-\frac{l}{2}-\frac{l^2}{2}. 
		\end{align*}
	Therefore the equivalence (\ref{factor:tensor2}) restricts to the equivalence (\ref{factor:tensor}). 
	\end{proof}
\subsection{Semiorthogonal decomposition of global window subcategories}
The following is the main result in this section:  
\begin{thm}\label{thm:wincon}
	For 
	 $l\ge 0$ and $0\le j_1 \le \cdots \le j_l \le m-l$, 
	 the categorified Hall product (\ref{Hahat-1}) restricts 
	 to the fully-faithful functor 
	 \begin{align}\label{FF:global}
	 	\Upsilon_{j_{\bullet}} \colon 
	 \boxtimes_{i=1}^l \MF(\mM_Q^{\theta\sss}(s_m), w)_{j_i+(2i-1)(m^2-m)}
	\boxtimes \left(\mathbb{W}_{\rm{glob}}^{\theta_-}(v-ls_m) \otimes \chi_0^{j_l+2l(m^2-m)}  \right) 
	\to 	\mathbb{W}_{\rm{glob}}^{\theta_+}(v)
	\end{align}
such that, by setting $\cC_{j_{\bullet}}$ to be the essential image of the above functor $\Upsilon_{j_{\bullet}}$,  
we have the semiorthogonal decomposition 
	\begin{align}\label{sod:global}
		\mathbb{W}_{\rm{glob}}^{\theta_+}(v)=
		\langle \cC_{j_{\bullet}} : l\ge 0, 0\le j_1 \le \cdots \le j_l \le m-l \rangle, 
			\end{align}
		where $\Hom(\cC_{j_{\bullet}}, \cC_{j_{\bullet}'})=0$
	for $j_{\bullet} \succ j_{\bullet}'$
	(see Definition~\ref{defi:orderj}). 
	\end{thm}
\begin{proof}
	We take a $\theta$-polystable representation $R$ of the form 
	(\ref{pstableR}), i.e. $R=R_{\infty} \oplus (V \otimes S_m)$
	with $\dim V=d$, 
	the corresponding closed point $p\in M_{Q}^{\dag, \theta \sss}(v)$, 
	and consider the quivers $Q_p^{\dag}$, $Q_p^{\dag, \rm{red}}$ as in 
	the previous subsections. 
	Note that if we remove the loops at the vertex $\{\infty\}$
	from $Q_p^{\dag}$, then 
	we obtain the quiver $Q_{a, b}$ for $a=a_{v, m, d}$
	and $b=b_{v, m, d}$ considered in Remark~\ref{rmk:Gab}.  
	By applying Corollary~\ref{cor:MF}
 for
	the above $Q_{a, b}$, and then taking the tensor product 
with $\chi_0^{d(m^2-m)}$, 
we obtain the semiorthogonal decomposition 
\begin{align*}
	&\mathbb{W}_{a_{v, m, d}}(d)_p \otimes \chi_0^{d(m^2-m)}= \\
	&\left \langle \boxtimes_{i=1}^l \MF(\widehat{\mM}_{Q_p^{\rm{red}}}(1), 0)_{j_i+d(m^2-m)} \boxtimes \left(\mathbb{W}_{b_{v, m, d}}(d-l)_{p_l} \otimes \chi_0^{(d-l)(m^2-m)}
	\otimes \chi_0^{j_l+l(m^2-m)} \right)  \right \rangle. 
	\end{align*} 
	Here $l \ge 0$, $p_l \in M_Q^{\dag, \theta \sss}(v-l s_m)$
	corresponds to $R_{\infty} \oplus (V' \otimes S_m)$ with $\dim V'=d-l$,
	 and 
	\begin{align}\label{j:ab}
		0 \le j_1 \le \cdots \le j_l \le a_{v, m, d}-b_{v, m, d}-l
	=m-l.
	\end{align}
	By applying Proposition~\ref{prop:Koszul}, Lemma~\ref{lem:Kn}
	and 
	Proposition~\ref{prop:compare},  
	we obtain the semiorthogonal decomposition 
		\begin{align*}
		\mathbb{W}^{+}(d)_p=
		\left \langle \boxtimes_{i=1}^l \MF(\widehat{\mM}_{Q_p}(1), w_p)_{j_i+(2i-1)(m^2-m)} \boxtimes 
		\left(\mathbb{W}^-(d-l)_{p_l} \otimes \chi_0^{j_l+2l(m^2-m)}  
		\right) \right \rangle. 
	\end{align*}
	By the identification of categorified Hall products (\ref{Hahat0}) with (\ref{Hahat1})
	together with the equivalence (\ref{etap:equiv}), 
	we obtain the semiorthogonal decomposition 
	\begin{align}\label{sod:formal}
		&\mathbb{W}_{\rm{loc}}^{\theta_+}(v)_p= \\
		&\notag \left \langle \boxtimes_{i=1}^l \MF(\widehat{\mM}_Q^{\theta \sss}(s_m), w)_{j_i+(2i-1)(m^2-m)} \boxtimes 
		\left(\mathbb{W}_{\rm{loc}}^{\theta_-}(v-ls_m)_{p_l} \otimes \chi_0^{j_l+2l(m^2-m)}  
		\right) \right \rangle. 
		\end{align}
	A key observation is that in the above semiorthogonal 
	decomposition there is no term involving $d=\dim V$ (which
	depends on 
	a choice of $\theta$-polystable object (\ref{pstableR}))
	so that we can globalize it. Indeed we have 
	globally defined functors (\ref{FF:global}) and, noting 
	Lemma~\ref{lem:globloc}, in order to show that they are fully-faithful and 
	forms a semiorthogonal decomposition it is enough to check 
	these properties formally locally at each closed point 
	of $M_{Q}^{\dag, \theta\sss}(v)$ corresponding to 
	a $\theta$-polystable $(Q^{\dag}, W)$-representation (see the arguments 
	in~\cite[Proposition~6.9, Theorem~6.11]{Totheta} for example). 
	
	Here we give some more details for how to derive the global 
	semiorthogonal decomposition (\ref{sod:global})
	from the formal local one (\ref{sod:formal}). 
	We first note that the categorified Hall product (\ref{Hahat-1}) restricts to the functor (\ref{FF:global}). This follows from the fact that the categorified Hall products commute 
	with base change to the formal completion of good moduli 
	spaces
	(see the diagram (\ref{bchange:2})), the fact (which 
	follows from (\ref{sod:formal})) 
	that 
	formally locally over $M_Q^{\dag, \theta \sss}(v)$
	the categorified Hall product restricts to the functor 
	\begin{align*}
		 \boxtimes_{i=1}^l \MF(\widehat{\mM}_Q^{\theta \sss}(s_m), w)_{j_i+(2i-1)(m^2-m)} \boxtimes 
		\left(\mathbb{W}_{\rm{loc}}^{\theta_-}(v-ls_m)_{p_l} \otimes \chi_0^{j_l+2l(m^2-m)}  
		\right) 
		\to \mathbb{W}_{\rm{loc}}^{\theta_+}(v)_p
		\end{align*}
	and noting Lemma~\ref{lem:globloc}. 
	
By Lemma~\ref{lem:radjoint} below, the functor 
$\Upsilon_{j_{\bullet}}$ admits a right adjoint 
$\Upsilon_{j_{\bullet}}^R$. 
Now in order to show that $\Upsilon_{j_{\bullet}}$ is fully-faithful, 
it is enough to show 
that the adjunction morphism 
\begin{align*}
	(-) \to \Upsilon^R_{j_{\bullet}} \circ \Upsilon_{j_{\bullet}}(-)
	\end{align*}
 is an isomorphism. Equivalently, it is enough to show the cone of the above 
 morphism is zero. 
By Lemma~\ref{lem:locvani}, 
this is a property formally locally over $M_Q^{\dag, \theta \sss}(v)$. So from 
the semiorthogonal decomposition (\ref{sod:formal})
we conclude that 
$\Upsilon_{j_{\bullet}}$ is fully-faithful. 
A similar argument also shows that 
$\cC_{j_{\bullet}}$ for $j_{\bullet}$ 
given in (\ref{j:ab}) are semiorthogonal. 

In order to show that $\cC_{j_{\bullet}}$ for $j_{\bullet}$ 
given in (\ref{j:ab}) generate
$\mathbb{W}_{\rm{glob}}^{\theta_+}(v)$, 
let us take $\eE \in \mathbb{W}_{\rm{glob}}^{\theta_+}(v)$
and $j_{\bullet}$ so that $j_{\bullet}$ is maximal
in the order of Definition~\ref{defi:orderj}. 
We have the distinguished triangle 
\begin{align*}
	\Upsilon_{j_{\bullet}}\Upsilon_{j_{\bullet}}^R(\eE) \to \eE \to \eE', \ 
	\eE' \in \cC_{j_{\bullet}}^{\perp}. 
	\end{align*}
By applying the above construction for $\eE'$ and the second 
maximal $j_{\bullet}$ and repeating, 
we obtain the distinguished triangle 
\begin{align*}
\eE_1 \to \eE \to \eE_2, \ 
\eE_1 \in \langle \cC_{j_{\bullet}} \rangle, \ 
\eE_2 \in \langle \cC_{j_{\bullet}} \rangle^{\perp}. 
	\end{align*}
Here $\langle \cC_{j_{\bullet}} \rangle$
is the right hand side of (\ref{sod:global}). 
From the semiorthogonal decomposition (\ref{sod:formal}), 
we have $\eE_2|_{\widehat{\mM}_Q^{\dag, \theta \sss}(v)_p}=0$
for any closed point $p \in M_Q^{\dag, \theta \sss}(v)$, 
therefore $\eE_2=0$ by Lemma~\ref{lem:locvani}. 
	Therefore
	$\eE \in \langle \cC_{j_{\bullet}} \rangle$, 
	and we have the desired semiorthogonal decomposition (\ref{sod:global}). 
	\end{proof}

The following corollary, which is an immediate consequence from 
Theorem~\ref{thm:wincon}, 
categorifies wall-crossing formula of 
the associated DT invariants in~\cite{NN}.  
\begin{cor}\label{cor:sod}
	There exists a semiorthogonal decomposition of the form 
	\begin{align*}
		\MF(\mM_{Q}^{\dag, \theta_+}(v), w)=
		\left \langle \MF(\mM_Q^{\dag, \theta_{-}}(v-l s_m), w)_{j_{\bullet}} : 
		l \ge 0, 0\le j_1 \le \cdots \le j_l \le m-l  \right \rangle. 
		\end{align*}
	Here $\MF(\mM_Q^{\dag, \theta_{-}}(v-l s_m), w)_{j_{\bullet}}$
	is a copy of $\MF(\mM_Q^{\dag, \theta_{-}}(v-l s_m), w)$. 
		\end{cor}
	\begin{proof}
		By the equivalences (\ref{equiv:glob}), the LHS of (\ref{sod:global}) is equivalent 
		to $\MF(\mM_{Q}^{\dag, \theta_+}(v), w)$. 
		On the other hand, the subcategory $\cC_{j_{\bullet}}$ 
		in (\ref{sod:global}) is equivalent to $\MF(\mM_Q^{\dag, \theta_{-}}(v-l s_m), w)$
		by the equivalences (\ref{equiv:glob}) together with Lemma~\ref{lem:crit:simple}. 
		\end{proof}

	\begin{rmk}\label{rmk:recover0}
		The semiorthogonal decomposition in Corollary~\ref{cor:sod} recovers 
	the numerical wall-crossing formula (\ref{intro:wcf}).
	Indeed the periodic cyclic homologies are additive with 
	respect to semiorthogonal decompositions 
	(see~\cite[Theorem~6.3, Section~6.1]{Tab}), so we have 
	\begin{align*}
		\mathrm{HP}_{\ast}(\MF(\mM_Q^{\dag, \theta_+\sss}(v), w))
		=\bigoplus_{l\ge 0}
		\mathrm{HP}_{\ast}(\MF(\mM_Q^{\dag, \theta_-\sss}(v), w))^{\oplus 
			\binom{m}{l}}. 
	\end{align*}
	By taking the Euler characteristics and using Lemma~\ref{lem:pcyc}, we obtain 
	the formula (\ref{intro:wcf}). 
	\end{rmk}

By applying Corollary~\ref{cor:sod} from the empty chamber in 
Figure~\ref{figure4} to 
the wall-crossing at $W_m$, and noting Lemma~\ref{lem:catwall}, 
we obtain the following: 
\begin{cor}\label{cor:sod2}
	For $\theta \in W_m$, 
	there exists a semiorthogonal decomposition 
	\begin{align*}
		\MF(\mM_Q^{\dag, \theta+}(v), w)
		=\left\langle \cC_{j_{\bullet}^{(\ast)}} \right\rangle .
		\end{align*}
	Here each $\cC_{j_{\bullet}^{(\ast)}}$ is equivalent to $\MF(\Spec \mathbb{C}, 0)$
	and $j_{\bullet}^{(\ast)}$ is a collection of non-positive 
	integers of the form 
	\begin{align*}
		j_{\bullet}^{(\ast)}=\{(0 \le j_1^{(i)} \le \cdots \le j_{l_i}^{(i)} \le i-l_i)\}_{1\le i\le m}
		\end{align*} for some integers
	$l_i \ge 0$ satisfying 
	\begin{align*}
		(v_0, v_1)=\sum_{i=1}^m l_i \cdot (i, i-1). 
		\end{align*}
	We have $\Hom(\cC_{j_{\bullet}^{(\ast)}}, \cC_{j_{\bullet}^{'(\ast)}})=0$ 
	if $j_{\bullet}^{(i)}=j_{\bullet}^{'(i)}$
	for $k< i\le m$ for some $k$ and 
	$j_{\bullet}^{(k)} \succ j_{\bullet}^{'(k)}$. 
	\end{cor}
\begin{proof}
	Let $\theta_{\rm{en}} \in \mathbb{R}^2$ lie 
	in the empty chamber in Figure~\ref{figure4}. 
	By Lemma~\ref{lem:catwall}, a
	successive application of
	 Corollary~\ref{cor:sod} 
	 gives the semiorthogonal decomposition 
	 \begin{align*}
	 	\MF(\mM_{Q}^{\dag, \theta_+}(v), w)=
	 \left \langle \MF(\mM_Q^{\dag, \theta_{\rm{en}}}(v-l_m s_m-l_{m-1}s_{m-1}
	 - \cdots -l_1s_1), w)_{j^{(m)}_{\bullet}, j^{(m-1)}_{\bullet}, \ldots, 
	 j^{(1)}_{\bullet}} \right\rangle	
	 	\end{align*}
 	Here $l_i \ge 0$ are integers and $0\le j_1^{(1)} \le \cdots 
 	\le j_{l_i}^{(i)} \le i-l_i$ for $1\le i\le m$. 
 	By applying Lemma~\ref{lem:empty}, we obtain the corollary. 
		\end{proof}

\begin{rmk}\label{rmk:owall}
	The arguments of Theorem~\ref{thm:wincon} and Corollary~\ref{cor:sod}
	work for other walls except walls at $\{\theta_0+\theta_1=0\}$. 
	For example, let us consider the wall in Figure~\ref{figure4}
	\begin{align*}
		W_m' \cneq \mathbb{R}_{>0}(-m-1, m), \ m \in \mathbb{Z}_{\ge 0}. 
		\end{align*}
	Then for $\theta \in W_m'$, 
	there is a unique $\theta$-stable $(Q, W)$-representation $S_m'$
	of dimension vector $s_m'=(m, m+1)$,  
	which corresponds to $\oO_C(-m-1)[1]$ under the equivalence $\Phi$ in (\ref{equiv:Phi})
	(see~\cite[Remark~3.6]{NN}). 
	The arguments of Theorem~\ref{thm:wincon} and Corollary~\ref{cor:sod}
	work verbatim by replacing $S_m$, $s_m$ with $S_m'$, $s_m'$, so that 
	we have the semiorthogonal decomposition 
		\begin{align*}
		\MF(\mM_{Q}^{\dag, \theta_+}(v), w)=
		\left \langle \MF(\mM_Q^{\dag, \theta_{-}}(v-l s_m'), w)_{j_{\bullet}} : 
		l \ge 0, 0\le j_1 \le \cdots \le j_l \le m-l  \right \rangle. 
	\end{align*}

On the other hand, the above arguments do not work 
at walls in $\{\theta_0+\theta_1=0\}$. 
For example at the DT/PT wall $\theta \in \mathbb{R}_{>0}(-1, 1)$, 
there exist infinite number of $\theta$-stable 
$(Q, W)$-representations corresponding to closed points in 
$X$, and the associated Ext-quivers are more complicated. 
		\end{rmk}

\subsection{Semiorthogonal decompositions of categorical stable pair theory}\label{subsec:catPT}
By definition a \textit{PT stable pair}~\cite{PT} on $X$ is a pair $(F, s)$
where $F$ is a pure one dimensional coherent sheaf on $X$
and $s \colon \oO_X \to F$ is surjective in dimension one. 
For $(\beta, n) \in \mathbb{Z}^2$, 
we denote by 
\begin{align*}
	P_n(X, \beta)
	\end{align*}
the moduli space of PT stable pair moduli space
$(F, s)$ on $X$ satisfying $[F]=\beta[C]$ and $\chi(F)=n$, where 
$[F]$ is the fundamental one cycle of $F$.   
Since any such a sheaf $F$ is supported on $C$, 
the moduli space $P_n(X, \beta)$ is a projective scheme. 

It is proved in~\cite[Proposition~2.11]{NN} that
the equivalence (\ref{equiv:Phi})
induces the isomorphism 
\begin{align*}
	\Phi_{\ast} \colon P_n(X, \beta) \stackrel{\cong}{\to}
	\mM_{(Q, W)}^{\dag, \theta_{\rm{PT}}}(n, n-\beta)
	\end{align*}
where $\theta_{\rm{PT}} \cneq (-1+\varepsilon, 1+\varepsilon)$ for 
$0<\varepsilon \ll 1$. 
The RHS is the critical locus of the function 
$w \colon \mM_{Q}^{\dag, \theta_{\rm{PT}}}(n, n-\beta) \to \mathbb{A}^1$
defined by (\ref{func:w}). 
Based on the above isomorphism, 
the categorical PT invariant is defined as follows: 
\begin{defi}\label{defi:catPT}
	We define the categorical PT invariant for the 
	resolved conifold $X$ to be 
\begin{align*}
	\mathcal{DT}(P_n(X, \beta)) \cneq 
	\MF(\mM_Q^{\dag, \theta_{\rm{PT}}}(n, n-\beta), w). 
	\end{align*}
\end{defi}
Similarly to Lemma~\ref{lem:pcyc}, the categorical PT invariant 
recovers the numerical PT invariant by
\begin{align}\label{eqn:PT}
	P_{n, \beta}=(-1)^{n+\beta} e_{\mathbb{C} \lgakko u \rgakko}
	(\mathrm{HP}_{\ast}(\mathcal{DT}(P_n(X, \beta)))). 
	\end{align}
By applying Corollary~\ref{cor:sod2} for 
$m\gg 0$, we obtain the following: 
\begin{cor}\label{cor:sod2.5}
	For any $(\beta, n) \in \mathbb{Z}^2$, 
	there exists a semiorthogonal decomposition 
	\begin{align*}
	\mathcal{DT}(P_n(X, \beta))
		=\left\langle \cC_{j_{\bullet}^{(\ast)}} \right\rangle .
	\end{align*}
	Here each $\cC_{j_{\bullet}^{(\ast)}}$ is equivalent to $\MF(\Spec \mathbb{C}, 0)$
	and $j_{\bullet}^{(\ast)}$ is a collection of non-positive 
	integers of the form 
	\begin{align*}
		j_{\bullet}^{(\ast)}=\{(0 \le j_1^{(i)} \le \cdots \le j_{l_i}^{(i)} \le i-l_i)\}_{i\ge 1}
	\end{align*} for some integers 
	$l_i \ge 0$ satisfying 
	\begin{align*}
		(\beta, n)=\sum_{i\ge 1} l_i \cdot (1, i). 
	\end{align*}
	We have $\Hom(\cC_{j_{\bullet}^{(\ast)}}, \cC_{j_{\bullet}^{'(\ast)}})=0$ 
	if $j_{\bullet}^{(i)}=j_{\bullet}^{'(i)}$
	for $i>k$ for some $k$ and 
	$j_{\bullet}^{(k)} \succ j_{\bullet}^{'(k)}$. 
\end{cor}

\begin{rmk}\label{rmk:recover}
Similarly to Remark~\ref{rmk:recover0}, 
 the semiorthogonal decomposition in Corollary~\ref{cor:sod2.5}
 implies 
 \begin{align*}
 	\mathrm{HP}_{\ast}(\mathcal{DT}(P_n(X, \beta)))
 	=\mathrm{HP}_{\ast}(\MF(\Spec \mathbb{C}, 0))^{\oplus a_{n, \beta}}
 	\end{align*}
 where $a_{n, \beta}$ is given by (\ref{def:an}). 
 By taking the Euler characteristics of both sides, we 
 obtain $P_{n, \beta}=(-1)^{n+\beta}a_{n, \beta}$, which 
recovers the formula (\ref{intro:formula}). 
	\end{rmk}

\section{Some technical lemmas}
In this section, we give proofs of some postponed technical lemmas. 

\subsection{Functoriality of Kn\"{o}rrer periodicity}
Let $\yY_1$, $\yY_2$ be stacks of the form 
$\yY_i=[Y_i/G_i]$ where $Y_i$ is a smooth affine scheme and 
$G_i$ is a reductive algebraic group which acts on $Y_i$. 
Let $\wW_i \to \yY_i$ be vector bundles. 
Then by Theorem~\ref{thm:period}, 
we have equivalences 
\begin{align}\label{equiv:Phii}
	\Phi_i \colon \MF(\yY_i, w_i) \stackrel{\sim}{\to} \MF(\wW_i \oplus 
	\wW_i^{\vee}, w_i+q_i)
	\end{align}
where $q_i$ is a natural quadratic form on $\wW_i \oplus \wW_i^{\vee}$, 
i.e. $q_i(x, x')=\langle x, x'\rangle$. 
On the other hand, the categories of quasi-coherent factorizations 
$\MF_{\qcoh}(\yY_i, w_i)$ are compactly generated by 
$\MF(\yY_i, w_i)$ 
(see~\cite[Proposition~3.15]{MR3270588}), so it 
is equivalent to the ind-completion of $\MF(\yY_i, w_i)$. 
Therefore by taking ind-completions of both sides in (\ref{equiv:Phii}), the 
above equivalences extend to equivalences
\begin{align*}
	\Phi_i \colon \MF_{\qcoh}(\yY_i, w_i) \stackrel{\sim}{\to} \MF_{\qcoh}(\wW_i \oplus 
	\wW_i^{\vee}, w_i+q_i). 
	\end{align*}

Suppose that we have a commutative diagram 
\begin{align*}
	\xymatrix{
\wW_1 \ar[r]^-{g} \ar[d] & \wW_2 \ar[d] \\
\yY_1 \ar[r]_-{f} & \yY_2
}
	\end{align*}
where $f$ is a morphism of stacks, 
and the top arrow is induced by a morphism of 
vector bundles $g \colon \wW_1 \to f^{\ast}\wW_2$. 
We have the induced diagram 
\begin{align}\label{mor:h12}
	\wW_1 \oplus \wW_1^{\vee} \stackrel{h_1}{\leftarrow}
	\wW_1 \oplus f^{\ast}\wW_2^{\vee} \stackrel{h_2}{\to}
	\wW_2 \oplus \wW_2^{\vee}
	\end{align} 
where $h_1=(\id_{\wW_1}, g^{\vee})$ and $h_2=(g, f)$. 
The following lemma is a variant of~\cite[Lemma~2.4.4]{TocatDT}. 
\begin{lem}\label{lem:commute1}
	The following diagram commutes: 
	\begin{align}\label{dia:MFqcoh}
		\xymatrix{
	\MF_{\qcoh}(\yY_1, w_1) \ar[r]^-{f_{\ast}}	\ar[d]_-{\Phi_1}^-{\sim} & \MF_{\qcoh}(\yY_2, w_2)
	\ar[d]^-{\Phi_2}_-{\sim} \\
	\MF_{\qcoh}(\wW_1 \oplus \wW_1^{\vee}, w_1+q_1) 
	\ar[r]^-{h_{2\ast}h_1^{\ast}} & 
	\MF_{\qcoh}(\wW_2 \oplus 
	\wW_2^{\vee}, w_2+q_2).
	}
		\end{align}
	\end{lem}
\begin{proof}
	We have the commutative diagram 
	\begin{align}\notag
		\xymatrix{
&  f^{\ast}\wW_2^{\vee} \ar[dl]_-{h_5} \ar@/^1.5pc/[rr]^-{h_6}\diasquare \ar@<-0.3ex>@{^{(}->}[r]^-{h_4} \ar[d]_-{h_3} & \wW_1 \oplus f^{\ast}\wW_2^{\vee} \ar[d]_-{h_1}
\ar[r]^-{h_2} & \wW_2 \oplus \wW_2^{\vee} \\
\yY_1 & \ar[l]^-{\rm{pr}_1} \wW_1^{\vee} \ar@<-0.3ex>@{^{(}->}[r]_-{i_1} & \wW_1 \oplus \wW_1^{\vee}. & 	
}		\end{align}
Here $\mathrm{pr}_1$ is the projection and $i_1(x)=(0, x)$. 
By the above diagram together with derived base change, we have 
\begin{align*}
	h_{2\ast}h_1^{\ast}\Phi_1(-)  \cong 
	h_{2\ast}h_1^{\ast}i_{1\ast}\mathrm{pr}_1^{\ast}(-) 
	\cong h_{2\ast}h_{4\ast}h_3^{\ast}\mathrm{pr}_1^{\ast}(-) 
	\cong h_{6\ast} h_{5}^{\ast}(-). 
	\end{align*}
On the other hand, 
we have the commutative diagram 
	\begin{align*}
	\xymatrix{
		 f^{\ast}\wW_2^{\vee}  \ar@/^1.5pc/[rr]^-{h_6}\diasquare \ar[r]^-{h_7} \ar[d]_-{h_5} & \wW_2^{\vee} \ar[d]_-{\mathrm{pr}_2}
		\ar@<-0.3ex>@{^{(}->}[r]^-{i_2} & \wW_2 \oplus \wW_2^{\vee} \\
		 \yY_1 \ar[r]_-{f} & \yY_2. & 	
}		\end{align*}
Here $\mathrm{pr}_2$ is the projection and $i_2(x)=(0, x)$. 
Similarly we have 
\begin{align*}
	\Phi_2 f_{\ast}(-) \cong i_{2\ast}\mathrm{pr}_{2}^{\ast}f_{\ast} 
	\cong i_{2\ast} h_{7\ast}h_5^{\ast} 
	\cong h_{6\ast}h_5^{\ast}(-). 
	\end{align*}
Therefore the diagram (\ref{dia:MFqcoh}) commutes. 
	\end{proof}

We also have the following lemma, 
which is a variant of~\cite[Lemma~2.4.7]{TocatDT}. 	
\begin{lem}\label{lem:commute2}
	Suppose that $g \colon \wW_1 \to f^{\ast}\wW_2$ is a surjective 
	morphism of vector bundles on $\yY_1$. 
	Then we have the commutative diagram 
		\begin{align}\notag
		\xymatrix{
			\MF(\yY_2, w_2) \ar[r]^-{f^{\ast}}	\ar[d]_-{\Phi_2}^-{\sim} & \MF(\yY_1, w_1)
			\ar[d]^-{\Phi_1}_-{\sim} \\
			\MF(\wW_2 \oplus \wW_2^{\vee}, w_2+q_2) 
			\ar[r]^-{h_{1!}h_2^{\ast}} & 
			\MF(\wW_1 \oplus 
			\wW_1^{\vee}, w_1+q_1).
		}
	\end{align}
	\end{lem}
\begin{proof}
	The assumption that $g \colon \wW_1 \to f^{\ast}\wW_2$ is surjective 
	implies that the morphism 
	$h_1$ in (\ref{mor:h12}) is a closed immersion, 
	hence $h_{1!}$ gives a left adjoint of $h_1^{\ast}$. The lemma 
	now follows by taking left adjoints of horizontal arrows in (\ref{dia:MFqcoh})
	and restrict to coherent factorizations. 
	\end{proof}

\subsection{The categories of factorizations on formal fibers}\label{subsec:lem:idem}
Let $G$ be a reductive algebraic group and $Y$ 
be a finite dimensional $G$-representation. 
We denote by $\widehat{Y}$ the formal fiber of 
the quotient morphism $Y \to Y \ssslash G$
at the origin (see Subsection~\ref{subsec:notation}
for the definition of formal fiber). 
Then 
\begin{align*}
	[\widehat{Y}/G] \to
	 \widehat{Y}\ssslash G =\Spec \widehat{\oO}_{Y\ssslash G, 0}
\end{align*}
is a good moduli space for $[\widehat{Y}/G]$, and 
is isomorphic to the formal fiber of the morphism
$[Y/G] \to Y\ssslash G$ 
at $0$. 
We take an element $w \in \Gamma(\oO_{[\widehat{Y}/G]})=
\widehat{\oO}_{Y\ssslash G, 0}$ with $w(0)=0$. 
We have the following lemma: 
\begin{lem}\label{lem:idem}
	For $w \neq 0$, the triangulated category $\MF([\widehat{Y}/G], w)$ 
	is idempotent complete. 
\end{lem}
\begin{proof}
	Let $\widehat{Z} \subset \widehat{Y}$ be the 
	closed subscheme defined by the zero locus of $w$. 
	We have the following version of Orlov equivalence~\cite{Orsin}
	relating the categories of factorizations and 
	those of singularities (see~\cite[Theorem~3.14]{MR3112502})
	\begin{align*}
		\MF([\widehat{Y}/G], w)
		\stackrel{\sim}{\to} D^b([\widehat{Z}/G])/\mathrm{Perf}([\widehat{Z}/G]). 
	\end{align*}
	Let $\mathbf{m}_0 \subset \oO_{\widehat{Z}}$ be the maximal 
	ideal which defines $0 \in \widehat{Z}$, 
	and denote by $\widehat{\oO}_{\widehat{Z}}$ 
	the formal completion of $\oO_{\widehat{Z}}$
	at $\mathbf{m}_0$. Let $Z^{(n)} \cneq \Spec \oO_{\widehat{Z}}/\mathbf{m}_0^{n}$ and 
	$\overline{Z}\cneq \Spec \widehat{\oO}_{\widehat{Z}}$. 
	By the coherent completeness for the stacks
	$[\widehat{Z}/G]$ and $[\overline{Z}/G]$ (see~\cite[Theorem~1.6]{AHR2}), we have the equivalences
	\begin{align*}
		\Coh([\widehat{Z}/G]) \stackrel{\sim}{\to}\mathop{\lim_{\longleftarrow}}_{n}
		\Coh([Z^{(n)}/G]) \stackrel{\sim}{\leftarrow} \Coh([\overline{Z}/G]). 
		\end{align*}
	In particular we have an equivalence 
	\begin{align*}
		D^b([\widehat{Z}/G]) \stackrel{\sim}{\to} 
		D^b([\overline{Z}/G])
	\end{align*}
	which restricts to the
	equivalence for subcategories of perfect objects. 
	Therefore we obtain the equivalence 
	\begin{align*}
		\MF([\widehat{Y}/G], w)
		\stackrel{\sim}{\to} D^b([\overline{Z}/G])/\mathrm{Perf}([\overline{Z}/G]).
	\end{align*}

	Since $\widehat{\oO}_{\widehat{Z}}$ is a complete local ring, 
	the singularity category 
	$D^b(\overline{Z})/\mathrm{Perf}(\overline{Z})$ 
	is well-known to be idempotent complete 
	(for example, see~\cite[Lemma~5.6]{Dyc}, \cite[Lemma~5.5]{KaYa}). 
	The argument can be easily extended to the $G$-equivariant setting. 
	Indeed following the proof of~\cite[Lemma~5.5]{KaYa}, 
	it is enough to show that for a $G$-equivariant  
	maximal Cohen-Macaulay $\widehat{\oO}_{\widehat{Z}}$-module $M$
	and an idempotent $e \in \underline{\End}^G(M)$, 
	it is lifted to a $G$-invariant idempotent 
	in $\End(M)$. 
	Here $\underline{\End}^G(M)$ is the set of morphisms in the 
	$G$-equivariant stable 
	category of maximal Cohen-Macaulay modules over $\widehat{\oO}_{\widehat{Z}}$.
	For an idempotent $e \in \underline{\End}^G(M)$, 
	we lift it to
	$a \in \End(M)$, which we can assume 
	to be $G$-invariant as $G$ is reductive. 
	Then as in the proof of~\cite[Theorem~6.7]{repbook}, 
	the limit 
	$\widetilde{e} \cneq \lim f_j(a)$
	converges, idempotent in $\End(M)$ which lifts 
	$e$.    
	Here $f_j(x)$ is given by 
	\begin{align*}
		f_j(x)=\sum_{i=0}^n 
		\dbinom{2n}{i}
			x^{2n-i}(1-x)^i. 
		\end{align*}
	By the construction $\widetilde{e}$ is $G$-invariant, so we obtain the 
	desired lifting property of the idempotents. 
\end{proof}

Let $W$ be another finite dimensional $G$-representation 
and $q \colon W \to \mathbb{A}^1$ be a $G$-invariant non-degenerate 
quadratic form. 
We take $w \in \widehat{\oO}_{Y\ssslash G, 0}$ with $w(0)=0$. 
We have the following lemma: 
\begin{lem}\label{lem:formW}
There is a natural morphism of stacks
\begin{align}\label{mor:iota}
	\iota \colon \left[(\widehat{Y \oplus W})/G  \right]
	\to\left[(\widehat{Y} \times W)/G \right]
	\end{align}
such that the induced functor 
\begin{align}\label{iota:ast}
	\iota^{\ast} \colon 
\MF\left(\left[(\widehat{Y} \times W)/G \right], w+q\right) 
\to \MF\left(\left[(\widehat{Y \oplus W})/G \right], w+q   \right)
	\end{align}
is fully-faithful with dense image. 
	\end{lem}
\begin{proof}
	Let $\pi_Y$, $\pi_{Y \oplus W}$ be the quotient morphisms 
	\begin{align*}
		\pi_Y \colon Y \to Y\ssslash G, \ 
		\pi_{Y \oplus W} \colon Y \oplus W \to (Y\oplus W)\ssslash G. 
		\end{align*}
	Then we have 
	$\pi_{Y \oplus W}^{-1}(0, 0) \subset \pi_Y^{-1}(0) \times W$, 
	therefore we have the induced natural morphism (\ref{mor:iota}) by the definition 
	of formal fibers. 
	
	Note that we have 
$\Crit(w+q)=\Crit(w) \times \{0\}$, so 
the morphism (\ref{mor:iota}) induces the isomorphism of 
critical loci of $w+q$ on 
$\widehat{Y} \times W$ and $\widehat{Y \oplus W}$, 
and also their formal neighborhoods. 
Therefore the functor (\ref{iota:ast}) is fully-faithful with 
dense image by~\cite[Theorem~2.10]{Orcomp} (in \textit{loc.~cit.~} it is stated without 
$G$-action, but the same argument applies to the $G$-equivariant 
setting verbatim). 
	\end{proof}

Suppose that $Y$ is quasi-projective variety 
with an action of a reductive algebraic group $G$ 
such that the good moduli space 
$\pi \colon [Y/G] \to Y\ssslash G$ exists. 
For each closed point $y \in Y \ssslash G$, 
we denote by $[\widehat{Y}_y/G]$ 
the formal fiber of $\pi$
at $y$. 
For a regular function 
$w \colon [Y/G] \to \mathbb{A}^1$, we denote by 
$\widehat{w}_y$ its restriction to $[\widehat{Y}_y/G]$, 
and $\widehat{\pi}_y \colon [\widehat{Y}_y/G]
\to \widehat{Y}_y \ssslash G$ its good moduli space. 
  We have the following lemma: 
  \begin{lem}\label{lem:locvani}
  	For $\eE \in \MF([Y/G], w)$, suppose that 
  	$\eE|_{[\widehat{Y}_y/G]} \in \MF([\widehat{Y}_y/G], \widehat{w}_y)$
  	is isomorphic to zero for any closed point $y \in Y\ssslash G$. 
  	Then we have $\eE \cong 0$. 
  	\end{lem}
  \begin{proof}
  	The inner homomorphism $\hH om^{\bullet}(\eE, \eE)$ is 
  	an object in $\MF([Y/G], 0)$, which is equivalent to 
  	the $\mathbb{Z}/2$-periodic 
  	derived category of coherent sheaves on $[Y/G]$. 
  By the derived base change, we have 
  	\begin{align*}
  		\pi_{\ast}\hH om^{\bullet}(\eE, \eE)
  		\otimes_{\oO_{Y\ssslash G}} \widehat{\oO}_{Y\ssslash G, y}
  		&\cong \widehat{\pi}_{y\ast}
  		\hH om^{\bullet}(\eE|_{[\widehat{Y}_y/G]}, \eE|_{[\widehat{Y}_y/G]}) \\
  		&\cong 0
  		\end{align*}
  	in the $\mathbb{Z}/2$-periodic derived category of quasi-coherent 
  	sheaves on $\widehat{Y}_y \ssslash G$. 
  	The object $\pi_{\ast}\hH om^{\bullet}(\eE, \eE)$
  	is an object in the $\mathbb{Z}/2$-periodic derived category of quasi-coherent
  	sheaves on $Y\ssslash G$ whose 
  	formal completions at any $y \in Y\ssslash G$ is zero, so 
  	it is isomorphic to zero. 
  	Then we have 
  	$\Hom^{\bullet}(\eE, \eE)=\dR \Gamma(\hH om^{\bullet}(\eE, \eE))=0$, 
  	so $\eE \cong 0$. 
  	\end{proof}
  
  \subsection{Right adjoint functor}
  \begin{lem}\label{lem:radjoint}
  	The functor $\Upsilon_{j_{\bullet}}$ in (\ref{FF:global}) 
  	admits a right adjoint $\Upsilon_{j_{\bullet}}^R$. 
  \end{lem}
  \begin{proof}
  	We consider the following diagram 
  \begin{align}\label{dia:Mplus}
 \footnotesize{ 		\xymatrix{
 		& \mM_Q^{\dag, \theta \sss}(v^{\bullet}) \ar[rd]^-{p} \ar[d]_-{q} & & \\
  \mM_Q^{\theta \sss}(s_m)^{\times l} \times 
\mM_Q^{\dag, \theta_- \sss}(v-ls_m) \ar@<-0.3ex>@{^{(}->}[r] \ar[d] &  \mM_Q^{\theta \sss}(s_m)^{\times l} \times 
\mM_Q^{\dag, \theta \sss}(v-ls_m) \ar[d] & \mM_Q^{\dag, \theta \sss}(v) \ar[d] & \ar@<0.3ex>@{_{(}->}[l] 
\mM_Q^{\dag, \theta_+ \sss}(v) \ar[d] \\
 M_Q^{\theta \sss}(s_m)^{\times l} \times 
M_Q^{\dag, \theta_- \sss}(v-ls_m) \ar[r] & M_Q^{\theta \sss}(s_m)^{\times l} \times M_Q^{\dag, \theta \sss}(v-ls_m) \ar[r]_-{\oplus} & 
M_Q^{\dag, \theta \sss}(v) \ar[d]_-{w}  & M_Q^{\dag, \theta_+\sss}(v) \ar[l] \\
& & \mathbb{A}^1. & 
   }}
  		\end{align}
  	
  	Similarly to (\ref{window:glob}),
  	let 
  	\begin{align*}
  		\widetilde{\mathbb{W}}_{\rm{glob}}^{\theta_{\pm}}(v) \subset D^b(\mM_Q^{\dag, \theta \sss}(v))
  	\end{align*} 	
  	be the window subcategory (\ref{window:m})
  	for the choice $m_{\bullet}^{\pm}$ in (\ref{choice:m}). 
  	We consider the composition functor 
  	\begin{align}\notag
  		D^b(M_Q^{\theta \sss}(s_m))^{\boxtimes l} \boxtimes 
  		D^b(M_Q^{\dag, \theta \sss}(v-ls_m))
  		&\stackrel{\sim}{\to}
  		\boxtimes_{i=1}^l D^b(\mM_Q^{\theta \sss}(s_m))_{j_i+(2i-1)(m^2-m)}
  		\boxtimes \widetilde{\mathbb{W}}_{\rm{glob}}^{\theta_-}(v-ls_m) \\
  		\label{compose:D} &\to D^b(\mM_Q^{\dag, \theta \sss}(v)) 
  		\to 
  		D^b(M_Q^{\dag, \theta_+\sss}(v)). 
  	\end{align}
  	Here the first equivalence is due to window theorem in Theorem~\ref{thm:window} together 
  	with the fact that (\ref{MQ:sm}) is a $\C$-gerbe, 
  	the second arrow is the categorified Hall product (i.e. $p_{\ast}q^{\ast}$ in the 
  	diagram (\ref{dia:Mplus})), 
  	and the last arrow is the restriction to the semistable locus. 
  	The first arrow is of Fourier-Mukai type 
  	by Lemma~\ref{lem:window:supp} below, 
  	and the second and the third arrows are also of Fourier-Mukai 
  	type by their constructions. Therefore the above 
  	composition functor is of Fourier-Mukai type. 
  	So we have the kernel object 
  	\begin{align*}
  		\pP \in D^b((M_Q^{\theta \sss}(s_m)^{\times l} \times M_Q^{\dag, 
  			\theta_- \sss}(v-ls_m)) \times M_Q^{\dag, \theta_+ \sss}(v)). 
  	\end{align*}
  
  Moreover the kernel objects of the second and the third arrows in (\ref{compose:D}) 
  are push-forward from the fiber products over $M_Q^{\dag, \theta \sss}(v)$ by their constructions. 
  By Lemma~\ref{lem:window:supp} below, the kernel object of the first arrow in (\ref{compose:D})
  is a push-forward from the fiber product over $\mathbb{A}^1$
  and supported on the fiber product over $M_Q^{\dag, \theta \sss}(v)$. 
  Therefore 
   the object $\pP$ is a push-forward of 
  an object 
  \begin{align}\label{P:omega}
  	\pP_{w} \in D^b((M_Q^{\theta \sss}(s_m)^{\times l} \times M_Q^{\dag, \theta_- \sss}(v-ls_m)) \times_{\mathbb{A}^1}
  	M_Q^{\dag, \theta_+ \sss}(v)) 
  \end{align}
supported on the fiber product over $M_Q^{\dag, \theta \sss}(v)$. 
    	Since $M_Q^{\dag, \theta_+ \sss}(v)$ and 
  	$M_Q^{\theta \sss}(s_m)^{\times l} \times 
  	M_Q^{\dag, \theta_- \sss}(v-ls_m)$
  	are proper over $M_Q^{\dag, \theta \sss}(v)$, 
  	the functor (\ref{compose:D}) admits a right 
  	adjoint given by the Fourier-Mukai kernel $\pP^R$ defined by 
  	\begin{align*}
  		\pP^R \cneq \pP^{\vee} \boxtimes \omega_{M_Q^{\theta \sss}(s_m)^{\times l} \times M_Q^{\dag, \theta_- \sss}(v-ls_m)}[\dim M_Q^{\theta \sss}(s_m)^{\times l} \times M_Q^{\dag, \theta_- \sss}(v-ls_m)].
  		\end{align*}  
  		
  	By
  	Theorem~\ref{thm:window:MF}, the functor $\Upsilon_{j_{\bullet}}$ in (\ref{FF:global}) is regarded 
  	as a functor
  	\begin{align}\label{funct:MFF}
  		\Upsilon_{j_{\bullet}} \colon 
  		\MF(M_Q^{\theta \sss}(s_m), w)^{\boxtimes l}
  		\boxtimes \MF(M_Q^{\dag, \theta_- \sss}(v-l s_m), w)
  		\to \MF(M_Q^{\dag, \theta_+ \sss}(v), w). 
  	\end{align}
 The above functor
  	is a Fourier-Mukai functor with kernel given by $\Xi(\pP_{w})$, 
  	where $\Xi$ is the natural 
  	functor (see~\cite[Theorem~5.5]{MR3581302})
  	\begin{align*}
  	\Xi	\colon &D^b((M_Q^{\theta \sss}(s_m)^{\times l})
  		\times M_Q^{\dag, \theta_- \sss}(v-ls_m)) 
  		\times_{\mathbb{A}^1}
  		M_Q^{\dag, \theta_+ \sss}(v)) \\
  		&\to \MF((M_Q^{\theta \sss}(s_m)^{\times l} \times M_Q^{\dag, \theta_- \sss}(v-ls_m)) \times
  		M_Q^{\dag, \theta_+ \sss}(v), w \boxplus (-w))
  	\end{align*} 
  By the Grothendieck Riemann-Roch theorem, the object 
  $\pP^R$ is the push-forward of an object 
  $\pP_w^R$ in the RHS of (\ref{P:omega}). 
  	Then the right adjoint of (\ref{funct:MFF}) is obtained by the Fourier-Mukai 
  	kernel $\Xi(\pP_{w}^R)$. 
  \end{proof}

\begin{lem}\label{lem:window:supp}
	In the setting of Theorem~\ref{thm:window}, 
	let $\yY=[Y/G]$, $\yY^{\rm{ss}}=[Y^{l\sss}/G]$, 
	and assume that $\yY^{\rm{ss}}$ is a projective scheme over $Y\ssslash G$. 
	Then the splitting of 
	$D^b(\yY) \twoheadrightarrow D^b(\yY^{\rm{ss}})$
	in Theorem~\ref{thm:window} (applied for $N'=0$) is of Fourier-Mukai type 
	whose kernel object $\pP \in D^b(\yY \times \yY^{\rm{ss}})$
	is supported on $\yY \times_{Y\ssslash G} \yY^{\rm{ss}}$. 
	Moreover for any non-constant $w \colon Y\ssslash G \to \mathbb{A}^1$, 
	we have 
	$\pP=i_{\ast}\pP_{w}$
	for some $\pP_w \in D^b(\yY \times_{\mathbb{A}^1} \yY^{\rm{ss}})$.
	Here $\yY \times_{\mathbb{A}^1} \yY^{\rm{ss}}$
	is given by the diagram
	\begin{align*}
		\xymatrix{
	\yY \times_{\mathbb{A}^1}\yY^{\rm{ss}} \ar[r]^-{i} \ar[d] & 
	\yY \times \yY^{\rm{ss}} \ar[d]^-{w \boxplus(-w)} \\
	0 \ar[r] & \mathbb{A}^1. 	
	}
		\end{align*}
	\end{lem}
\begin{proof}
	The KN stratification of $\yY$ pulls back to the one on 
	$\yY \times \yY^{\rm{ss}}$ via the first projection, thus by a choice of 
	$m_{\bullet}$ in Theorem~\ref{thm:window} we have the 
	splitting $\Psi$ of 
	$D^b(\yY \times \yY^{\rm{ss}}) \twoheadrightarrow 
	D^b(\yY^{\rm{ss}} \times \yY^{\rm{ss}})$. 
	From its construction, $\Psi$ is linear over 
	$\mathrm{Perf}(Y\ssslash G \times Y\ssslash G)$. 
	Therefore for any non-constant $w$, 
by~\cite[Proposition~5.5]{MR3327537} there is a splitting $\Phi_{w}$ of 
$D^b(\yY \times_{\mathbb{A}^1} \yY^{\rm{ss}}) \twoheadrightarrow 
D^b(\yY^{\rm{ss}} \times_{\mathbb{A}^1}\yY^{\rm{ss}})$
such that the following diagram commutes: 
\begin{align*}
	\xymatrix{
D^b(\yY^{\rm{ss}} \times_{\mathbb{A}^1} \yY^{\rm{ss}}) \ar[r]^-{\Phi_{w}}
 \ar[d]_-{i_{\ast}}
& D^b(\yY \times_{\mathbb{A}^1} \yY^{\rm{ss}}) \ar[d]_-{i_{\ast}} \\
D^b(\yY^{\rm{ss}} \times \yY^{\rm{ss}}) \ar[r]^-{\Phi} & 
D^b(\yY \times \yY^{\rm{ss}}). 
}
	\end{align*}
Since $\yY^{\rm{ss}}$ is a quasi-projective scheme, we have 
$\oO_{\Delta} \in D^b(\yY^{\rm{ss}} \times \yY^{\rm{ss}})$. 
We set $\pP=\Phi(\oO_{\Delta})$ and $\pP_{w}=\Phi_{w}(\oO_{\Delta})$. 
Then $\pP=i_{\ast}\pP_w$. 
Since this holds for any $w$, the object $\pP$ is supported on 
$\yY \times_{Y\ssslash G} \yY^{\rm{ss}}$. 
Then the object $\pP$ induces the Fourier-Mukai functor 
$D^b(\yY^{\rm{ss}}) \to D^b(\yY)$
which gives the splitting in Theorem~\ref{thm:window}
by the argument in~\cite[Section~2.3]{MR3327537}. 
	\end{proof}
  
  \subsection{Proof of Proposition~\ref{prop:func}}\label{subsec:prop:func}
 \begin{proof}
 	The assertion is trivial if $\dim V \le 1$. 
 	Below we assume that $\dim V \ge 2$. 
 	Note that $\mathrm{ord}_0(w_p) \ge 2$ where $\mathrm{ord}_0(w_p)$ is the vanishing 
 	order of $w_p$ at $0$. 
 	This is because $w_p(0)=0$ by the first inclusion in (\ref{w-10}) together with the fact that 
 	$0 \in \mathrm{Crit}(w_p) \neq \emptyset$. 
 	
 	Let us consider the Hessian of $w_p$
 	\begin{align*}
 		\mathrm{Hess}(w_p) \colon \Ext_{Q^{\dag}}^1(R, R) \otimes \oO_{\widehat{\mM}_{Q_p}^{\dag}(d)}
 		\to \Ext_{Q^{\dag}}^1(R, R)^{\vee} \otimes \oO_{\widehat{\mM}_{Q_p}^{\dag}(d)}.
 	\end{align*}
 	The kernel of the above morphism at the origin 
 	is $\Ext_{(Q^{\dag}, W)}^1(R, R)$. 
 	By the relation (\ref{Phi:T}), we have 
 	\begin{align*}
 		\Ext_{(Q, W)}^1(S_m, S_m) &=
 		\Ext_X^1(\oO_C(m-1), \oO_C(m-1)) \\
 		&=0.
 	\end{align*}
 	It follows that 
 	\begin{align}\label{Hess:w}
 		\Ker(\mathrm{Hess}(w_p)|_{0}) \cap 
 		(\End(V) \otimes \Ext_Q^1(S_m, S_m))=0. 
 	\end{align}
 	By Lemma~\ref{lem:morse} below, 
 	by replacing the isomorphism  $\eta_p$ in (\ref{ffiber})
 	if necessary, 
 	there exist linear subspaces 
 	\begin{align*}
 		W_1 \subset \Ext_{Q^{\dag}}^1(R_{\infty}, R_{\infty}), 
 		W_2 \subset \Ext_{Q^{\dag}}^1(R_{\infty}, S_m), \ 
 		W_3 \subset \Ext_{Q^{\dag}}^1(S_m, R_{\infty})
 	\end{align*}
 	such that $w_p$ is written as 
 	$w_p=w_1+w_2$, where 
 	$w_1$
 	does not contain variables from 
 	$\End(V) \otimes \Ext_Q^1(S_m, S_m)$ with $\deg(w_1) \ge 3$, 
 	and 
 	$w_2$ is a non-degenerate 
 	$G$-invariant quadratic form on 
 	\begin{align*}
 		&W_1 \oplus (W_2 \otimes V) \oplus (W_3 \otimes V^{\vee}) \oplus (\End(V) \otimes \Ext_Q^1(S_m, S_m)) \\
 		&=(W_1 \oplus \Ext_Q(S_m, S_m)) \oplus (W_2 \otimes V) \oplus (W_3 \otimes V^{\vee}) \oplus 
 		(\End_0(V) \otimes \Ext_Q^1(S_m, S_m)). 
 	\end{align*}
 	As we assumed that $\dim V \ge 2$, 
 	the $\GL(V)$-representation $\End_0(V)$ is a non-trivial irreducible $\GL(V)$-representation, 
 	and it is not isomorphic to $V$ nor $V^{\vee}$. 
 	Therefore $w_2$ is written as $w_2=w_3+q$ where 
 	$w_3$ does not contain variables from $\End_0(V) \otimes \Ext_Q^1(S_m, S_m)$
 	and $q$ is a non-degenerate $\GL(V)$-invariant quadratic form on 
 	$\End_0(V) \otimes \Ext_Q^1(S_m, S_m)$. Moreover $w_3$ is non-zero, 
 	since otherwise it contradicts with 
 	(\ref{Hess:w}) and $\End_0(V) \subsetneq \End(V)$. 
 	By replacing the isomorphism (\ref{isom:W}) if necessary, 
 	we can also assume that $q$ coincides with (\ref{q:W}). 
 	Therefore we obtain a desired form (\ref{func:wp}). 
 \end{proof}
 
 We have used the following lemma, whose proof is a variant of~\cite[Proposition~2.24]{MR3399099}: 
 \begin{lem}\label{lem:morse}
 	Let $G$ be a reductive algebraic group 
 	and $V$ be a finite dimensional $G$-representation. 
 	Let $w \colon \widehat{V} \to \mathbb{A}^1$ be 
 	a $G$-invariant formal function such that $\mathrm{ord}_0(w) \ge 2$. 
 	Let $V_1$ be the kernel of the Hessian at the origin 
 	\begin{align*}
 		V_1=\Ker(\mathrm{Hess}(w)|_{0} \colon V \to V^{\vee}). 
 		\end{align*}
 	Then there exists a direct sum decomposition 
 	$V=V_1 \oplus V_2$ of $G$-representations 
 	and a $G$-equivariant isomorphism 
 	$\phi \colon \widehat{V} \stackrel{\cong}{\to} \widehat{V}$
 	such that $\phi^{\ast}w=w_1+w_2$, where 
 	$w_1 \in \oO_{\widehat{V}_1}$ is $G$-invariant with $\mathrm{ord}_0(w_1) \ge 3$, and
 	$w_2 \in \mathrm{Sym}^2(V_2^{\vee})$ is a $G$-invariant 
 	non-degenerate quadratic form on $V_2$. 
 \end{lem}
 \begin{proof}
 	As $w$ is $G$-invariant, the 
 	Hessian of $w$ at the origin 
 	$\mathrm{Hess}(w)|_{0} \colon V \to V^{\vee}$
 	is $G$-equivariant. 
 	As $G$ is reductive, there is a splitting 
 	$V=V_1 \oplus V_2$ as $G$-representations
 	and the Hessian at the origin 
 	is written as 
 	\begin{align*}
 		\mathrm{Hess}(w)|_{0}=
 		\left( \begin{array}{cc}
 			0 & 0 \\
 			0 & q \end{array}
 		\right) \colon V_1 \oplus V_2 \to 
 		V_1^{\vee} \oplus V_2^{\vee}
 	\end{align*}
 where $q$ is a $G$-equivariant isomorphism 
 $q \colon V_2 \stackrel{\cong}{\to}V_2^{\vee}$
 with $q^{\vee}=q$. 
 We identify $q$ as an element $q \in \mathrm{Sym}^2(V_2^{\vee})^{G}$, 
 which is a $G$-invariant non-degenerate quadratic form $q$ on $V_2$. 
 	For $(y_1, y_2) \in V_1 \oplus V_2$, 
 	we can write $w(y_1, y_2)$ as 
 	\begin{align}\label{w:deg3}
 		w(y_1, y_2)=w^{\ge 3}(y_1, y_2)+q(y_2)
 	\end{align}
 	where $w^{\ge 3}(y_1, y_2)$ consists of terms with 
 	degrees bigger than or equal to three. 
 	We set $d_i=\dim V_i$ and fix
 	basis 
 	of $V_1$, $V_2$ so that we write elements of them 
 	as $y_1=\{y_1^{(i)}\}_{1\le i \le d_1}$, 
 	$y_2=\{y_2^{(i)}\}_{1\le i\le d_2}$ respectively. 
 	Here we take an orthonormal basis for $V_2$
 	so $q$ is written as 
 	\begin{align*}
 		q(y_2)=\frac{1}{2}\sum_{i=1}^{d_1} (y_2^{(i)})^2.
 	\end{align*} 
 	Then the closed subscheme
 	\begin{align*}
 		\left\{ \frac{\partial w}{\partial y_2^{(i)}}=0 : 1 \le i \le d_2 \right\}
 		=\left\{y_2^{(i)}+ \frac{\partial w^{\ge 3}}{\partial y_2^{(i)}}=0 : 1 \le i \le d_2 \right\}
 		\subset \widehat{V}
 	\end{align*}
 	is smooth of codimension $d_2$. 
 	By the variable change 
 	\begin{align}\label{vchange}
 		y_2^{(i)} \mapsto 
 		 		\frac{\partial w}{\partial y_2^{(i)}}=
 		y_2^{(i)}+\frac{\partial w^{\ge 3}}{\partial y_2^{(i)}}
 		\end{align}
 	we may assume that 
 	$\Crit(w)$ is contained in $\{y_2=0\} \subset \widehat{V}$. 
 	The variable change (\ref{vchange}) can be described without 
 	coordinates as follows. 
 	Let $dw$ be the morphism given by the derivation of $w$
 	\begin{align}\label{mor:dw}
 		dw \colon V \otimes \oO_{\widehat{V}} \to \oO_{\widehat{V}}. 
 		\end{align}
 	We have the following morphisms 
 	\begin{align*}
 		\phi \colon	V^{\vee}=V_1^{\vee} \oplus V_2^{\vee}
 		\stackrel{(\id, q^{-1})}{\longrightarrow} V_1^{\vee} \oplus V_2
 		\stackrel{(\id, dw|_{V_2})}{\longrightarrow}
 		\widehat{\oO}_V.  
 	\end{align*}
 	The above composition induces the isomorphism 
 	$\oO_{\widehat{V}} \stackrel{\cong}{\to} \oO_{\widehat{V}}$, 
 	which is identified with
 	the variable change (\ref{vchange}). 
 	The above construction is $G$-equivariant, 
 	so the variable change (\ref{vchange}) is $G$-equivariant. 
 	
 	The condition that 
 	$\Crit(w) \subset \{y_2=0\}$ implies that 
 	each $y_2^{(i)}$ is written as 
 	\begin{align*}
 		y_2^{(i)}=\sum_{j=1}^{d_1}a_{ij} \frac{\partial w}{\partial y_1^{(j)}}
 		+\sum_{j=1}^{d_2}b_{ij} \frac{\partial w}{\partial y_2^{(j)}}
 	\end{align*}
 	for some $a_{ij}, b_{ij} \in \oO_{\widehat{V}}$. 
 	By writing $b_{ij}=b_{ij}(0)+b_{ij}^{\ge 1}$ and comparing the 
 	degree one terms for $y_2$, we see that 
 	$b_{ij}(0)=\delta_{ij}$.
 	Therefore we obtain the relation 
 	\begin{align*}
 		-\frac{\partial w^{\ge 3}}{\partial y_2^{(i)}}
 		=\sum_{j=1}^{d_1}a_{ij} \frac{\partial w^{\ge 3}}{\partial y_1^{(i)}}+
 		\sum_{j=1}^{d_2}b_{ij}^{\ge 1}\left(y_2^{(j)}+\frac{\partial w^{\ge 3}}{\partial y_2^{(j)}}   \right). 
 	\end{align*}
 	The Nakayama lemma implies the inclusion of ideals 
 	\begin{align}\label{Nak:imp}
 		\left( \frac{\partial w^{\ge 3}}{\partial y_2^{(i)}} : 1\le i\le d_2 \right)
 		\subset 	\left( \frac{\partial w^{\ge 3}}{\partial y_1^{(j)}}, 
 		y_2^{(i)} : 1\le j\le d_1, 1\le i\le d_2 \right) 
 	\end{align}
 in $\widehat{\oO}_{\widehat{V}}$, the formal completion 
 at the maximal ideal of $\oO_{\widehat{V}}$.
  Since these are $G$-invariant 
 ideals, by the coherent completeness of $[\widehat{V}/G]$ 
 the inclusion (\ref{Nak:imp}) also holds in $\oO_{\widehat{V}}$
 (see the proof of Lemma~\ref{lem:idem}). 
 	In particular there is a relation of the form 
 	\begin{align}\label{rel:w}
 		\left.\frac{\partial w}{\partial y_2^{(i)}}\right\vert_{y_2=0}=
 		\sum_{i, j}c_{ij}\left.\frac{\partial w}{\partial y_1^{(j)}}\right\vert_{y_2=0}
 	\end{align}
 	for some $c_{ij} \in \oO_{\widehat{V}_1}$. 
 	We apply the variable change 
 	\begin{align}\label{varchang}
 		\widetilde{y}_1^{(i)} = y_1^{(i)}+\sum_{j}c_{ij}y_2^{(i)}, \ 
 		\widetilde{y}_2^{(i)}=y_2^{(i)}. 
 	\end{align}
 	Then we have 
 	\begin{align*}
 		\left.\frac{\partial w}{\partial \widetilde{y}_2^{(i)}}\right\vert_{\widehat{y}_2=0} 
 		&=\left.\left(\sum_{j}\frac{\partial y_1^{(j)}}{\partial \widetilde{y}_2^{(i)}} 
 		\frac{\partial w}{\partial y_1^{(j)}}
 		+\sum_{j}\frac{\partial y_2^{(j)}}{\partial \widetilde{y}_2^{(i)}}
 		\frac{\partial w}{\partial y_2^{(j)}}
 		\right)\right\vert_{\widetilde{y}_2=0} \\
 		&=-\sum_{j} c_{ij} \left.\frac{\partial w}{\partial y_1^{(j)}}\right\vert_{\widetilde{y}_2=0}
 		+\left.\frac{\partial w}{\partial y_2^{(i)}}\right\vert_{\widetilde{y}_2=0}=0. 
 	\end{align*}
 	It follows that we can assume that 
 	$(\partial w/\partial y_2^{(i)})|_{y_2=0}=0$. 
 	
 	We see that the variable change (\ref{varchang}) can be taken
 	to be $G$-equivariant. 
 	For the morphism (\ref{mor:dw}), we can write $dw \otimes \oO_{\widehat{V}_1}$ as 
 	\begin{align*}
 		dw \otimes \oO_{\widehat{V}_1}
 		=\alpha^{(1)}\oplus \alpha^{(2)} \colon 
 		(V_1 \otimes \oO_{\widehat{V}_1}) \oplus 
 		(V_2 \otimes \oO_{\widehat{V}_1})
 		\to \oO_{\widehat{V}_1}. 
 	\end{align*}
 	Then the ideals of $\oO_{\widehat{V}_1}$
 	\begin{align*}
 		I_1=\left(\left.\frac{\partial w}{\partial y_1^{(i)}}\right\vert_{y_2=0}   \right), \ 
 		I_2=\left(\left.\frac{\partial w}{\partial y_2^{(i)}}\right\vert_{y_2=0}   \right)
 	\end{align*}
 	are generated by the images of $\alpha^{(1)}$, $\alpha^{(2)}$ 
 	respectively, so in particular 
 	they are $G$-invariant.
 	By the relation (\ref{rel:w}) we have $I_2 \subset I_1$. 
 	We have the following $G$-equivariant diagram 
 	\begin{align*}
 		\xymatrix{
 			V_2 \otimes \oO_{\widehat{V}_1} \ar[r]^-{\alpha^{(2)}} \ar@{.>}[d]_-{\phi}& I_2  \ar@<-0.3ex>@{^{(}->}[d]  \\
 			V_1 \otimes \oO_{\widehat{V}_1} \ar[r]^-{\alpha^{(1)}} & I_1 
 		}
 	\end{align*}
 	where each horizontal arrows are surjections. 
 	As $G$ is reductive, from the above diagram 
 	there is a $G$-equivariant dotted arrow $\phi$ which makes
 	the above diagram commutative. 
 	A choice of $\phi$ corresponds to a choice of $c_{ij}$ in (\ref{rel:w}).
 	Then we have the $G$-equivariant morphism 
 	\begin{align*}
 		V^{\vee}=V_1^{\vee} \oplus V_2^{\vee}
 		\stackrel{(\id+\phi^{\vee}, \id)}{\longrightarrow} \oO_{\widehat{V}}.
 	\end{align*} 
 	The above morphism 
 	induces the $G$-equivariant isomorphism $\widehat{\oO}_V \stackrel{\cong}{\to} \widehat{\oO}_V$, which corresponds to the variable change (\ref{varchang}). 
 	In particular we can choose $c_{ij}$ so that (\ref{varchang}) is 
 	$G$-equivariant. 
 	
 	Finally we set 
 	\begin{align*}
 		g(y_1, y_2) \cneq w(y_1, y_2)-w(y_1, 0). 
 	\end{align*}
 	Then from the above arguments 
 	we have 
 	$g(y_1, 0)=0$ and $(\partial g/\partial y_2^{(i)})|_{y_2=0}=0$. 
 	It follows that $g(y_1, y_2)$ is written as 
 	\begin{align*}
 		g(y_1, y_2)=\sum_{i, j} y_2^{(i)} y_2^{(j)} Q_{ij}(y_1, y_2)
 	\end{align*}
 	for some $Q_{ij} \in \oO_{\widehat{V}}$. 
 	As the quadratic term of $g(y_1, y_2)$ coincides with $q$ by (\ref{w:deg3}), 
 	we have 
 	$Q_{ij}(0)=1/2 \cdot \delta_{ij}$. 
 	It follows that 
 	the critical locus of $g(y_1, y_2)$ is $\{y_2=0 \} \subset \widehat{V}$, so 
 	the $G$-equivariant Morse lemma (see~\cite[Section~17.3]{Arnod})
 	applied for $g$ implies that
 	by a $G$-equivariant variable change of the form 
 	$\widetilde{y}_1^{(i)}=y_1^{(i)}$, 
 	$\widetilde{y}_2^{(i)}=\sum_{i, j} \alpha^{(ij)}(y_1, y_2) y_2^{(j)}$
 	we can make $g(\widetilde{y}_1, \widetilde{y}_2)=q(\widetilde{y}_2)$. 
 	As $\mathrm{ord}_0(w(y_1, 0)) \ge 3$ from (\ref{w:deg3}), the lemma is proved.  	
 \end{proof}

		\bibliographystyle{amsalpha}
	\bibliography{math}

	\vspace{5mm}
	
	Kavli Institute for the Physics and 
	Mathematics of the Universe (WPI), University of Tokyo,
	5-1-5 Kashiwanoha, Kashiwa, 277-8583, Japan.

	\textit{E-mail address}: yukinobu.toda@ipmu.jp
\end{document}